\documentclass[reqno]{amsart}  
\usepackage{amsmath,amssymb,amsthm,mathrsfs}
\theoremstyle{plain}
\newtheorem{theorem}{Theorem}[section]
\newtheorem{lemma}[theorem]{Lemma}
\newtheorem{proposition}[theorem]{Proposition}

\newtheorem{corollary}[theorem]{Corollary}
\newtheorem{problem}[theorem]{Problem}
\newtheorem{remark}[theorem]{Remark}
\numberwithin{equation}{section}
\allowdisplaybreaks[1]

\theoremstyle{definition}

\theoremstyle{remark}

\newcommand{\sbm}[1]{\left[\begin{smallmatrix} #1
    \end{smallmatrix}\right]}
\newcommand{\im}{\operatorname{im}}
\newcommand{\bbm}[1]{\begin{bmatrix}#1\end{bmatrix}}

\newcommand{\cU}{{\mathcal U}}
\newcommand{\cL}{{\mathcal L}}
\newcommand{\cD}{{\mathcal D}}
\newcommand{\cE}{{\mathcal E}}
\newcommand{\cF}{{\mathcal F}}
\newcommand{\cG}{{\mathcal G}}
\newcommand{\cK}{{\mathcal K}}
\newcommand{\cH}{{\mathcal H}}
\newcommand{\cR}{{\mathcal R}}

\newcommand{\cN}{{\mathcal N}}
\newcommand{\cS}{{\mathcal S}}

\newcommand{\cI}{{\mathcal I}}
\newcommand{\cV}{{\mathcal V}}

\newcommand{\cX}{{\mathcal X}}
\newcommand{\cY}{{\mathcal Y}}
\newcommand{\bcG}{{\boldsymbol{\mathcal G}}}
\newcommand{\bcK}{{\boldsymbol{\mathcal K}}}
\newcommand{\bcH}{{\boldsymbol{\mathcal H}}}
\newcommand{\bi}{{\mathbf i}}
\newcommand{\bcD}{{\boldsymbol{\mathcal D}}}
\newcommand{\bDelta}{{\boldsymbol{\Delta}}}
\newcommand{\bfrakS}{{\boldsymbol{\mathfrak S}}}
\newcommand{\ts}{\widetilde s}

\begin{document}

\title[Inverse Lifting Problem]{The inverse commutant lifting
problem: characterization
of associated Redheffer linear-fractional maps}

\author{Joseph A. Ball}
\address{Department of Mathematics \\ Virginia Tech \\ Blacksburg,
Virginia 24061}
\email{ball@math.vt.edu}
\author{Alexander Kheifets $^*$}
\address{Department of Mathematics\\ University of Massachusetts
Lowell \\
Lowell, MA 01854} \email{Alexander\_Kheifets@uml.edu}

\thanks{$^*$ The work of the second author was partially supported by
the University of Massachusetts Lowell Research and Scholarship Grant,
project number: H50090000000010}



\begin{abstract}
    It is known that the set of all solutions of a commutant
    lifting and other interpolation problems   admits
    a Redheffer linear-fractional
    parametrization. The  method of unitary coupling identifies
     solutions of the lifting
    problem with minimal unitary extensions of a partially defined
    isometry constructed explicitly from the problem data. A special
    role is played by a particular unitary extension, called the
    {\em central or universal unitary extension}. The
    coefficient matrix for the Redheffer linear-fractional map has a
    simple expression in terms of the universal unitary extension.
    The universal unitary extension can be seen as a unitary
    coupling of four unitary operators (two bilateral shift operators
    together with two unitary operators coming from the problem data)
    which has special geometric structure. We use this special
   geometric structure to obtain an inverse theorem
   (Theorem \ref{T:converse1} as well as Theorem \ref{T:converse2})
   which characterizes the
    coefficient matrices for a Redheffer linear-fractional map
    arising in this way from a lifting problem.  When expressed in
    terms of Hellinger-space functional models (Theorem
\ref{T:converse3}),
    these results lead to
    generalizations of classical results of Arov
    and to characterizations of the coefficient matrix-measures of
    the lifting problem in terms of the density properties of the
    corresponding model spaces. The main tool is the formalism of
    unitary scattering systems developed in \cite{Kheifets-Dubov},
    \cite{HarmAIP}.
\end{abstract}

\maketitle

\section{Introduction}  \label{S:Intro}

One of the seminal results in the development of operator theory and
its applications over the past half century is the Commutant Lifting
Theorem:  {\em given contraction operators $T'$, $T''$ on respective
Hilbert spaces $\cH'$, $\cH''$ with respective isometric dilations
$\cV'$ and $\cV''$ on respective Hilbert spaces $\cK' \supset \cH'$
and $\cK'' \supset \cH''$ and given a contractive operator $X \colon
\cH' \to \cH''$ such that $X T' = T'' X$, then there exists a
operator $Y \colon \cK' \to \cK''$ also with $\| Y\| \le 1$ such that
$Y \cV' = \cV'' Y$ and $XP_{\cH'} = P_{\cH''}Y$} (where $P_{\cH'}$
and $P_{\cH''}$ are the orthogonal projections of $\cK'$ to $\cH'$
and $\cK''$ to
$\cH''$ respectively).  It is well known
that the general case can be reduced to the case where $T' =
\cU_{+}'$ is an isometry on a Hilbert space $\cK'_{+}$ with unitary
extension $\cU'$ on $\cK' \supset \cK'_{+}$ and where $T''$ is a
coisometry on $\cK''_{-}$ with unitary lift $\cU''$ on $\cK'' \supset
\cK''_{-}$.
Then this normalized commutant lifting problem can be formalized as
follows:

     \begin{problem}[\bf{Lifting Problem}] \label{P:lift}
     Given two  unitary operators $\cU'$ and $\cU''$ on Hilbert
spaces $\cK'$ and
     $\cK''$, respectively, along with subspaces $\cK'_{+} \subset
\cK'$ and
     $\cK''_{-}\subset\cK''$ that are assumed to be $*$-cyclic
     for $\mathcal U'$ and $\mathcal U''$ respectively (i.e., the
smallest
     reducing subspace for $\cU'$ containing $\cK'_{+}$ is the whole
     space $\cK'$ and similarly for $\cU''$ and $\cK''_{-}$) and
     such that
      \begin{equation} \label{invariant}
     {\mathcal U}^{\prime} {\mathcal K}'_{+} \subset {\mathcal
      K}'_{+}, \qquad {\mathcal U}''^* {\mathcal K}''_{-} \subset
{\mathcal
      K}''_{-},
      \end{equation}
      and given a contractive operator $X \colon {\mathcal K}'_{+}
\to  {\mathcal
      K}''_{-}$ which satisfies the intertwining condition
      \begin{equation}  \label{Xintertwine}
     X {\mathcal U}'|_{{\mathcal K}'_{+}} = P_{{\mathcal K}''_{-}}
     {\mathcal U}'' X,
       \end{equation}
       characterize all contractive intertwiners $Y$ of
       $({\mathcal U}', \cK')$
       and $({\mathcal U}'', \cK'')$ which lift $X$ in the (Halmos)
sense that
       \begin{equation}  \label{Y=liftX}
      P_{{\mathcal K}''_{-}}Y |_{{\mathcal K}'_{+}} = X.
       \end{equation}
       \end{problem}

 An important special case of this theorem was first proved by Sarason
 \cite{Sarason}; there he also explains the connections with
 classical Nevanlinna-Pick and Carath\'eodory-Fej\'er interpolation.
 Since the result was first formulated and proved in its full
 generality by Sz.-Nagy-Foias \cite{NF-CLT} (see also
 \cite{NF}), applications have been made to a variety of other
 contexts, including Nevanlinna-Pick interpolation for operator-valued
 functions and best approximation by analytic functions to a given
 $L^{\infty}$-function in $L^{\infty}$-norm (the Nehari problem)---we
 refer to the books \cite{FFbook} and \cite{FFGK} for an overview of
 all these developments. Moreover, the theorem has been generalized
 to still other contexts, e.g., to representations of nest
 algebras/time-varying systems \cite{Davidson, FFGK} as well as
 representations of more exotic Hardy algebras \cite{Frazho84,
Popescu-CLT1,
 BTV, BLTT, Sultanic, MS, McCS}
 with applications to more exotic Nevanlinna-Pick interpolation
 theorems \cite{Popescu98, Popescu03, MS04, Popescu06}.  There has
 also appeared a weighted version \cite{TV, BFF} as well as a relaxed
 version \cite{LiTimotin, FFK, FtHK1, tH1} of the theorem leading to
still other
 types of applications.  There are now also results on
 linear-fractional parametrizations for the set of all solutions (see
 \cite{AAK71, Kheifets-IWOTA96} for the Nehari problem---see also
 \cite[Chapter 5]{Peller} for an overview,
 see \cite[Chapter XIV]{FFbook} and \cite[Theorem VI.5.1]{FFGK}
 and the references there for the
 standard formulation Problem \ref{P:lift} of the Lifting Problem,
 see \cite{FtHK2, tH1, tH2} for the relaxed version of the
 lifting theorem); in the
 context of classical Nevanlinna-Pick interpolation, such
 parametrization results go back to the papers of Nevanlinna
 \cite{Nev19, Nev29}.

 The associated inverse problem asks for a
 characterization of which Redheffer linear-fractional
 coefficient-matrices arise in this way for some Lifting Problem.
 The inverse problem has been studied much less than the direct
problem; there are only a few publications in this direction
 (\cite{Arov90, Kh-regulariz, Kh-exposed}. We refer also to
 \cite{Kats1, Kats2, Sarason2, Sarason3})
 for some special cases of the inverse Lifting Problem (Nehari problem and
 Nevanlinna-Pick/Carath\'eodory-Schur
 interpolation problem), and the quite recent work \cite{tH-pre} on the
 inverse version of the relaxed commutant lifting problem.

 Our contribution here is to further develop the ideas in
 \cite{Kheifets-IWOTA96, HarmAIP} to obtain
 new results on the inverse problem
 (Theorems \ref{T:converse1}, \ref{T:converse2}, \ref{T:converse3})
 in terms of certain invariants associated with a Hellinger-space
model
 for the Lifting Problem.

 The starting point for our approach is the coupling method first
 introduced by Adamjan-Arov-Kre\u{\i}n \cite{AAK68} and developed
 further in \cite{CS, Arocena1, Arocena2, KKY, Kh-PhD, Moran,
 Kh-AIP1, Kh-AIP2, KhYu, Kheifets-Berkeley, Kheifets-IWOTA96, HarmAIP,
 FtHK1, LiTimotin}
 (some of these in several-variable or relaxed contexts---see also
 \cite{Sarason-Halmos} for a nice exposition).  In this approach one
 identifies solutions of the Lifting Problem with minimal unitary
 extensions of an isometry constructed in a natural way from the
 problem data.  We use here the term {\em isometry} (sometimes also
 called {\em semiunitary operator})  in the following technical
 sense:  we are given a Hilbert space $\cH_{0}$ and subspaces $\cD$
 and $\cD_{*}$ of $\cH_{0}$ together with a linear operator $V$ which
 maps $\cD$ isometrically onto $\cD_{*}$; we then say that $V$ is an
 {\em isometry on $\cH_{0}$ with domain $\cD$ and codomain $\cD_{*}$.}
 By a {\em minimal unitary extension} of $V$ we mean a unitary
 operator $\cU$ on a Hilbert space $\cK$ containing $\cH_{0}$ as a
 subspace such that the restriction of $\cU$ to $\cD$ agrees with $V$
 and the smallest $\cU$-reducing subspace containing $\cH_{0}$ is all
 of $\cK$.  From the work of Arov-Grossman \cite{AG1, AG2} and
 Katsnelson-Kheifets-Yuditskii \cite{KKY}, it is
 known that there is a special unitary colligation $U_{0}$ (called
 the universal unitary colligation) so that any such unitary
 extension $\cU^{*}$ of $V$ arises as the lower feedback connection
 $\cU^{*} = \cF_{\ell}(U_{0}, U_{1})$ of $U_{0}$ with a
 free-parameter unitary colligation $U_{1}$ (see Theorem
 \ref{T:ext-struct}).  A special unitary
 extension of $V$ is obtained as the unitary dilation $\cU_{0}^{*}$
 of the universal unitary colligation $U_{0}$ (or, in the language of
 \cite{BCU}, $\cU_{0}$ is the unitary evolution operator for the
 Lax-Phillips scattering system in which $U_{0}$ is embedded).  This
 special unitary extension $\cU_{0}^{*}$ of $V$ is called the {\em
 universal unitary extension}.

 Unlike other contexts where the ``lurking isometry'' approach has
 been used (see in particular \cite{KKY, KhYu,  Kheifets-Berkeley,
 Kupin, BT-AIP}),
 the connection between unitary extensions $\cU^{*}$ of $V$ and
 $(\cU', \cU'')$-intertwiners $Y$ solving the Lifting Problem,
 as in \cite{Kheifets-IWOTA96, HarmAIP},
 involves an extra step: computation of the lift $Y$ from the unitary
 extension $\cU$ is not immediately explicit but rather involves a
 wave-operator construction demanding computation of powers of $\cU$.
 The lift $Y$ is uniquely determined from its moments $w_{Y}(n) =
 i_{*}^{*} Y^{n} i$ where $i_{*}$ and $i$ are certain isometric
 embedding operators (or {\em scale operators} in the sense of
 \cite{Kheifets-Dubov}).  Calculation of such moments (the collection
 of which we call the {\em symbol} of the lift $Y$) requires the
 computation of powers of $\cU^{*} = \cF_{\ell}(U_{0}, U_{1})$ in
 terms of the coefficients of the universal unitary colligation
 $U_{0}$ determined by the problem data and the coefficients of the
 free-parameter unitary colligation $U_{1}$ (or in terms of its
 characteristic function $\omega(\zeta)$).  In Section \ref{S:FB} we
 identify a general principle of independent interest
 for the explicit computation of the powers of an operator $\cU^{*}$
 given as a feedback connection $\cU^{*} = \cF_{\ell}(U_{0}, U_{1})$
 of two unitary colligations $U_{0}$ and $U_{1}$.  With the
 application of this general principle, we arrive at an
 explicit Redheffer-type linear-fractional parametrization of the set
 of symbols $\{w_{Y}(n)\}_{n \in {\mathbb Z}}$ associated with the
 set of solutions $Y$ of a Lifting Problem (see Theorem
 \ref{T:symbol-param}).  The symbol for the Redheffer coefficient
 matrix is a simple explicit formula in terms of the universal unitary
 extension $\cU_{0}$ (see formula \eqref{Redheffer-central} in
 Theorem \ref{T:Redheffer-central} below).

 This general principle (already implicitly present in
\cite{Kheifets-IWOTA96})
 can be summarized as follows.
 Suppose that the operator $\cU$ is given as the lower feedback
 connection $\cU^* = \cF_{\ell}(U_{0}, U_{1})$ of two colligation
 matrices
 $$ U_{0} = \begin{bmatrix} A_{0} & B_{0} \\ C_{0} & D_{0}
\end{bmatrix} \colon \begin{bmatrix} \cX_{0} \\ \cD \end{bmatrix} \to
\begin{bmatrix} \cX_{0} \\ \cD_{*} \end{bmatrix}, \quad
    U_{1} = \begin{bmatrix} A_{1} & B_{1} \\ C_{1} & D_{1}
    \end{bmatrix} \colon \begin{bmatrix} \cX_{1} \\ \cD_{*}
\end{bmatrix} \to
    \begin{bmatrix} \cX_{1} \\ \cD \end{bmatrix}.
$$
(In our context we always have $D_0=0$).
Associated with any colligation matrix $U = \sbm{ A & B \\ C & D }
\colon \sbm{ \cX \\ \cE} \to \sbm{\cX \\ \cE_{*}}$ is the
discrete-time linear system
\begin{equation}   \label{sys'}
\Sigma_{U} \colon \begin{bmatrix} x(n+1) \\ e_{*}(n)
\end{bmatrix} = U \begin{bmatrix} x(n) \\ e(n) \end{bmatrix}, \quad
x(0) = x_{0}
\end{equation}
which recursively defines what we call the {\em augmented
input-output map} (extending the usual input-output map in the sense
that it takes into account an initial condition $x(0)
= x_{0}$ not necessarily equal to zero as well as the internal state
trajectory $\{x(n)\}_{n \in {\mathbb Z}}$):
$$
W(U)^{+}: = \begin{bmatrix} W(U)^{+}_{0} & W(U)^{+}_{2} \\
W(U)^{+}_{1}  & W(U_{0})^{+} \end{bmatrix} \colon  \begin{bmatrix}
x_{0} \\ \{e(n)\}_{n \in {\mathbb Z}_{+}} \end{bmatrix} \mapsto
\begin{bmatrix} \{x(n)\}_{n \in {\mathbb Z}_{+}} \\ \{ e_{*}(n)\}_{n
    \in {\mathbb Z}_{+}} \end{bmatrix}
 $$
if $\{ e(n), x(n), e_{*}(n)\}_{n \in {\mathbb Z}_{+}}$ solves the
system equations \eqref{sys'}.  Then the general principle asserts:
{\em powers $\cU^{n}$ of $\cU = \cF_{\ell}(U_{0}, U_{1})$ can be
computed via performing a feedback connection at the
system-trajectory level:}
$$\left\{ \cU^{n} \begin{bmatrix} x_{0} \\ x_{1} \end{bmatrix}
\right\}_{n \in {\mathbb Z}_{+}} = \cF_{\ell}\left( W(U_{0})^{+},
W(U_{1})^{+}\right) \begin{bmatrix} x_{0} \\ x_{1} \end{bmatrix}.
$$

The general structure for the universal unitary extension $\cU_{0}$
with embedded subspaces related to the original problem data $(\cU',
\cK')$, $(\cU'', \cK'')$ and coefficient spaces $\widetilde \Delta$,
$\widetilde \Delta_{*}$ for the free-parameter characteristic
function can be viewed as a four-fold Adamjan-Arov (AA) unitary
coupling
in the general sense of \cite{Adamjan-Arov}.  In this setting one can
identify the special geometry corresponding to the case where the
four-fold AA-unitary coupling arises from a Lifting Problem. In
this way we arrive at the inverse theorem
(Theorems \ref{T:converse1}, \ref{T:converse2}, \ref{T:converse3}),
specifically, a
characterization of which Redheffer coefficient matrices arise as
the coefficient matrix for the linear-fractional parametrization of
the set of all solutions of some Lifting Problem
({\sl with given operators $\cU', \cU''$ and subspaces $\cK'_{+}
\subset \cK', \cK''_{-}\subset\cK''$ }) generalizing results
of \cite{Arov90, Kh-regulariz, Kh-exposed}
obtained in the context of the Nehari Problem and the bitangential
Nevanlinna-Pick problem.

The solution of the inverse problem for the Lifting Problem as
presented here appears to be quite different from the inverse
problem considered in
\cite{tH-pre}.   We discuss the connections between
the results of this paper and those of \cite{tH-pre}
in detail in Remark \ref{R:param} (for the direct
problem) and in Remarks \ref{1013}, \ref{R:relaxed} (for the inverse problem).

The paper is organized as follows. After the present Introduction, in
Section \ref{S:FB} we present the general principle for computation
of powers of $\cU = \cF_{\ell}(U_{0}, U_{1})$ via the
trajectory-level feedback connection of the augmented input-output
operator of $U_{0}$ with that of $U_{1}$.  Section \ref{S:Hellinger'}
reviews preliminary material from \cite{Kheifets-Dubov}  concerning
Hellinger-space functional models for unitary operators equipped also
with a scaling operator.  Section \ref{S:Intertwiners-Couplings}
reviews basic ideas from \cite{Adamjan-Arov} concerning the
correspondence
between contractive intertwiners $Y$ of two unitary operators $\cU'$
and $\cU''$ on the one hand and unitary couplings $\cU$ of $\cU'$ and
$\cU''$ on the other.  Section \ref{S:solutions-extensions} adds the
constraint that the intertwiner $Y$ should be a lift of a given
contractive intertwiner $X$ of restricted/compressed versions
$\cU'_{+}$, $\cU''_{-}$ of $\cU'$, $\cU''$ and identifies the
correspondence between solutions $Y$ of the lifting problem and
unitary extensions $\cU^{*}$ of the isometry $V$ constructed directly
from the data for the Lifting Problem.  Section \ref{S:unitext}
recalls the result from \cite{AG1, AG2} that such unitary extensions
arise as the lower feedback connection of the universal unitary
colligation $U_{0}$ with a free-parameter unitary colligation
$U_{1}$.  Section \ref{S:param} uses the general principle from
Section \ref{S:FB} to obtain a parametrization for the set of symbols
$\{w_{Y}(n)\}_{n \in {\mathbb Z}}$ associated with solutions $Y$ of
the Lifting Problem.  Section \ref{S:centralext} introduces the
universal unitary extension.  Here the universal unitary extension is
identified as the four-fold AA-unitary coupling of the two unitary
operators $\cU'$, $\cU''$ appearing in the Lifting-Problem data
together
with the bilateral shift operators associated with the input and
output spaces for the
free-parameter unitary colligation.  Here the special geometric
structure is identified which leads to the coordinate-free version
of our inverse theorem (Theorem \ref{T:converse1})
characterizing which four-fold AA-unitary couplings arise in
this way from a Lifting Problem.  Here also is established the
formula for the
Redheffer-coefficient matrix in terms of the universal unitary
extension.  Sections \ref{S:charmeas} and \ref{S:compactHel}
convert these results to more concrete function-theoretic form in the
setting of Hellinger-model spaces. In particular, we get two more
concrete versions of the inverse Theorem \ref{T:converse1} (Theorems
\ref{T:converse2} and \ref{T:converse3}).
In Section \ref{S:Nehari} we apply our results to the classical
Nehari problem.

\section{Calculus of feedback connection of unitary colligations}
\label{S:FB}

Suppose that we are given linear spaces $\cX_{0}, \widetilde
    \cX_{0}, \cX_{1}, \widetilde \cX_{1}, \cF, \cF_{*}$ and linear
    operators presented in block matrix form
 \begin{align}
     U_{0} & = \begin{bmatrix} A_{0} & B_{0} \\ C_{0} & D_{0}
\end{bmatrix} \colon
 \begin{bmatrix} \cX_{0} \\ \cF \end{bmatrix} \to
     \begin{bmatrix} \widetilde \cX_{0} \\ \cF_{*} \end{bmatrix},
\notag \\
     U_{1} & = \begin{bmatrix} A_{1} & B_{1} \\ C_{1} & D_{1}
\end{bmatrix} \colon
     \begin{bmatrix} \cX_{1} \\ \cF_{*} \end{bmatrix} \to
 \begin{bmatrix} \widetilde \cX_{1} \\ \cF \end{bmatrix}.
     \label{U0U1}
\end{align}
  We define the feedback connection
$U:= \cF_{\ell}(U_{0}, U_{1}) \colon \sbm{ \cX_{0} \\ \cX_{1}} \to
\sbm{ \widetilde \cX_{0} \\ \widetilde \cX_{1}}$ (when it exists) by
\begin{align}
 &\cF_{\ell}(U_{0}, U_{1}) \begin{bmatrix} x_{0} \\ x_{1}
\end{bmatrix} =
 \begin{bmatrix} \widetilde x_{0} \\ \widetilde x_{1} \end{bmatrix}
     \text{ if there exist }
    f \in \cF \text{ and } f_{*} \in \cF_{*} \text{ so that } \notag
\\
    & \begin{bmatrix} A_{0} & B_{0} \\ C_{0} & D_{0} \end{bmatrix}
    \begin{bmatrix} x_{0} \\ f \end{bmatrix} =
	\begin{bmatrix} \widetilde x_{0} \\ f_{*} \end{bmatrix} \text{ and }
 \begin{bmatrix} A_{1} & B_{1} \\ C_{1} & D_{1} \end{bmatrix}
\begin{bmatrix} x_{1} \\ f_{*} \end{bmatrix} =
\begin{bmatrix} \widetilde x_{1} \\ f \end{bmatrix}.
    \label{FBconnection}
    \end{align}
    We also define the {\em elimination operator}
    $\Gamma_{\ell}(U_{0}, U_{1})$ (when it exists) by
    \begin{equation}   \label{elimination}
    \Gamma_{\ell}(U_{0},U_{1}) \colon \begin{bmatrix}
     x_{0} \\ x_{1} \end{bmatrix}  \mapsto \begin{bmatrix}  f \\ f_{*}
     \end{bmatrix} \text{ if there exist } \widetilde x_{0},
     \widetilde x_{1} \text{ so that \eqref{FBconnection} holds.}
    \end{equation}
  As explained in the following result, the feedback connection and
  elimination operator exist
 and are well-defined as long as the operator $I - D_{1} D_{0}$ is
 invertible as an operator on $\cF$.

 \begin{theorem}  \label{T:FBcon}
  Suppose that we are given block-operator matrices $U_{0}$ and
  $U_{1}$ as in \eqref{U0U1}.
 Assume $(I - D_{1}D_{0})^{-1}$ and hence also $(I -
D_{0}D_{1})^{-1}$ exist as operators on $\cF$ and $\cF_{*}$
respectively.  Then the feedback connection \eqref{FBconnection} is
{\em well-posed}, i.e., for each $\sbm{ x_{0} \\ x_{1} } \in \cX_{0}
\oplus \cX_{1}$ there exists a unique $f \in \cF$ and $f_{*} \in
\cF_{*}$ so that the equations \eqref{FBconnection} determine a
unique $\sbm{\widetilde x_{0} \\ \widetilde x_{1}} \in  \widetilde
\cX_{0} \oplus  \widetilde \cX_{1}$ which we then define to be
$\cF_{\ell}(U_{0}, U_{1}) ( \sbm{ x_{0} \\ x_{1}})$.  More
explicitly, the feedback connection operator $\cF_{\ell}(U_{0},U_{1})
\colon \sbm{x_{0} \\ x_{1} } \mapsto \sbm{\widetilde x_{0} \\
\widetilde x_{1} }$ is given by
\begin{equation}  \label{FB}
    \cF_{\ell}(U_{0}, U_{1}) = \begin{bmatrix}
 A_{0} + B_{0} (I - D_{1}D_{0})^{-1} D_{1}C_{0} &
 B_{0}(I - D_{1}D_{0})^{-1} C_{1} \\
 B_{1}(I - D_{0}D_{1})^{-1}C_{0}  &
 A_{1} + B_{1}(I - D_{0}D_{1})^{-1}D_{0} C_{1} \end{bmatrix}.
 \end{equation}
 The elimination operator \eqref{elimination} which assigns instead
the uniquely
 determined $\left[ \begin{smallmatrix} f \\ f_{*} \end{smallmatrix}
 \right]$ to $\left[ \begin{smallmatrix} x_{0} \\ x_{1}
\end{smallmatrix} \right]$ is then given explicitly by
 \begin{equation}  \label{EL}
     \Gamma_{\ell}(U_{0}, U_{1}) = \begin{bmatrix} (I -
D_{1}D_{0})^{-1}D_{1}C_{0} &
     (I - D_{1}D_{0})^{-1}C_{1} \\
     (I - D_{0}D_{1})^{-1} C_{0} & (I - D_{0}D_{1})^{-1} D_{0} C_{1}
     \end{bmatrix} \colon \begin{bmatrix}
  x_{0} \\ x_{1} \end{bmatrix}  \mapsto \begin{bmatrix}  f \\ f_{*}
  \end{bmatrix}.
 \end{equation}
\end{theorem}

\begin{proof}
    The definition $\cF_{\ell}(U_{0}, U_{1}) \sbm{x_{0} \\ x_{1}} =
    \sbm{\widetilde x_{0} \\ \widetilde x_{1}}$ means that there is $f
    \in \cF$ and $f_{*} \in \cF_{*}$ so that \eqref{FBconnection}
    holds.
  From the second equation of the first system in
\eqref{FBconnection} we have
  $$
    f_{*} = C_{0} x_{0} +  D_{0} f.
  $$
  Plug this into the second equation of the second system to get
  \begin{align*}  f & = C_{1} x_{1} + D_{1} f_{*} = C_{1}x_{1} +
D_{1}(C_{0}x_{0} +
  D_{0} f)  \\
  & = D_{1} C_{0} x_{0} + C_{1} x_{1} + D_{1} D_{0} f.
  \end{align*}
  Under the assumption that $I - D_{1}D_{0}$ is invertible we then
  can solve for $f$ to get
  $$
  f = (I - D_{1} D_{0})^{-1} D_{1}C_{0} x_{0} + (I - D_{1}
  D_{0})^{-1} C_{1} x_{1}.
  $$
  Plug this into the second equation of the first system in
  \eqref{FBconnection} to then get
  \begin{align*}
  f_{*} & = C_{0} x_{0} + D_{0}(I - D_{1}D_{0})^{-1} D_{1}C_{0} x_{0}
+
  D_{0} (I - D_{1} D_{0})^{-1} C_{0} x_{1}  \\
  & = (I - D_{0}D_{1})^{-1}
  C_{0} x_{0} + (I - D_{0} D_{1})^{-1} D_{0} C_{1} x_{1}.
  \end{align*}
  In this way we get the formula \eqref{EL} for the elimination
  operator $\Gamma_{\ell}(U_{0}, U_{1})$.  It is now a simple matter
to plug in these
  values for $f, f_{*}$ in terms of $x_{0}, x_{1}$ into the first
  equations in the two systems \eqref{FBconnection} to arrive at the
  formula \eqref{FB} for the feedback connection operator
  $\cF_{\ell}(U_{0}, U_{1})$.
  \end{proof}

  While the formula \eqref{FB} exhibits $\cF_{\ell}(U_{0}, U_{1})$
    explicitly in terms of $U_{0}$ and $U_{1}$, direct computation of
powers
    $U^{n}$ ($n=2,3, \dots$) of $U = \cF_{\ell}(U_{0}, U_{1})$
appears to
    be rather laborious.  We next show how efficient computation of
    powers $\cF_{\ell}(U_{0}, U_{1})$ can be achieved by use of a
    feedback connection at the level of system trajectories.  Toward
    this end, we first introduce some useful notation.

    For $\cG$ any linear space, we let $\ell_{\cG}({\mathbb Z})$
    (alternatively often written as $\cG^{{\mathbb Z}}$ in the
    literature) denote the space of all $\cG$-valued functions on the
    integers ${\mathbb Z}$.  Similarly we let $\ell_{\cG}({\mathbb
    Z}_{+})$ be the space of all $\cG$-valued functions on the
    nonnegative integers ${\mathbb Z}_{+}$; we often identify
    $\ell_{\cG}({\mathbb Z}_{+})$ with the subspace of
    $\ell_{\cG}({\mathbb Z})$ consisting of all $\cG$-valued
functions on
    ${\mathbb Z}$      which vanish on the negative integers.
    Similarly, $\ell_{\cG}({\mathbb Z}_{-})$ is the space of all
    $\cG$-valued functions on ${\mathbb Z}_{-}$ and is frequently
    identified with the subspace of $\ell_{\cG}({\mathbb Z})$
    consisting of all $\cG$-valued functions on ${\mathbb Z}$
    vanishing on ${\mathbb Z}_{+}$.  By $\mathcal P^+$ and
$\mathcal P^-$ we denote the natural projections of
$\ell_{\cG}({\mathbb Z})$
onto $\ell_{\cG}({\mathbb Z}_+)$ and $\ell_{\cG}({\mathbb Z}_-)$,
respectively. Sometimes we will use notations
\begin{equation}\label{colligation-06}
\vec g^+=\mathcal P^+ \vec g, \quad \vec g^-=\mathcal P^- \vec g.
\end{equation}
We also consider the bilateral shift operator $J: \vec g \mapsto \vec
g'$,
where $\vec g'(n)=\vec g(n-1)$.

Given a colligation matrix $U$ of the form
\begin{equation}  \label{col}
 U = \begin{bmatrix} A & B \\ C & D \end{bmatrix} \colon
\begin{bmatrix} \cX \\ \cE \end{bmatrix} \to \begin{bmatrix} \cX \\
    \cE_{*} \end{bmatrix}
\end{equation}
we may consider the associated discrete-time input/state/output
linear system
\begin{equation}  \label{sys}
    \begin{bmatrix} x(n+1) \\ e_{*}(n) \end{bmatrix} = U
	\begin{bmatrix} x(n) \\ e(n) \end{bmatrix} =
	    \begin{bmatrix} A x(n) + B e(n) \\ C x(n) + D e(n)
	    \end{bmatrix}.
 \end{equation}
 Given an initial state $x(0) = x_{0}$ and an input string $\vec e
 \in \ell_{\cE}({\mathbb Z}_{+})$, the system equations \eqref{sys}
 recursively uniquely determine the state trajectory $\vec x \in
 \ell_{\cX}({\mathbb Z}_{+})$ and the output string $\vec e_{*} \in
 \ell_{\cE_{*}}({\mathbb Z}_{+})$; explicitly we have
 \begin{align}
 x(n) & = A^{n} x_{0} + \sum_{k=0}^{n-1} A^{n-1-k} B  e(k),
 \notag \\
  e_{*}(n) & = C A^{n} x_{0} + \sum_{k=0}^{n-1} C A^{n-1-k} B
 e(k) + D  e(n) \text{ for } n = 0,1,2, \dots.
 \label{sys-solve}
 \end{align}
 If we view elements of $\ell_{\cX}({\mathbb Z}_{+})$ (and of
 $\ell_{\cE}({\mathbb Z}_{+})$ and $\ell_{\cE_{*}}({\mathbb Z}_{+})$)
 as column vectors, then operators between these various spaces can
 be represented as block matrices.  We may then write the content of
 \eqref{sys-solve} in matrix form as
 $$
     \begin{bmatrix} \vec x \\ \vec e_{*} \end{bmatrix} =
	 \begin{bmatrix} W_{0}^{+} & W_{2} ^{+} \\ W_{1}^{+} & W^{+}
	     \end{bmatrix}  \begin{bmatrix} x(0) \\ \vec e
	 \end{bmatrix}
$$
where the block-operator matrix
\begin{equation}  \label{augIO}
 {\mathbf W}^{+} :=  \begin{bmatrix}  W_{0}^{+} & W_{2}^{+} \\
W_{1}^{+} & W^{+}
  \end{bmatrix} \colon \begin{bmatrix}  \cX \\ \ell_{\cE}({\mathbb
  Z}_{+}) \end{bmatrix} \to \begin{bmatrix}  \ell_{\cX}({\mathbb
Z}_{+})
  \\ \ell_{\cE_{*}}({\mathbb Z}_{+}) \end{bmatrix}
\end{equation}
is given explicitly by
\begin{align}
    & W_{0}^{+} = \begin{bmatrix} I_{\cX} \\ A \\ \vdots \\ A^{n-1}
    \\ \vdots \end{bmatrix}, \quad
    W_{2}^{+} = \begin{bmatrix} 0 & 0 & 0 & \dots & & & \\
    B & 0 & 0 & \dots & & & \\
    AB & B & 0 & \dots & & & \\
    \vdots & \vdots & \vdots & & & & \\
    A^{n-1} B & A^{n-2} B & A^{n-3}B & \dots & B & 0 & \dots & \\
    & \ddots & \ddots &   \ddots & & \ddots &  \ddots  \end{bmatrix},
    \notag \\
   & W_{1}^{+} = \begin{bmatrix} C \\ CA \\ CA^{2} \\ \vdots \\
   CA^{n-1} \\ \vdots \end{bmatrix}, \quad
   W^{+} = \begin{bmatrix} D & 0 & 0 & & & \\ CB & D & 0  & & \\
   CAB & CB & D & & \\
   \vdots & \vdots & &  \ddots & \\
   C A^{n-1}B & C A^{n-2}B & C A^{n-3}B & \dots & D & \\
     & \ddots    & \ddots  & \ddots  &  & \ddots
   \end{bmatrix}
   \label{sys-solve-explicit}
 \end{align}
 The operator $W_{0}^{+}$ is the (forward-time) {\em
initial-state/state-trajectory
 map}, the operator $W_{2}^{+}$ is the {\em input/state-trajectory
map},
 the operator $W_{1}^{+}$ is the {\em observation operator} and the
 operator $W^{+}$ is what is traditionally known as the {\em
 input-output map} in the control literature.  We note that the
 multiplication operator associated with $W^{+}$ after applying the
 $Z$-transform
 $$ \vec f \mapsto \sum_{n \in {\mathbb Z}_{+}} \vec f(n) z^{n}
 $$
 to the input and output strings $\vec e$ and $\vec e_{*}$
 respectively has multiplier
 $$
   \widehat W^{+}(z) = D + z C (I - zA)^{-1} B
 $$
 equal to the {\em characteristic function} of the colligation $U$
 (also known as the {\em transfer function} of the linear system
 \eqref{sys}).  We shall refer to the whole $2 \times 2$-block
 operator  matrix ${\mathbf W}^{+}$ simply as the (forward-time) {\em
augmented
 input/output map} associated with the colligation $U$.

 If the colligation matrix $U$ \eqref{col} is invertible, then we
 can also run the system in backwards time:
 \begin{equation} \label{backward-sys}
     \begin{bmatrix} x(n) \\ e(n) \end{bmatrix} =
	 U^{-1} \begin{bmatrix} x(n+1) \\ e_{*}(n) \end{bmatrix} =
	 \begin{bmatrix} \alpha x(n+1) + \beta e_{*}(n) \\ \gamma
	     x(n+1) + \delta e_{*}(n) \end{bmatrix}
 \end{equation}
 where we set
 $U^{-1} = \left[\begin{smallmatrix} \alpha & \beta \\ \gamma & \delta
\end{smallmatrix}  \right] \colon  \cX \oplus \cE_{*}
\to \cX  \oplus \cE$.
 In this case, specification of an initial state $x(0)$ and of the
 output string over negative time $\vec e_{*} \in
 \ell_{\cE_{*}}({\mathbb Z}_{-})$ determines recursively via the
 backward-time system equations \eqref{backward-sys} the
 state-trajectory over negative time $\vec x_{-} \in
 \ell_{\cX}({\mathbb Z}_{-})$ and the input string over negative time
 $\vec e_{-} \in \ell_{\cE}({\mathbb Z}_{-})$.  Explicitly we have
 \begin{align}
      x(n) & = \alpha^{n} x(0) + \sum_{k=1}^{n} \alpha^{n-k} \beta
     e_{*}(-k), \notag \\
       e(-n)  & = \gamma \alpha^{n-1} x(0) + \sum_{k=1}^{n-1} \gamma
     \alpha^{n-1} \beta e_{*}(-k) + \delta e_{*}(-n) \text{ for }
n=1, 2, \dots.
     \label{backward-solve-sys}
     \end{align}
 If we write elements $\vec x = \{ x(n)\}_{n \in {\mathbb Z}_{-}}$ of
 $\ell_{\cX}({\mathbb Z}_{-})$ as
 infinite column matrices
 $$
    \vec x = \begin{bmatrix} \vdots \\ \\ x(-3) \\ x(-2) \\  x(-1)
    \end{bmatrix}
 $$
 then linear operators between spaces of the type
 $\ell_{\cX}({\mathbb Z}_{-})$ can be written as  matrices
 with infinitely many rows as one ascends to the top.
 Then the relations \eqref{backward-solve-sys} can be expressed in
 $2 \times 2$-block operator matrix form as
 \begin{equation}  \label{backward-sys-block}
     \begin{bmatrix} \vec x_{-} \\ \vec e_{-} \end{bmatrix}
	 = \begin{bmatrix}  W_{0}^{-} & W_{1}^{-} \\ W_{2} ^{-} &
	 W^{-} \end{bmatrix} \begin{bmatrix}  x(0) \\ \vec e_{* -}
	 \end{bmatrix},
\end{equation}
where the $2 \times 2$-block operator matrix
$$
  {\mathbf W}^{-} : = \begin{bmatrix} W_{0}^{-} & W_{1}^{-} \\
  W_{2}^{-} & W^{-} \end{bmatrix} \colon \begin{bmatrix} \cX \\
  \ell_{\cE_{*}}({\mathbb Z}_{-}) \end{bmatrix} \to \begin{bmatrix}
  \ell_{\cX}({\mathbb Z}_{-}) \\ \ell_{\cE}({\mathbb Z}_{-})
\end{bmatrix}
$$
is given explicitly by
\begin{align}
    & W_{0}^{-} = \begin{bmatrix} \vdots \\ \alpha^{n}\\ \vdots \\
\alpha^{2}
    \\ \alpha \end{bmatrix}, \quad
    W_{1}^{-} = \begin{bmatrix} \ddots & & \ddots & \ddots & \ddots &
    \\  & \beta & \dots & \alpha^{n-3} \beta & \alpha^{n-2} \beta &
    \alpha^{n-1} \beta \\
    & & \ddots & \vdots  & \vdots & \vdots \\
    & & & \beta & \alpha \beta & \alpha^{2} \beta \\
    & & & & \beta & \alpha \beta \\
    & & & & & \beta \end{bmatrix}, \notag \\
    & W_{2}^{-} = \begin{bmatrix}  \vdots \\ \gamma \alpha^{n} \\
    \vdots \\ \gamma \alpha^{2} \\ \gamma \alpha \end{bmatrix}, \quad
    W^{-} = \begin{bmatrix} \ddots &  & \ddots  & \ddots & \ddots & &
\\
    & \delta & \dots & \gamma \alpha^{n-3} \beta & \gamma
    \alpha^{n-2} \beta & \gamma \alpha^{n-1} \beta  \\
    & & & \vdots & \vdots & \vdots  \\
    & & & \delta & \gamma \beta & \gamma \alpha \beta  \\
    & & & & \delta & \gamma \beta \\
    & & & & & \delta \end{bmatrix}.
    \label{backward-sys-explicit}
  \end{align}
  Here the operators $W_{0}^{-}$, $W_{1}^{-}$, $W_{2}^{-}$ and
$W^{-}$ are the
  {\em backward-time} versions of the {\em
  initial-state/state-trajectory, input/state-trajectory,
  observation} and {\em input/output} operators, respectively, and we
  refer to the aggregate operator ${\mathbf W}^{-}$ simply as the
  {\em backward-time augmented input-output map}.

  Let us now suppose that $U_{0}$ and $U_{1}$ are two colligation
matrices
  \begin{align}
      U_{0} & = \begin{bmatrix} A_{0} & B_{0} \\ C_{0} & D_{0}
\end{bmatrix} \colon \begin{bmatrix} \cX_{0} \\ \cD \end{bmatrix} \to
\begin{bmatrix} \cX_{0} \\ \cD_{*} \end{bmatrix}, \label{U0'} \\
  U_{1} & = \begin{bmatrix} A_{1} & B_{1} \\ C_{1} & D_{1}
\end{bmatrix} \colon \begin{bmatrix} \cX_{1} \\ \cD_{*} \end{bmatrix}
\to
\begin{bmatrix} \cX_{1} \\ \cD \end{bmatrix}
  \label{U1'}
 \end{align}
 such that $I - D_{1}D_{0}$ and hence also $I -
 D_{0}D_{1}$ are invertible on $\cD$ and on $\cD_{*}$ respectively.
 Then the feedback connection $U = \cF_{\ell}(U_{0}, U_{1})$ is
 well-defined as an operator on $\cX_{0} \oplus \cX_{1}$
 as explained in Theorem \ref{T:FBcon}.  Then we also have associated
 augmented input-output maps for $U_{0}$ and $U_{1}$ given by
 \begin{align}
    & {\mathbf W}(U_{0})^{+} = \begin{bmatrix} W(U_{0})_{0}^{+} &
     W(U_{0})^{+}_{2} \\ W(U_{0})^{+}_{1} & W(U_{0})^{+}
 \end{bmatrix} \colon \begin{bmatrix} \cX_{0} \\ \ell_{\cD}({\mathbb
 Z}_{+}) \end{bmatrix}  \to \begin{bmatrix} \ell_{\cX_{0}}({\mathbb
 Z}_{+}) \\ \ell_{\cD_{*}}({\mathbb Z}_{+}) \end{bmatrix}  \notag \\
& {\mathbf W}(U_{1})^{+} = \begin{bmatrix} W(U_{1})_{0}^{+} &
      W(U_{1})^{+}_{2} \\ W(U_{1})^{+}_{1} & W(U_{1})^{+}
  \end{bmatrix} \colon \begin{bmatrix} \cX_{1} \\
\ell_{\cD_{*}}({\mathbb
  Z}_{+}) \end{bmatrix}  \to \begin{bmatrix} \ell_{\cX_{1}}({\mathbb
  Z}_{+}) \\ \ell_{\cD}({\mathbb Z}_{+}) \end{bmatrix}.
  \label{bWU01}
 \end{align}
 Under the assumption that
 \begin{equation}  \label{assume}
     I_{\ell_{\cD}({\mathbb Z}_{+})} - W(U_{1})^{+} W(U_{0})^{+}
     \text{ is invertible on } \ell_{\cD}({\mathbb Z}_{+}),
 \end{equation}
 it makes sense to form the feedback connection $\cF_{\ell}({\mathbf
 W}(U_{0})^{+}, {\mathbf W}(U_{1})^{+})$.  The following lemma
guarantees
 that this connection is
 well-posed whenever the connection $\cF_{\ell}(U_{0}, U_{1})$ is
well-posed.

 \begin{lemma}  \label{L:sysFB-wellposed}
     Let $U_{0}$ and $U_{1}$ be as in \eqref{U0'} and \eqref{U1'}
     and assume that $I - D_{1} D_{0}$ is invertible on $\cD$.  Then
     also $I - W(U_{1})^{+} W(U_{0})^{+}$ is invertible on
     $\ell_{\cD}({\mathbb Z}_{+})$.
  \end{lemma}

  \begin{proof}
      From the formula for ${\mathbf W}^{+}$ in \eqref{augIO} and
\eqref{sys-solve-explicit}, we
      see that $W(U_{1})^{+}$ and $W(U_{0})^{+}$ are given by lower
      triangular Toeplitz matrices with diagonal entries equal to
      $D_{1}$ and $D_{0}$ respectively.   Hence $I - W(U_{1})^{+}
      W(U_{0})^{+}$ is also lower triangular Toeplitz with diagonal
      entry equal to $I - D_{1} D_{0}$.  A general fact is that an
      operator on $\ell_{\cD}({\mathbb Z}_{+})$ given by a lower
      triangular Toeplitz matrix with invertible diagonal entry is
      invertible on $\ell_{\cD}({\mathbb Z}_{+})$.  It follows that
      $I - W(U_{1})^{+} W(U_{0})^{+}$ is invertible on
      $\ell_{\cD}({\mathbb Z}_{+})$ as asserted.
      \end{proof}

      We now come to the main result of this section, namely: the
      computation of powers of $\cF_{\ell}(U_{0}, U_{1})$ via the
      feedback connection
      $\cF_{\ell}({\mathbf W}(U_{0}), {\mathbf W}(U_{1}))$.
      For this purpose it is convenient to introduce the following
      general notation.  For $U$ an operator on a linear space $\cK$
      and $\cG$ a subspace of $\cK$ with $i_{\cG}^{*} \colon \cK \to
      \cG$ the adjoint of the inclusion map $i_{\cG} \colon \cG \to
      \cK$, we define an operator $\Lambda_{\cG, +}(U) \colon \cK \to
      \ell_{\cG}({\mathbb Z}_{+})$ (called the {\em Fourier
      representation operator}) by
      \begin{equation}  \label{Lambda+}
       \Lambda_{\cG, +}(U) \colon k \to \{ i_{\cG}^{*} U^{n} k \}_{n
\in
       {\mathbb Z}_{+}}.
      \end{equation}
      Note that in case we take $\cG = \cK$ we have simply
      $$
       \Lambda_{\cK,+}(U) \colon k \to \{ U^{n} k \}_{n \in {\mathbb
       Z}_{+}}.
      $$

      \begin{theorem}  \label{T:FB-traj}
	  Suppose that we are given two colligation matrices
	  \eqref{U0'}, \eqref{U1'} such that $I - D_{1} D_{0}$ is
	  invertible on $\cD$ and we set $U = \cF_{\ell}(U_{0},
	  U_{1}) \in \cL(\cX_{0} \oplus \cX_{1})$. Then the
	  trajectory-level feedback connection operator
	  $\cF_{\ell}({\mathbf W}(U_{0})^{+}, {\mathbf
	  W}(U_{1})^{+})$ computes the powers of $U =
	  \cF_{\ell}(U_{0}, U_{1})$:
	  \begin{equation}  \label{Un}
	      \Lambda_{\cX_{0} \oplus \cX_{1},+}(U) =
	      \cF_{\ell}({\mathbf W}(U_{0})^{+}, {\mathbf
	      W}(U_{1})^{+}) \colon \cX_{0} \oplus \cX_{1} \to
	      \ell_{\cX_{0} \oplus \cX_{1}}({\mathbb Z}_{+}).
	  \end{equation}
	  Hence, after application of the natural identification
	  between the spaces $\ell_{\cX_{0} \oplus \cX_{1}}({\mathbb Z}_{+})$
	  and
	  $\ell_{\cX_{0}}({\mathbb Z}_{+}) \oplus
	  \ell_{\cX_{1}}({\mathbb Z}_{+})$, we have the explicit
	  formulas
	  $$ \Lambda_{\cX_{0} \oplus \cX_{1},+}(U) = \begin{bmatrix}
	  \Lambda_{\cX_{0} \oplus \cX_{1},+}(U)_{11} &
	  \Lambda_{\cX_{0} \oplus \cX_{1},+}(U)_{12}  \\
	  \Lambda_{\cX_{0} \oplus \cX_{1},+}(U)_{21} &
	  \Lambda_{\cX_{0} \oplus \cX_{1},+}(U)_{22} \end{bmatrix}
	  \colon \begin{bmatrix} \cX_{0} \\ \cX_{1} \end{bmatrix} \to
	  \begin{bmatrix} \ell_{\cX_{0}}({\mathbb Z}_{+}) \\
	      \ell_{\cX_{1}}({\mathbb Z}_{+}) \end{bmatrix}
	  $$
	  where the  matrix entries $\Lambda_{\cX_{0} \oplus
	  \cX_{1},+}(U)_{ij}$ ($i,j = 1,2$) are given explicitly by
	  \begin{align}
	    &  \Lambda_{\cX_{0} \oplus \cX_{1},+}(U)_{11} =
	    W(U_{0})_{0}^{+} + W(U_{0})_{2}^{+} (I -
	    W(U_{1})^{+}W(U_{0})^{+})^{-1} W(U_{1})^{+}
	    W(U_{0})_{1}^{+}, \notag \\
	 &   \Lambda_{\cX_{0} \oplus \cX_{1},+}(U)_{12} =
	 W(U_{0})_{2}^{+} (I - W(U_{1})^{+} W(U_{0})^{+})
	 W(U_{1})^{+}_{1}, \notag \\
	& \Lambda_{\cX_{0} \oplus \cX_{1},+}(U)_{21} =
	W(U_{1})_{2}^{+} (I - W(U_{0})^{+}W(U_{1})^{+})^{-1}
	W(U_{0})_{1}^{+}, \notag \\
	& \Lambda_{\cX_{0} \oplus \cX_{1},+}(U)_{22} =
	W(U_{1})_{0}^{+} + W(U_{1})_{2}^{+} (I - W(U_{0})^{+}
	W(U_{1})^{+})^{-1} W(U_{0})^{+} W(U_{1})_{1}^{+}.
	\label{Un-explicit}
	\end{align}
	\end{theorem}
	
	\begin{proof}
	    Note that Lemma \ref{L:sysFB-wellposed} guarantees that
	    the trajectory-level feedback connection
	    $\cF_{\ell}({\mathbf W}(U_{0})^{+}, {\mathbf
	    W}(U_{1})^{+})$ is well-posed.
	    By definition, we see that
	    \begin{equation}  \label{def-trajFB}
	     \cF_{\ell}({\mathbf W}(U_{0})^{+}, {\mathbf W}(U_{1})^{+})
	     \begin{bmatrix} x_{0} \\ x_{1}
	    \end{bmatrix} = \begin{bmatrix} \{ x_{0}(n)\}_{n \in {\mathbb
Z}_{+}}
	    \\ \{ x_{1}(n) \}_{n \in {\mathbb Z}_{+}} \end{bmatrix}
	    \end{equation}
	    means that
	    \begin{align}
		\begin{bmatrix} x_{0}(n+1) \\ d_{*}(n) \end{bmatrix} &  =
		    \begin{bmatrix} A_{0} & B_{0} \\ C_{0} & D_{0} \end{bmatrix}
			\begin{bmatrix} x_{0}(n) \\ d(n) \end{bmatrix}, \notag \\
			    \begin{bmatrix} x_{1}(n+1) \\ d(n) \end{bmatrix} & =
	    \begin{bmatrix} A_{1} & B_{1} \\ C_{1} & D_{1} \end{bmatrix}
		\begin{bmatrix} x_{1}(n) \\ d_{*}(n) \end{bmatrix}
	     \label{stringFB}
	     \end{align}
	     for uniquely determined strings $\{d(n)\}_{n \in {\mathbb
Z}_{+}}
	     \in \ell_{\cD}({\mathbb Z}_{+})$ and $\{d_{*}(n)\}_{n \in
{\mathbb
	     Z}_{+}} \in \ell_{\cD_{*}}({\mathbb Z}_{+})$.
	     As $U = \cF_{\ell}(U_{0}, U_{1})$, the particular case
	     $n = 0$
	    of the equations \eqref{stringFB} is just the
	    assertion that
	    $$
	    \begin{bmatrix} x_{0}(1) \\ x_{1}(1) \end{bmatrix} = U
		\begin{bmatrix} x_{0}(0) \\ x_{1}(0) \end{bmatrix}.
	    $$
	    Inductively assume that
	    \begin{equation} \label{ind-assume}
	     \begin{bmatrix} x_{0}(n) \\ x_{1}(n) \end{bmatrix} = U^{n}
		 \begin{bmatrix} x_{0}(0) \\ x_{1}(0) \end{bmatrix}.
	    \end{equation}
	    The $n$-th equation in \eqref{stringFB} amounts to the assertion
that
	    $$
	     \begin{bmatrix} x_{0}(n+1) \\ x_{1}(n+1) \end{bmatrix} = U
		 \begin{bmatrix} x_{0}(n) \\ x_{1}(n) \end{bmatrix}.
	    $$
	    Combining with the inductive assumption \eqref{ind-assume} then
gives
	    us that \eqref{ind-assume} holds with $n+1$ in place of $n$ and
hence
	    \eqref{ind-assume} holds for all $n = 0,1, 2,\dots$. Note
	    next that \eqref{ind-assume} combined with
	    \eqref{def-trajFB} amounts to the identity \eqref{Un}.
	    The explicit formulas \eqref{Un-explicit} then follow
	    from formula \eqref{FB} with ${\mathbf W}(U_{0})^{+}$,
	    ${\mathbf W}(U_{1})^{+}$ as in \eqref{bWU01} in place of
	    $U_{0}$, $U_{1}$.
	    \end{proof}
	
	    If $U_{0}$ and $U_{1}$ are invertible with
	    \begin{equation}  \label{U0U1inverse}
	    U_{0}^{-1} = \begin{bmatrix} \alpha_{0} & \beta_{0} \\
	    \gamma_{0} & \delta_{0} \end{bmatrix} \colon
	    \begin{bmatrix} \cX_{0} \\ \cD_{*} \end{bmatrix} \to
		\begin{bmatrix}  \cX_{0} \\ \cD \end{bmatrix},
		    \quad
	  U_{1}^{-1} = \begin{bmatrix} \alpha_{1} & \beta_{1} \\
	  \gamma_{1} & \delta_{1} \end{bmatrix} \colon
	  \begin{bmatrix} \cX_{1} \\ \cD \end{bmatrix} \to
	      \begin{bmatrix} \cX_{1} \\ \cD_{*} \end{bmatrix}
	  \end{equation}
	  with $I_{\cD_{*}} - \gamma_{1} \gamma_{0}$ invertible, a
	  similar analysis can be brought to bear to compute negative
	  powers of $U = \cF(U_{0}, U_{1})$.  Let us introduce an
      operator $\Lambda_{\cG,-} \colon \cK
	  \to \ell_{\cG}({\mathbb Z}_{-})$ (also called a {\em
	  Fourier representation operator} along with \eqref{Lambda+})
	  by
	  $$
	    \Lambda_{\cG,-}(U) \colon k  \mapsto  \{ i_{\cG}^{*}  U^{n} k
	    \}_{n \in {\mathbb Z}_{-}}
	  $$
	  where in particular
	  $$
	   \Lambda_{\cK,-}(U) k \colon  \mapsto \{ U^{n} k \}_{n \in {\mathbb
	   Z}_{-}}.
	   $$
	   Then we have the following backward-time result parallel
	   to Theorem \ref{T:FB-traj}.  As the proof is completely
	   analogous, we omit the details of the proof.
	
	   \begin{theorem}  \label{T:backwardFB-traj}
	       Suppose that we are given two colligation matrices
	       $U_{0}$ and $U_{1}$ as in
			 \eqref{U0'}, \eqref{U1'} with inverses as in
			 \eqref{U0U1inverse} such that $I -
			 \delta_{1} \delta_{0}$ is
			 invertible on $\cD_{*}$ and we set $U = \cF_{\ell}(U_{0},
			 U_{1}) \in \cL(\cX_{0} \oplus \cX_{1})$. Then the
			 trajectory-level feedback connection operator
			 $\cF_{\ell}({\mathbf W}(U_{0})^{-}, {\mathbf
			 W}(U_{1})^{-})$ computes the negative powers of $U =
			 \cF_{\ell}(U_{0}, U_{1})$:
			 \begin{equation}  \label{backward-Un}
			     \Lambda_{\cX_{0} \oplus \cX_{1},-}(U) =
			     \cF_{\ell}({\mathbf W}(U_{0})^{-}, {\mathbf
			     W}(U_{1})^{-}) \colon \cX_{0} \oplus \cX_{1} \to
			     \ell_{\cX_{0} \oplus \cX_{1}}({\mathbb
			     Z}_{-}).
			 \end{equation}
			 Hence, application of the natural identification
			 between $\ell_{\cX_{0} \oplus \cX_{1}}({\mathbb Z}_{-})$
			 and
			 $\ell_{\cX_{0}}({\mathbb Z}_{-}) \oplus
			 \ell_{\cX_{1}}({\mathbb Z}_{-})$ leads to explicit
			 formulas
			 $$ \Lambda_{\cX_{0} \oplus \cX_{1},-}(U) = \begin{bmatrix}
			 \Lambda_{\cX_{0} \oplus \cX_{1},-}(U)_{11} &
			 \Lambda_{\cX_{0} \oplus \cX_{1},-}(U)_{12}  \\
			 \Lambda_{\cX_{0} \oplus \cX_{1},-}(U)_{21} &
			 \Lambda_{\cX_{0} \oplus \cX_{1},-}(U)_{22} \end{bmatrix}
			 \colon \begin{bmatrix} \cX_{0} \\ \cX_{1} \end{bmatrix} \to
			 \begin{bmatrix} \ell_{\cX_{0}}({\mathbb Z}_{-}) \\
			     \ell_{\cX_{1}}({\mathbb Z}_{-}) \end{bmatrix}
			 $$
			 where the  matrix entries $\Lambda_{\cX_{0} \oplus
			 \cX_{1},+}(U)_{ij}$ ($i,j = 1,2$) are given explicitly by
			 \begin{align}
			   &  \Lambda_{\cX_{0} \oplus \cX_{1},-}(U)_{11} =
			   W(U_{0})_{0}^{-} + W(U_{0})_{1}^{-} (I -
			   W(U_{1})^{-}W(U_{0})^{-})^{-1} W(U_{1})^{-}
			   W(U_{0})_{2}^{-}, \notag \\
			&   \Lambda_{\cX_{0} \oplus \cX_{1},-}(U)_{12} =
			W(U_{0})_{1}^{-} (I - W(U_{1})^{-} W(U_{0})^{-})
			W(U_{1})^{-}_{2}, \notag \\
		       & \Lambda_{\cX_{0} \oplus \cX_{1},-}(U)_{21} =
		       W(U_{0})_{1}^{-} (I - W(U_{0})^{-}W(U_{1})^{-})^{-1}
		       W(U_{0})_{2}^{-}, \notag \\
		       & \Lambda_{\cX_{0} \oplus \cX_{1},+}(U)_{22} =
		       W(U_{1})_{0}^{-} + W(U_{1})_{1}^{-} (I -  W(U_{0})^{-}
		       W(U_{1})^{-})^{-1} W(U_{0})^{-} W(U_{1})_{2}^{-}.
		       \label{backward-Un-explicit}
		       \end{align}
		       \end{theorem}

\section{Unitary scattering systems and their models}
\label{S:Hellinger'}

\subsection{Unitary scattering systems}  \label{S:systems}

Following \cite{Kheifets-Dubov} we define a {\em unitary
scattering system} to be a collection ${\mathfrak S}$ of the
form
\begin{equation}  \label{scatsys}
 {\mathfrak S} = ({\mathcal U}, \Psi; {\mathcal K}, {\mathcal E})
\end{equation}
where ${\mathcal K}$ (the {\em ambient space}) and
${\mathcal E}$ (the {\em coefficient space}) are Hilbert spaces,
${\mathcal U}$ is a unitary operator on ${\mathcal K}$ (called the {\em
evolution} operator), and $\Psi$ (called the {\em scale} operator)
is an operator from ${\mathcal E}$ into ${\mathcal K}$. A
fundamental object associated with any unitary scattering
system ${\mathfrak S}$ \eqref{scatsys} is its so-called {\em
characteristic function} $w_{{\mathfrak S}}(\zeta)$ defined
by
\begin{align}
 w_{{\mathfrak S}}(\zeta) & = \sum_{n = 1}^{\infty}\Psi^{*}
 {\mathcal U}^{n} \Psi \overline{\zeta}^{n} + \sum_{n=0}^{\infty}
\Psi^{*} {\mathcal U}^{*n} \Psi \zeta^{n}  \notag \\
& = \Psi^{*} \left[ (I - \overline{\zeta} {\mathcal U})^{-1} +
(I -\zeta {\mathcal U}^{*})^{-1} - I \right] \Psi  \notag \\
& = (1 - |\zeta|^{2}) \Psi^{*} (I - \overline{\zeta} {\mathcal U})^{-1}
(I - \zeta {\mathcal U}^{*})^{-1} \Psi.
 \label{charfunc}
\end{align}
From \eqref{charfunc} we see that $w_{{\mathfrak S}}(\zeta)$
is a positive harmonic operator-function (values are
operators on $\mathcal E$)
\begin{equation}  \label{charfuncpos}
 w_{{\mathfrak S}}(\zeta) \ge 0 \text{ for } \zeta \in {\mathbb D}
 \text{ and } \frac{\partial^{2}}{\partial \overline{\zeta} \partial \zeta}
 w_{{\mathfrak S}}(\zeta) = 0 \text{ for all } \zeta \in {\mathbb D}.
 \end{equation}
If we introduce the convention
$$
\zeta^{[n]} = \begin{cases} \zeta^{n} &\text{if } n \ge 0 \\
\overline{\zeta}^{-n} &\text{if } n<0
\end{cases}
$$
then the first formula in \eqref{charfunc} can be written more succinctly as
$$
w_{{\mathfrak S}}(\zeta) = \sum_{n=-\infty}^{\infty} (\Psi^{*}
{\mathcal U}^{*n} \Psi) \zeta^{[n]}.
$$
We shall refer to the string of coefficients
$\{w_{{\mathfrak S},n}\}_{n \in {\mathbb Z}}$ given by
\begin{equation}  \label{char-moment-sequence}
w_{{\mathfrak S, n}}: = \Psi^{*} \cU^{*n} \Psi
\end{equation}
as the {\em characteristic moment sequence}.

Let us now introduce the spectral measure $E_{{\mathcal
U}}(\cdot)$ (see e.g.~\cite{Lax}) for ${\mathcal U}$; we then define the
{\em characteristic measure} $\sigma_{{\mathfrak S}}$ for the unitary scattering
system ${\mathfrak S}$ to be the spectral
measure for ${\mathcal U}$ compressed by the action of $\Psi$ given by
\begin{equation}  \label{charmeas}
\sigma_{{\mathfrak S}}(\cdot) = \Psi^{*} E_{{\mathcal U}}(\cdot) \Psi.
\end{equation}
Thus the spectral measure $E_{{\mathcal U}}$ is a strong Borel
measure on the unit circle ${\mathbb T}$ with values equal to
orthogonal projection operators in ${\mathcal L}({\mathcal K})$
while the characteristic measure $\sigma_{{\mathfrak S}}$ is a
strong Borel measure on ${\mathbb T}$ with values equal to
positive-semidefinite operators in ${\mathcal L}({\mathcal E})$.
Note that the characteristic function $w_{{\mathfrak
S}}(\zeta)$ can be expressed in terms of the characteristic measure
$\sigma_{{\mathfrak S}}$ via the Poisson integral:
\begin{align} \notag
w_{{\mathfrak S}}(\zeta) & = \Psi^{*}\left[ \int_{{\mathbb T}}
((1 - \overline{\zeta}t)^{-1} + ( 1 - \zeta \overline{t})^{-1}-1)\
E_{{\mathcal U}}(dt) \right] \Psi \\
& = \int_{{\mathbb T}} {\mathcal P}(t, \zeta)\ \sigma_{{\mathfrak S}}(dt)
\label{charfunc-measure}
\end{align}
where
\begin{equation}  \label{Poissonker}
{\mathcal P}(t, \zeta) = \frac{ 1- |\zeta|^{2}}{|1 -
\overline{\zeta}t|^{2}} \text{ for } t\in {\mathbb T} \text{ and }
\zeta \in {\mathbb D}
\end{equation}
is the classical {\em Poisson kernel} and where the Lebesgue
integral converges in the strong operator topology.

We say that two unitary scattering systems $({\mathcal U}, \Psi;
{\mathcal K}, {\mathcal E})$ and $({\mathcal U}', \Psi';
{\mathcal K}', {\mathcal E})$ (with the same coefficient space
${\mathcal E}$) are {\em unitarily equivalent} if there is a unitary
map $\tau \colon {\mathcal K} \to {\mathcal K}'$ so that
\begin{equation}  \label{unequiv}
\tau {\mathcal U} = {\mathcal U}' \tau, \qquad \tau \Psi = \Psi'.
\end{equation}
We say that a unitary scattering system
${\mathfrak S} = ({\mathcal U}, \Psi; {\mathcal K}, {\mathcal E})$
is {\em minimal} in case the linear manifold $\Psi {\mathcal E} \subset
{\mathcal K}$ is $*$-cyclic for ${\mathcal U}$, i.e., the smallest subspace
${\mathcal K}_{0}$ containing $\Psi {\mathcal E}$ and invariant
for both ${\mathcal U}$ and ${\mathcal U}^{*}$ is the whole space
${\mathcal K}$.   The following elementary result makes
precise the idea that {\em the characteristic function is a complete
unitary invariant for minimal unitary scattering systems}.

\begin{proposition}  \label{P:uninv}
Two unitarily equivalent unitary scattering systems
$$
{\mathfrak S} = ({\mathcal U}, \Psi; {\mathcal K}, {\mathcal E})
\text{ and } {\mathfrak S}' = ({\mathcal U}', \Psi'; {\mathcal K}',
{\mathcal E})
$$
have the same characteristic functions $(w_{{\mathfrak S}}(\zeta)
= w_{{\mathfrak S}'}(\zeta)$ for all $\zeta \in {\mathbb D}$).

Conversely, if ${\mathfrak S} = ({\mathcal U}, \Psi;
{\mathcal K}, {\mathcal E})$ and ${\mathfrak S}' = ({\mathcal U}', \Psi';
{\mathcal K}', {\mathcal E})$ are two {\em minimal} unitary
scattering systems with the same characteristic function,
then ${\mathfrak S}$ and ${\mathfrak S}'$ are unitarily equivalent.
\end{proposition}

\begin{proof}  This is essentially Theorem 4.1$'$ in \cite{Kheifets-Dubov}.
For the reader's convenience, we recall the proof here.
If $\tau \colon {\mathcal K} \to {\mathcal K}'$
satisfies the intertwining conditions \eqref{unequiv}, then
\begin{align*}
 w_{{\mathfrak S}'}(\zeta) & = \Psi^{\prime *}
 \left[ (I - \overline{\zeta} {\mathcal U}')^{-1} +
 (I - \zeta   {\mathcal U}^{\prime*})^{-1} - I \right] \Psi' \\
& = \Psi^{*}\tau^{*} \left[ (I -\overline{\zeta} {\mathcal U}')^{-1}
+ (I - \zeta  {\mathcal U}^{\prime *})^{-1} - I \right] \tau \Psi \\
 & = \Psi^{*} \left[ (I - \overline{\zeta} {\mathcal U})^{-1}
 + (I - \zeta {\mathcal U}^{*})^{-1} - I \right] \Psi \\
& = w_{{\mathfrak S}}(\zeta).
\end{align*}

Conversely, suppose that ${\mathfrak S}$ and ${\mathfrak
S}'$ are {\em minimal} unitary scattering systems with the same
coefficient space ${\mathcal E}$ and with identical characteristic functions
$w_{{\mathfrak S}}(\zeta) = w_{{\mathfrak S}'}(\zeta)$ for
all $\zeta \in {\mathbb D}$.  The identity
$w_{{\mathfrak S}}(\zeta) = w_{{\mathfrak S}'}(\zeta)$ between
harmonic functions implies that coefficients of powers of
$\zeta$ and of $\overline{\zeta}$ match up:
$$
\Psi^{*} {\mathcal U}^{n} \Psi = \Psi^{\prime *} {\mathcal U}^{\prime n}
\Psi' \text{ for } n = 0, \pm 1, \pm 2, \dots.
$$
Note that
$$
\langle {\mathcal U}^{n} \Psi e, {\mathcal U}^{m} \Psi
\widetilde e \rangle = \langle \Psi^{*} {\mathcal U}^{n-m} \Psi e,
\widetilde e \rangle = \langle \Psi^{\prime *} {\mathcal U}^{\prime n-m}
\Psi' e, \widetilde e \rangle  = \langle {\mathcal U}^{\prime n}
\Psi' e, {\mathcal U}^{\prime m} \Psi' \widetilde e \rangle
$$
for all $n,m = 0, \pm 1, \pm 2, \dots$ and $e, \widetilde e \in {\mathcal E}$,
and hence the formula
\begin{equation}  \label{formulaU}
\tau  : {\mathcal U}^{n} \Psi e \mapsto {\mathcal U}^{\prime n} \Psi' e
\end{equation}
extends by linearity and continuity to define a well-defined isometry
from $$ {\mathcal D} = \overline{\mathop{\text{span}}}
\{{\mathcal U}^{n} \Psi e \colon n \in {\mathbb Z},\ e \in {\mathcal E} \}
$$
onto
$$
{\mathcal R}:= \overline{\mathop{\text{span}}} \{{\mathcal U}^{\prime n}
\Psi' e \colon n \in {\mathbb Z},\ e \in {\mathcal E} \}.
$$
Under the assumption that both ${\mathfrak S}$ and ${\mathfrak S}'$
are minimal, we see that ${\mathcal D} = {\mathcal K}$ and ${\mathcal R} = {\mathcal K}'$,
and hence $\tau$ is unitary from ${\mathcal K}$ onto ${\mathcal K}'$.
From the formula \eqref{formulaU} for $\tau$ specialized to the case $n=0$,
we see that $\tau \Psi = \Psi'$.  From the general case of the formula we see that
$\tau {\mathcal U} k = {\mathcal U}' \tau k$ in case $k \in {\mathcal K}$ has the
form $k = {\mathcal U}^{n} \Psi^{*}e$ for some $n \in {\mathbb Z}$ and
$e \in {\mathcal E}$.  By the minimality assumption on ${\mathfrak
S}$ the span of such elements is dense in ${\mathcal K}$, and hence
the validity of the intertwining $\tau {\mathcal U} k =
{\mathcal U}' \tau k$ extends to the case of a general element $k$ of ${\mathcal
K}$.  This concludes the proof of Proposition \ref{P:uninv}.
\end{proof}

\subsection{\textbf{The Hellinger space ${\mathcal L}^{\sigma}$}} \label{S:Hellinger}

We mostly follow here definitions and notations of \cite{Kheifets-Dubov}.
Earlier expositions were given in \cite{KKY}, \cite{Kh-AIP1}, \cite{Kh-AIP2}, \cite{Kh-PhD},
\cite{KhYu}. Classical references for the Hellinger integral are \cite{Grenander},
\cite{Schmulyan1}, \cite{Schmulyan2}, \cite{Schmulyan3}; we refer to
\cite{RFM} for an application to stochastic differential equations.
		
Let ${\mathcal E}$ be a Hilbert space and let $\sigma$
be a positive ${\mathcal L}({\mathcal E})$-valued Borel
measure on ${\mathbb T}$. We define the space ${\mathcal L}^{\sigma}$
to be the space of all ${\mathcal E}$-valued vector measures $\nu$
for which there exists a scalar measure $\mu$ on ${\mathbb T}$
such that the operator
\begin{equation}  \label{majorize}
\begin{bmatrix}  \mu(\Delta) & \nu(\Delta)^{*} \\
\nu(\Delta) &
\sigma(\Delta) \end{bmatrix} \in {\mathcal L}({\mathbb C}
\oplus {\mathcal E} )
\end{equation}
is positive semidefinite for all Borel subsets $\Delta \subseteq {\mathbb T}$.
It follows that, for a given $\nu \in {\mathcal L}^{\sigma}$
and $\Delta$ any Borel subset of ${\mathbb T}$, $\nu(\Delta) \in
\mathop{\text{im}} \sigma(\Delta)^{1/2}$ and that the smallest
constant $C_{\nu}(\Delta)$ which can be substituted for
$\mu(\Delta)$ in \eqref{majorize} is
\begin{equation}  \label{Cnu}
    C_{\nu}(\Delta) = \| \sigma(\Delta)^{[-1/2]}\nu(\Delta)\|^{2}
\end{equation}
where $\sigma(\Delta)^{[-1/2]}$ is the Moore-Penrose inverse of
$\sigma(\Delta)^{1/2}$; here in general we use the notation $X^{[-1]}$ for the {\em
Moore-Penrose inverse} of the operator $X$:
\begin{align*}
&  \mathop{\text{domain}} X^{[-1]} = \mathop{\text{im}} X \oplus
(\mathop{\text{im X}})^{\perp}  \text{ with } \\
& X^{[-1]} x = 0 \text{ if } x \in (\mathop{\text{im
X}})^{\perp}
\text{ and } X^{[-1]}(Xy) = P_{ (\mathop{\text{ker}} X)^{\perp}} y.
\end{align*}
It can be further shown (see \cite{Kheifets-Dubov} for detail)
that there is a unique such dominating scalar measure $\mu_{\nu}$ which is {\em minimal}
in the sense that, for any other scalar measure $\mu$ for which
\eqref{majorize} holds, then
necessarily
$$ \mu(\Delta) \ge \mu_{\nu}(\Delta) \text{ for all } \Delta
\subseteq {\mathbb T}
$$
This unique measure $\mu_{\nu}$ is called the {\em Hellinger (scalar-valued)
measure of the vector measure $\nu$ with respect
to $\sigma$}. The function $\Delta \mapsto C_{\nu}(\Delta)$
in general is not even an additive function of the set $\Delta$
but merely subadditive, i.e., for $\Delta_{1}, \Delta_{2}$ disjoint
Borel subsets of ${\mathbb T}$, we are guaranteed only
$$
C_{\nu}(\Delta_{1} \cup \Delta_{2}) \le C_{\nu}(\Delta_{1}) + C_{\nu}(\Delta_{2}).
$$
Therefore, the Hellinger measure $\mu_{\nu}$ is defined as the
minimal measure dominating the subadditive function $C_{\nu}$, namely
\begin{align} \notag
 \mu_{\nu}(\Delta) & = \sup \{ \sum_{k=1}^{n} \|
\sigma(\Delta_{k})^{[-1/2]} \nu(\Delta_{k}) \|^{2}_{{\mathcal E}}
\colon  \Delta_{1}, \dots, \Delta_{n} \\ \notag
& \qquad \text{ form a finite disjunctiv partition of } \Delta\}  \\
& = \lim_{\Delta  = \dot \cup_{k=1}^{n} \Delta_{k}} \sum_{k=1}^{n}\|
\sigma(\Delta_{k})^{[-1/2]} \nu(\Delta_{k}) \|^{2}_{{\mathcal E}}
\label{munu}
\end{align}
where the limit is taken along the directed set of finite
partitions of $\Delta$ ordered by refinement. In particular, part
of the assertion here is that $\nu(\Delta) \in \mathop{\text{im}}
\sigma(\Delta)^{1/2}$ for each Borel subset $\Delta$ of
${\mathbb T}$.

We next define a norm  $\| \cdot \|_{{\mathcal L}^{\sigma}}$ by
$$
\| \nu \|_{{\mathcal L}^{\sigma}} = \mu_{\nu}({\mathbb T})^{1/2}.
$$
Then one can show that ${\mathcal L}^{\sigma}$ is a Hilbert space
in this norm (see \cite{Kheifets-Dubov}).  By polarization of the
formula \eqref{munu} we see that the inner product can be expressed as
\begin{equation}  \label{Lsigmainnerprod}
\langle \nu_{1}, \nu_{2} \rangle_{{\mathcal L}^{\sigma}} =
\lim_{{\mathbb T} = \dot \cup_{k=1}^{n} \Delta_{k}}
\sum_{k=1}^{n} \left\langle \sigma(\Delta_{k})^{[-1/2]} \nu_{1}(\Delta_{k}),
\sigma(\Delta_{k})^{[-1/2]} \nu_{2}(\Delta_{k})
\right\rangle_{{\mathcal E}}.
\end{equation}

We give below two propositions from \cite{Kheifets-Dubov} that will be used in
this paper. Both can be verified for simple functions using
formulas (\ref{munu}) and (\ref{Lsigmainnerprod}) and then for
measurable functions by the limit process.
		
\begin{proposition}  \label{L:1}
If $\nu \in {\mathcal L}^{\sigma}$ with associated Hellinger
scalar-valued measure $\mu_{\nu}$ and if $f$ is a measurable scalar function
on ${\mathbb T}$ for which $\int_{{\mathbb T}} |f(t)|^{2}
\mu_{\nu}(dt) < \infty$, then the measure $\nu \cdot f$ given by
$$
(\nu \cdot f)(\Delta) = \int_{\Delta}  \nu(dt) f(t)
$$
is in ${\mathcal L}^{\sigma}$ with associated Hellinger
measure $\mu_{\nu \cdot f}$ equal to the scalar-valued measure $\mu_{\nu} \cdot
|f|^{2}$ given by
$$
(\mu_{\nu} \cdot |f|^{2})(\Delta) = \int_{\Delta} \mu_{\nu}(dt) |f(t)|^{2}.
$$
\end{proposition}
			
\begin{proposition}\label{sigmabyf}
Let $f$ be any ${\mathcal E}$-valued measurable function on ${\mathbb T}$ for which
$$
\int_{{\mathbb T}} \langle \sigma(dt) f(t), f(t) \rangle < \infty,
$$
then the vector measure $\sigma \cdot f$ given by
$$
(\sigma \cdot f)(\Delta) = \int_{\Delta} d \sigma(t) f(t)
$$
is in ${\mathcal L}^{\sigma}$ with $\mu_{\sigma \cdot f} =f^{*} \cdot
\sigma \cdot f$, where $f^{*} \cdot \sigma \cdot f$ is the scalar
measure defined by
\begin{equation}\label{sigmafnorm}
(f^{*} \cdot \sigma \cdot f)(\Delta) = \int_{\Delta} \langle
\sigma(dt) f(t), f(t) \rangle_{{\mathcal E}}.
\end{equation}
Moreover, for every
$\nu\in{\mathcal L}^{\sigma}$ and every such $f$ it holds
that
\begin{equation}  \label{Lsigmapreinnerprod}
\langle \nu\ , \ \sigma \cdot f \rangle_{{\mathcal L}^{\sigma}}
= \int_{{\mathbb T}} \langle \nu(dt),f(t) \rangle_{{\mathcal E}}.
\end{equation}
\end{proposition}
		
We will need the following uniqueness result.
\begin{lemma}\label{C:unique}
If $\sigma$ and $\widetilde \sigma$ are two positive $\cL({\mathcal E})$-valued
measures such that $\cL^{\sigma} = \cL^{\widetilde \sigma}$ with identity of norms,
then $\sigma = \widetilde \sigma$.
\end{lemma}

\begin{proof}
Suppose that $\cL^{\sigma} = \cL^{\widetilde \sigma}$ with
$$
 \| \nu \|^{2}_{\cL^{\sigma}} = \| \nu \|^{2}_{\cL^{\widetilde \sigma}}
\text{ for all } \nu \in \cL^{\sigma} = \cL^{\widetilde \sigma}.
$$
Then this also holds for $\chi_{_\Delta}\nu$ with $\Delta$ an
arbitrary Borel set. Then this implies that
$\mu_\nu=\widetilde\mu_\nu$, where $\mu_\nu$ and $\widetilde\mu_\nu$
are the Hellinger measures of $\nu$ with respect to $\sigma$ and $\widetilde\sigma$.
Take $\nu=\sigma e$, then, in view of (\ref{munu}), $\mu_{\sigma
e}=e^*\sigma e$,
therefore, $\widetilde\mu_{\sigma e}=\mu_{\sigma e}=e^*\sigma e$.
This implies that
$$
\begin{bmatrix} e^*\sigma e & e^*\sigma \\
\sigma e & \widetilde \sigma
\end{bmatrix}\ge 0 .
$$
Therefore,
$$
\begin{bmatrix} e^*\sigma e & e^*\sigma e\\e^* \sigma e & e^* \widetilde \sigma e
\end{bmatrix}\ge 0.
$$
The latter implies that $ e^* \widetilde \sigma e \ge e^*\sigma e $.
In other words $ \widetilde \sigma \ge \sigma $. The reverse inequality
is obtained similarly by taking
$\nu=\widetilde\sigma e$ and using the fact that $\mu_{\widetilde\sigma e}=
\widetilde\mu_{\widetilde\sigma e}=e^*\widetilde\sigma e$.
The lemma follows.
\end{proof}

The connection between unitary scattering systems \eqref{scatsys}
and spaces of vector measures ${\mathcal L}^{\sigma}$ (see \eqref{majorize}) is
as follows.  A consequence of Proposition \ref{L:1} is that $\nu \cdot t \in
{\mathcal L}^{\sigma}$ whenever $\nu \in {\mathcal L}^{\sigma}$
with $\| \nu \cdot t \|_{{\mathcal L}^{\sigma}} = \| \nu\|_{{\mathcal L}^{\sigma}}$
(here $t$ stands for the function $f(t) = t$ on ${\mathbb T}$);
since the same story holds for $\nu \cdot t^{-1} = \nu \cdot \overline{t}$, we see that the
operator
$$
{\mathcal U}_{\sigma} \colon \nu \mapsto \nu \cdot t
$$
is unitary on ${\mathcal L}^{\sigma}$ with spectral measure
$E_{{\mathcal U}_{\sigma}}$ given by
$$
E_{{\mathcal U}_{\sigma}}(\Delta) \colon \nu \mapsto \nu \cdot \chi_{\Delta}
\text{ for } \nu \in {\mathcal L}^{\sigma}.
$$
It follows from the definition of the Hellinger space and the
Hellinger measure that $\sigma e \in {\mathcal L}^\sigma $ and
that $\mu_{\sigma e}(\Delta)=\langle\sigma(\Delta)e,e\rangle$.
In particular,
$$
 \Vert \sigma e \Vert^2_{{\mathcal L}^\sigma} = \langle\sigma
({\mathbb T}) e,e\rangle_{\mathcal E}\le\Vert\sigma({\mathbb T})\Vert\cdot\Vert e\Vert^2 .
$$
Therefore, the operator $\Psi_{\sigma}$ given by
$$
\Psi_{\sigma} \colon e \mapsto \sigma \cdot e
$$
is a bounded linear operator from ${\mathcal E}$ into ${\mathcal L}^{\sigma}$.
Hence the collection
\begin{equation}  \label{sigmascatsys}
    {\mathfrak S}_{\sigma} = ({\mathcal U}_{\sigma}, \Psi_{\sigma};
{\mathcal L}^{\sigma}, {\mathcal E})
\end{equation}
is a unitary scattering system for any positive ${\mathcal L}({\mathcal E})$-valued
measure $\sigma$. Moreover, an easy computation
\begin{align*}
\langle \Psi_{\sigma}^{*} E_{{\mathcal U}_{\sigma}}(\Delta)
\Psi_{\sigma}e, \ e \rangle_{{\mathcal L}^{\sigma}} & =
\langle E_{{\mathcal U}_{\sigma}}(\Delta) \sigma e,\ \sigma e\rangle_{{\mathcal L}^{\sigma}} \\
& = \langle \sigma \cdot \chi_{\Delta}e, \ \sigma e \rangle_{{\mathcal L}^{\sigma}} \\
& = \langle \sigma(\Delta) e, \ e \rangle_{{\mathcal E}} \text{
by  formula (\ref{Lsigmainnerprod})}
\end{align*}
shows that we recover the preassigned positive operator measure
$\sigma$ as the characteristic measure $\Psi_{\sigma}^{*}
E_{{\mathcal U}_{\sigma}}(\cdot) \Psi_{\sigma}$ for the unitary
scattering system ${\mathfrak S}_{\sigma}$.  In particular, {\em for any
positive ${\mathcal L}({\mathcal E})$-valued Borel measure
$\sigma$, there exists a unitary scattering system ${\mathfrak S}_{\sigma}$
having characteristic measure equal to $\sigma$}.  Since positive operator
measures $\sigma$ are in one-to-one correspondence with positive operator-valued harmonic
functions $w$ via the Poisson representation $w(\zeta) = \int_{{\mathbb T}}
{\mathcal P}(t, \zeta) \sigma(dt)$ (where ${\mathcal P}$ is the
Poisson kernel \eqref{Poissonker}), we also see that {\em given
any positive operator-valued harmonic function $w$, there is a unitary scattering system
${\mathfrak S}_{w}$ having characteristic function \eqref{charfunc} equal to $w$}.
The following result (Theorem 4.1 of \cite{Kheifets-Dubov})
is the converse to this statement.

\begin{theorem}  \label{T:Lsigmamodel}
Let ${\mathfrak S} = ({\mathcal U}, \Psi; {\mathcal K},
{\mathcal E})$ be a unitary scattering system as in \eqref{scatsys}.  Let
$\sigma$ be the associated characteristic measure
$$
\sigma(\cdot) = \Psi^{*} E_{{\mathcal U}}(\cdot) \Psi.
$$
For $k \in {\mathcal K}$, define an ${\mathcal E}$-valued Borel
measure $\nu_{k}$ on ${\mathbb T}$ by
$$
\nu_{k} = \Psi^{*} E_{{\mathcal U}} k.
$$
Then the operator $\Phi^{\wedge m}$ (the {\em vector-measure
valued Fourier representation operator for ${\mathfrak S}$}) given by
\begin{equation}  \label{Fourierrep}
\Phi^{\wedge m} \colon k \mapsto \nu_{k}
\end{equation}
is a coisometry from ${\mathcal K}$ onto ${\mathcal L}^{\sigma}$
satisfying the intertwining relations
$$
\Phi^{\wedge m} {\mathcal U} = {\mathcal U}_{\sigma} \Phi^{\wedge m},
\qquad \Phi^{\wedge m} \Psi  = \Psi_{\sigma}
$$
with initial space equal to the smallest reducing subspace for
${\mathcal U}$ containing $\mathop{\text{im}} \Psi$.
In particular, if ${\mathfrak S}$ is minimal, then the unitary
scattering system ${\mathfrak S}$ is unitarily equivalent (via
the unitary operator $\Phi^{\wedge m} \colon {\mathcal K} \to
{\mathcal L}^{\sigma}$) to the model unitary scattering system
${\mathfrak S}_{\sigma}$ given by \eqref{sigmascatsys}.
\end{theorem}

\begin{remark}  \label{R:tH}
{\em A space of vector-valued harmonic functions on the unit disk ${\mathbb D}$
can be associated to the space $\cL^\sigma$ of the vector-valued measures on the circle
${\mathbb T}$ via integration against the Poisson kernel. The formalism here
can be translated from measures on the circle to harmonic functions on the disk.  Such an analysis
is worked out in \cite{Kh-AIP1, Kh-AIP2, KKY, Kh-PhD, Kheifets-Dubov, FtHK2}.

We also mention that a functional model using formal reproducing kernel Hilbert spaces
(rather than measures and Hellinger spaces) for analogues of unitary scattering systems
as defined in Subsection \ref{S:systems}, where the unitary operator $\cU$ is
replaced by (1) a tuple of commuting unitary operators $(\cU_{1}, \dots, \cU_{d})$, or
(2) a row-unitary operator $\begin{bmatrix} \cU_{1} & \cdots & \cU_{d}\end{bmatrix}$
(i.e., a representation of the Cuntz algebra), is presented in \cite{FRKHS}.
} \end{remark}

In the sequel we shall have need of the following result
concerning orthogonal decompositions of Hellinger spaces
 $\cL^{\sigma}$(see Theorem 2.8 in \cite{Kheifets-Dubov}).

\begin{theorem}  \label{T:Hellinger-decom}
Suppose that the coefficient Hilbert space $\cE$ has an
orthogonal direct-sum decomposition $\cE = \cE_{1} \oplus\cE_{2}$ and that
$\sigma$ is an $\cL(\cE)$-valued positive measure with block-matrix decomposition
$$
\sigma = \begin{bmatrix} \sigma_{11} & \sigma_{12} \\
\sigma_{21} & \sigma_{22} \end{bmatrix} \text{ where }
\sigma_{ij}(\Delta) \in \cL(\cE_{j}, \cE_{i}) \text{ for each Borel set }
\Delta \text{ and for } i,j=1,2.
$$
Define subspaces $\cL^{\sigma}_{1}$ and $\cL^{\sigma}_{2}$ of $\cL^{\sigma}$ by
\begin{align*}
& \cL^{\sigma}_{1} = \left\{ \nu =
\begin{bmatrix} \nu_{1} \\ \nu_{2} \end{bmatrix} \in \cL^{\sigma}
\colon \| \nu\|_{\cL^{\sigma}} = \| \nu_{1}\|_{\cL^{\sigma_{11}}} \right\}, \\
 & \cL^{\sigma}_{2} = \left\{ \nu = \begin{bmatrix} \nu_{1} \\ \nu_{2} \end{bmatrix}
 \in \cL^{\sigma} \colon \nu_{1} = 0 \right\}.
\end{align*}
Then:
\begin{enumerate}
 \item
 $\cL^{\sigma}_{1} = \cL^{\sigma}-\textup{clos.}
\left\{ \sigma \begin{bmatrix} I \\ 0 \end{bmatrix} p
\colon p \in \cE_{1}[t, t^{-1}] \right\}$ where
$\cE_{1}[t,t^{-1}]$ is the space of trigonometric
polynomials with coefficients in $\cE_{1}$, and
\item $\cL^{\sigma}_{2} = \cL^{\sigma} \ominus
\cL^{\sigma}_{1} = \left\{ \begin{bmatrix} 0 \\ q\end{bmatrix}
\colon q \in \cL^{\sigma_{11}^{\perp}} \right\}$
where $\sigma_{11}^{\perp} = \sigma_{22} - \sigma_{21}\sigma_{11}^{[-1]} \sigma_{12}$
is the measure Schur-complement of $\sigma_{11}$ inside $\sigma$ in
the sense of \cite[Section 2]{Kheifets-Dubov}. In particular,
\begin{equation}   \label{L1=L}
\cL^{\sigma}_{1} = \cL^{\sigma} \text{ if and only if } \sigma_{11}^{\perp} = 0.
\end{equation}
\end{enumerate}
\end{theorem}

\begin{proof} We prove here Statement (1) which is contained
implicitly in Theorem 2.8 of \cite{Kheifets-Dubov}.
		
For $p \in \cE_{1}[t,t^{-1}]$, the computation
\begin{align*}
\left\langle \sigma \begin{bmatrix} I \\ 0 \end{bmatrix} p,\,
\sigma \begin{bmatrix} I \\ 0 \end{bmatrix} p \right\rangle_{\cL^{\sigma}}
& =  \int_{{\mathbb T}} p(t)^{*} \begin{bmatrix} I & 0\end{bmatrix} \sigma(dt)
\begin{bmatrix} I \\ 0 \end{bmatrix}p(t)  \\
& = \int_{{\mathbb T}} p(t)^{*} \sigma_{11}(dt) p_{11}(t)  \\
& = \langle \sigma_{11} p, \, \sigma_{11} p \rangle_{\cL^{\sigma_{11}}}
\end{align*}
shows that the map
$$
\sigma_{11}p \mapsto \begin{bmatrix} \sigma_{11} \\ \sigma_{21}
\end{bmatrix} p \text{ for } p \in \cE_{1}[t,t^{-1}]
$$
embeds a dense subset of $\cL^{\sigma_{11}}$ isometrically into
$\cL^{\sigma}_{1}$. The fact that the image of this map is dense in
$\cL^{\sigma}_{1}$ follows from the definition of $\cL^{\sigma}_{1}$ and from the
fact that $\{ \sigma_{11}p \colon p \in \cE_{1}[t,t^{-1}] \}$ is dense in $\cL^{\sigma_{11}}$.
Statement (2) is proved in Theorem 2.8 of \cite{Kheifets-Dubov}.
\end{proof}

\section{Intertwiners and unitary couplings of unitary operators}
\label{S:Intertwiners-Couplings}

In this section we present some preliminary material on unitary
couplings due originally to Adamjan and Arov \cite{Adamjan-Arov}
which is needed for a reformulation of the Lifting Problem to be
presented in the next section.

Suppose that we are given unitary operators $(\cU',{\mathcal K}')$ and
$(\cU'',{\mathcal K}'')$.  We say that the collection $(\cU, i_{{\mathcal K}'},
i_{{\mathcal K}''}; \cK)$ is an {\em Adamjan-Arov unitary coupling} (or, more briefly,
{\em AA-unitary coupling}) of $(\cU', {\mathcal K}')$ and $(\cU'',
{\mathcal K}'')$
if $\cU$ is a unitary operator on the Hilbert space ${\mathcal K}$ and
$i_{{\mathcal K}'} \colon {\mathcal K}' \to  {\mathcal K}$ and $i_{{\mathcal K}''} \colon
{\mathcal K}'' \to  {\mathcal K}$ are isometric embeddings of ${\mathcal K}'$
and ${\mathcal K}''$ respectively into ${\mathcal K}$ such that
\begin{equation}  \label{AAembedding}
i_{{\mathcal K}'} {\mathcal U}' = {\mathcal U} i_{{\mathcal K}'}, \qquad
i_{{\mathcal K}''} {\mathcal U}'' = \cU i_{{\mathcal K}''}.
\end{equation}
In this case it is clear that $Y = i_{\cK''}^{*} i_{\cK'} \colon \cK' \to \cK''$
is contractive ($\| Y \| \le 1$) and that $Y$ intertwines $\cU'$ with $\cU''$ since
$$ Y \cU' = i_{\cK''}^{*} i_{\cK'}\cU' = i_{\cK''}^{*} \cU i_{\cK'} =
\cU'' i_{\cK''}^{*} i_{\cK'} =  \cU'' Y.
$$
The following theorem provides a converse to this observation. To
formalize the ideas, let us say that the AA-unitary coupling
$(\cU, i_{\cK'}, i_{\cK''}; \cK)$ is {\em minimal} in case
\begin{equation}  \label{minAAcoupling}
\operatorname{im} i_{{\mathcal K}'} + \operatorname{im}
i_{{\mathcal K}''} \text{ is dense in } {\mathcal K}
\end{equation}
and that two AA-unitary couplings $({\mathcal U}, i_{{\mathcal K}'}, i_{{\mathcal K}''};
{\mathcal K})$ and $(\widetilde {\mathcal U}, \widetilde{i}_{{\mathcal K}'},
\widetilde{i}_{{\mathcal K}''}; \widetilde {\mathcal K})$ of $({\mathcal U}', {\mathcal K}')$
and $({\mathcal U}'',{\mathcal K}'')$ are {\em unitarily equivalent} if there is a
unitary operator $\tau \colon {\mathcal K} \to  \widetilde{\mathcal K}$
such that
\begin{equation}  \label{unequivAAcoupl}
 \tau {\mathcal U} = \widetilde {\mathcal U} \tau, \qquad
 \tau i_{{\mathcal K}'} = \widetilde{i}_{{\mathcal K}'},
\qquad
\tau i_{{\mathcal K}''} = \widetilde{i}_{{\mathcal K}''}.
\end{equation}
Then we have the following fundamental connection between AA-unitary couplings
and contractive intertwiners of two given unitary operators
(see e.g.~\cite{Adamjan-Arov} and \cite{CS}).

\begin{theorem} \label{T:AAcoupling}
Suppose that we are given two unitary operators $\cU'$ and
 $\cU''$ on Hilbert spaces ${\mathcal K}'$ and ${\mathcal K}''$, respectively.
Then there is a one-to-one correspondence between
unitary equivalence classes of minimal AA-unitary couplings
$({\mathcal U}, i_{{\mathcal K}'}, i_{{\mathcal K}''};
{\mathcal K})$ and contractive intertwiners $Y \colon {\mathcal K}' \to
{\mathcal K}''$ of $(\cU', {\mathcal K}')$ and $(\cU'',
{\mathcal K}'')$.  More precisely:
\begin{enumerate}
\item Suppose that ${\mathfrak A}: = (\cU, i_{{\mathcal K}'},
i_{{\mathcal K}''}; {\mathcal K})$ is an AA-unitary coupling of $({\mathcal U}',
{\mathcal K}')$ and $(\cU''$, ${\mathcal K}'')$.
Define an operator $Y = Y({\mathfrak A})   \colon {\mathcal K}' \to  {\mathcal K}''$
via
\begin{equation}  \label{AAcoupop}
Y = Y({\mathfrak A}) = i_{{\mathcal K}''}^{*} i_{{\mathcal K}'}.
\end{equation}
Then $Y$ is a contractive intertwiner of $(\cU', {\mathcal K}')$ and
$({\mathcal U}'', {\mathcal K}'')$. Equivalent AA-unitary couplings produce
the same intertwiner $Y$ via  \eqref{AAcoupop}.

\item Suppose that $Y \colon {\mathcal K}' \to  {\mathcal K}''$
is a contractive intertwiner of $(\cU',  {\mathcal K}')$ and $({\mathcal U}'', {\mathcal K}'')$.
Define a Hilbert space ${\mathcal K} := {\mathcal K}'' *_{Y}{\mathcal K}'$
as the completion of the space
$\left[\begin{smallmatrix} {\mathcal K}'' \\{\mathcal K}' \end{smallmatrix} \right]$
in the inner product
\begin{equation}  \label{coupled-innerprod}
\left\langle \begin{bmatrix} k'' \\ k' \end{bmatrix},
\begin{bmatrix} h'' \\ h' \end{bmatrix} \right\rangle_{{\mathcal K}' *_{Y}
    {\mathcal K}''} = \left\langle \begin{bmatrix} I_{{\mathcal K}''}
    & Y \\ Y^{*} & I_{{\mathcal K}'} \end{bmatrix} \begin{bmatrix} k'' \\ k'
\end{bmatrix},
\begin{bmatrix} h'' \\ h' \end{bmatrix} \right\rangle_{{\mathcal K}''
    \oplus {\mathcal K}'}
\end{equation}
for  $k'',h'' \in {\mathcal K}''$ and $k',h' \in {\mathcal K}'$
(with pairs $\sbm {k'' \\ k' } $ with zero self inner product identified with $0$).
Define an operator ${\mathcal U} = {\mathcal U}''*_{Y}{\mathcal U}'$
densely on ${\mathcal K}$ by
\begin{equation}\label{AAoperator}
{\mathcal U} \colon \bbm{ k'' \\ k' } \mapsto \bbm {\cU'' k'' \\ \cU'k' }
\end{equation}
together with inclusion maps $i_{{\mathcal K}'} \colon {\mathcal K}' \to  {\mathcal K}$
and $i_{{\mathcal K}''} \colon {\mathcal K}'' \to  {\mathcal K}$ given by
\begin{equation}  \label{inclusions}
i_{{\mathcal K}'} \colon k' \mapsto \bbm{ 0 \\ k' }, \qquad
i_{{\mathcal K}''} \colon k'' \mapsto \bbm{ k'' \\ 0 }.
\end{equation}
Then the resulting collection
\begin{equation} \label{Y-cU}
{\mathfrak A} = {\mathfrak A}(Y):= ({\mathcal U}''*_{Y}{\mathcal U}', i_{{\mathcal K}'},
i_{{\mathcal K}''}; {\mathcal K}''*_{Y} {\mathcal K}')
\end{equation}
is a minimal AA-unitary coupling of $(\cU', {\mathcal K}')$ and $(\cU'',
{\mathcal K}'')$ such that we recover $Y$ as $Y = i_{\cK''}^{*} i_{\cK'}$.
Any minimal AA-unitary coupling ($\cU, i_{\cK'}, i_{\cK''};\cK)$
of $(\cU', \cK')$ and $(\cU'', \cK'')$ with intertwiner $Y$ as in \eqref{AAcoupop}
is unitarily equivalent to the one defined by \eqref{coupled-innerprod}, \eqref{AAoperator},
\eqref{inclusions}.
\end{enumerate}
Moreover, the maps ${\mathfrak A} \mapsto Y({\mathfrak A})$ in (1) and
$Y \mapsto {\mathfrak A}(Y)$ in (2) are inverse to each other and set up a one-to-one
correspondence between contractive intertwiners $Y$ and unitary equivalence classes of minimal
AA-unitary couplings of $(\cU',\cK')$ and $(\cU'', \cK'')$.
\end{theorem}

\begin{proof}  The proof of part (1) was already observed in the discussion preceding
the statement of the theorem.

Conversely, suppose that $Y \colon {\mathcal K}' \to {\mathcal K}''$
is any contractive intertwiner of $(\cU', {\mathcal K}')$ and $(\cU'', {\mathcal K}'')$
and let $(\cU'*_{Y}\cU'',i_{{\mathcal K}'}, i_{{\mathcal K}''};
{\mathcal K}'*_{Y}{\mathcal K}'')$ be the AA-unitary coupling of $(\cU', {\mathcal K}')$
and $(\cU'', {\mathcal K}'')$ given by \eqref{Y-cU}.  From the form of the inner product
\eqref{coupled-innerprod}, we see that the maps $i_{{\mathcal K}'} \colon {\mathcal K}' \to
{\mathcal K}$ and $i_{{\mathcal K}''} \colon {\mathcal K}'' \to  {\mathcal K}$
given by \eqref{inclusions} are isometric. By the definition of ${\mathcal K}$
as the completion of $\sbm{{\mathcal K}'' \\ {\mathcal K}' }$ in the
${\mathcal K}'' *_{Y} {\mathcal K}'$-inner product, we
see that the span of the images $\operatorname{im}
i_{{\mathcal K}'} + \operatorname{im} i_{{\mathcal K}''}$ is dense in
${\mathcal K}:= {\mathcal K}' *_{Y} {\mathcal K}''$ by construction.
By using the intertwining condition $Y \cU' = \cU'' Y$
together with the unitary property of ${\mathcal U}''$ and ${\mathcal
U}'$, we see that
\begin{align*}
&    \left\langle \begin{bmatrix} I_{{\mathcal K}''} & Y \\ Y^{*}
& I_{{\mathcal K}'} \end{bmatrix}
\begin{bmatrix} {\mathcal U}'' k'' \\{\mathcal U}' k' \end{bmatrix},
\begin{bmatrix}{\mathcal U}'' h'' \\ {\mathcal U}' h' \end{bmatrix}
 \right\rangle_{{\mathcal K}'' \oplus {\mathcal K}'} \\
& \qquad  \qquad = \left\langle \begin{bmatrix} {\mathcal U}'' & 0 \\ 0 &
{\mathcal U}' \end{bmatrix}
\begin{bmatrix} I_{{\mathcal K}''} & Y \\ Y^{*} & I_{{\mathcal K}'} \end{bmatrix}
\begin{bmatrix} k'' \\ k' \end{bmatrix},
\begin{bmatrix} {\mathcal U}'' & 0 \\ 0 & {\mathcal U}' \end{bmatrix}
    \begin{bmatrix} h'' \\ h' \end{bmatrix} \right\rangle_{{\mathcal K}'' \oplus {\mathcal K}'}  \\
& \qquad \qquad = \left\langle \begin{bmatrix} I_{{\mathcal K}''} & Y \\ Y^{*} &
I_{{\mathcal K}'} \end{bmatrix} \begin{bmatrix} k'' \\ k' \end{bmatrix},
\begin{bmatrix} h'' \\ h' \end{bmatrix} \right\rangle_{{\mathcal K}'' \oplus {\mathcal K}'}
\end{align*}
and hence the operator
\begin{equation}  \label{Ucoupled}
{\mathcal U} \colon \begin{bmatrix} k'' \\ k' \end{bmatrix}
\mapsto \begin{bmatrix} {\mathcal U}'' k'' \\ {\mathcal U}' k' \end{bmatrix}
\end{equation}
extends to define a unitary operator on ${\mathcal K} = {\mathcal
K}'' *_{Y} {\mathcal K}'$.  We have thus verified that $(\cU'*_{Y}\cU'',
i_{{\mathcal K}'}, i_{{\mathcal K}''}; {\mathcal K}'*_{Y}{\mathcal K}'')$
defined as in \eqref{Y-cU} is a minimal AA-unitary coupling of $(\cU', {\mathcal K}') $
and $(\cU'', {\mathcal K}'')$, and statement (2) of the theorem follows.
From the form of the inner product, we see that we recover $Y$ as $Y =
i_{{\mathcal K}''}^{*} i_{{\mathcal K}'}$.

Suppose now that $(\cU, i_{{\mathcal K}'}, i_{{\mathcal K}''};
{\mathcal K})$ is any AA-unitary coupling of $(\cU', {\mathcal K}')$ and
$(\cU'', {\mathcal K}'')$ and we set $Y = i_{{\mathcal K}''}^{*} i_{{\mathcal K}'}
\colon {\mathcal K}' \to  {\mathcal K}''$.  For $k', \ell' \in {\mathcal K}'$ and
$k'', \ell'' \in {\mathcal K}''$, we compute
\begin{align*}
&  \left\langle \bbm{ I & Y \\ Y^{*} & I } \bbm{ k'' \\ k' },
\bbm {\ell'' \\ \ell' } \right\rangle  =
	\langle k'' + i_{{\mathcal K}''}^{*} i_{{\mathcal K}} k',
	\ell'' \rangle_{{\mathcal K}''} +
\langle  i_{{\mathcal K}'}^{*} i_{{\mathcal K}''} k'' + k',
\ell' \rangle_{{\mathcal K}'} \\
&\qquad  = \langle k'', \ell'' \rangle_{{\mathcal K}''}
+ \langle i_{{\mathcal K}'} k', i_{{\mathcal K}''} \ell''
\rangle_{{\mathcal K}} + \langle i_{{\mathcal K}''}k'',
i_{{\mathcal K}'} \ell' \rangle_{{\mathcal K}} +
\langle k', \ell'  \rangle_{{\mathcal K}'} \\
& \qquad = \langle i_{{\mathcal K}''}k'' + i_{{\mathcal K}'}k',
i_{{\mathcal K}''}\ell'' + i_{{\mathcal K}'} \ell'
\rangle_{{\mathcal K}}.
\end{align*}
We conclude that the map
$$
\tau \colon \bbm{ k'' \\ k' } \mapsto i_{{\mathcal K}''}k'' +
i_{{\mathcal K}'} k'
$$
maps the dense subspace $\sbm{ {\mathcal K}'' \\ {\mathcal K}'}$ of
${\mathcal K}'' *_{Y} {\mathcal K}'$ onto $\operatorname{im}
i_{{\mathcal K}'} + \operatorname{im} i_{{\mathcal K}''}$.
Hence $\tau$ extends to an isometric mapping of all of ${\mathcal K}'' *_{Y}
{\mathcal K}'$ onto the closure of  $\operatorname{im} i_{{\mathcal K}''}
+ \operatorname{im} i_{{\mathcal K}'}$ and hence
$\tau$ is unitary exactly when $(\cU, i_{{\mathcal K}'}, i_{{\mathcal
K}''}; {\mathcal K})$
is a minimal AA-unitary coupling.  Moreover,
from the form of $\tau$ it is easily verified that $\tau (\cU' *_{Y}\cU'') = \cU \tau$.
By definition of $\tau$, it transforms the embeddings of ${\mathcal K}'$ and  ${\mathcal K}''$
into ${\mathcal K}'' *_{Y} {\mathcal K}'$ to the embeddings of
${\mathcal K}'$ and  ${\mathcal K}''$ into ${\mathcal K}$.
In this way we see that the above correspondence between
contractive intertwiners and minimal AA-unitary couplings is bijective.
			    bijective.  \end{proof}
			
Given an AA-unitary coupling $(\cU, i_{\cK'}, i_{\cK''}; \cK)$ of the
unitary operators $(\cU'', \cK'')$ and $(\cU', \cK')$ and
two subspaces $\cG' \subset \cK'$ and $\cG'' \subset \cK''$
which are $*$-cyclic for $\cU'$ and $\cU''$ respectively, let $i_{\cG'}
\colon \cG' \to \cK$ and $i_{\cG''} \colon \cG'' \to \cK$ be the
compositions of the inclusion of $\cG'$ into $\cK'$ with the
inclusion of $\cK'$ into $\cK$ and of the inclusion of $\cG''$
into $\cK''$ with the inclusion of $\cK''$ into $\cK$, respectively:
\begin{align*}
&  i_{\cG'}: = i_{\cG' \to \cK} = i_{\cK' \to \cK} i_{\cG' \to \cK'}, \\
& i_{\cG''}: = i_{\cG'' \to \cK} = i_{\cK'' \to \cK} i_{\cG'' \to \cK''}.
\end{align*}
Then we may view
$$ {\mathfrak S}_{AA}: = \left( \cU, \begin{bmatrix} i_{\cG''} &
i_{\cG'} \end{bmatrix}; \cK, \cG'' \oplus \cG' \right)
$$
as a unitary scattering system with characteristic measure
\begin{equation}  \label{coupling-char-meas}
\sigma_{{\mathfrak S}_{AA}}( \cdot) =
\begin{bmatrix} i_{\cG''}^{*} \\ i_{\cG'}^{*} \end{bmatrix}
E_{\cU}(\cdot) \begin{bmatrix} i_{\cG''} & i_{\cG'} \end{bmatrix},
\end{equation}
characteristic function
$$ \widehat w_{{\mathfrak S}_{AA}}(\zeta) =
\begin{bmatrix} i_{\cG''}^{*} \\ i_{\cG'}^{*} \end{bmatrix}
[ (I - \overline{\zeta} \cU)^{-1} + (I - \zeta \cU^{*})^{-1} -I ]
\begin{bmatrix} i_{\cG''} & i_{\cG'} \end{bmatrix}
$$
and characteristic moment-sequence
$$
\{ w_{{\mathfrak S}_{AA}}(n)\}_{n \in {\mathbb Z}} \text{ with } w_{\mathfrak S_{AA}}(n) =
\begin{bmatrix} i_{\cG''}^{*} \\ i_{\cG'}^{*} \end{bmatrix} \cU^{*n}
\begin{bmatrix} i_{\cG''} & i_{\cG'} \end{bmatrix}.
$$
We note that the (1,2)-entry in the $n$-th characteristic moment
$w_{{\mathfrak S}_{AA},n}$ is closely associated with the
intertwiner $Y = i_{\cK''}^{*} i_{\cK}$ associated with the
AA-unitary coupling $(\cU, i_{\cK'}$, $i_{\cK''}; \cK)$:
$$
\left\langle \left[ w_{{\mathfrak S}_{AA}}(n) \right]_{12} g', g'' \right\rangle =
\left\langle  i_{\cG''}^{*} \cU^{*n} i_{\cG'} g', g''
\right\rangle= \left\langle Y \cU^{\prime * n} g', g''
\right\rangle = \left\langle Y g', \cU^{\prime \prime * n} g'' \right\rangle.
$$
Given any contractive intertwiner $Y \colon \cK' \to \cK''$ of
$(\cU', \cK')$ and $(\cU'', \cK'')$, we refer to the bilateral
sequence of operators $\{w_{Y}(n)\}_{n \in {\mathbb Z}} $ given by
\begin{equation}  \label{symbol}
w_{Y}(n) = [w_{{\mathfrak S}_{AA}}(n)]_{12} =
i_{\cG'' \to \cK''}^{*} \cU^{\prime \prime * n} Y i_{\cG' \to
\cK'} = i_{\cG'' \to \cK''}^{*} Y \cU^{\prime * n} i_{\cG' \to \cK'}
\end{equation}
as the {\em symbol} (associated with given subspaces $\cG' \subset \cK'$ and $\cG'' \subset \cK''$)
of the intertwiner $Y$. If $\cG'$ is $*$-cyclic for $\cU'$ and $\cG''$ is $*$-cyclic for $\cU''$,
then the subspaces
$$
\cK'_{0} = \overline{\operatorname{span}} \{ \cU^{\prime n} g'
\colon g' \in \cG' \text{ and } n \in {\mathbb Z}\}, \quad \cK''_{0} =
\overline{\operatorname{span}} \{ \cU^{\prime \prime n} g'' \colon  g'' \in \cG''
\text{ and } n \in {\mathbb Z}\}
$$
are equal to all of $\cK'$ and $\cK''$ respectively and the observation
$$
\langle Y \cU^{\prime n} g', \cU^{\prime \prime m} g'' \rangle = \langle w_Y(m-n) g', g'' \rangle_{\cG''}
$$
shows that there is a one-to-one correspondence between symbols $w_{Y}$ (with respect to the two $*$-cyclic
subspaces $\cG'$ and $\cG''$)  of $Y$ and the associated contractive intertwiners $Y$.
Moreover, we have the following characterization of which bilateral
$\cL(\cG', \cG'')$-valued sequences $w=\{w(n)\}_{n \in {\mathbb Z}}$ arise as
the symbol $w = w_{Y}$ for some contractive intertwiner $Y$.

\begin{theorem} \label{T:symbol-Fourier}
Suppose that $\{w(n)\}$ is the symbol \eqref{symbol} for a
contractive intertwiner $Y$ of the unitary operators $(\cU', \cK')$ and $(\cU'', \cK'')$
associated with $*$-cyclic subspaces $\cG' \subset \cK'$ and $\cG'' \subset \cK''$.
Then $\{w(n) \}_{n \in {\mathbb Z}}$ is the sequence of trigonometric moments
$$
w(n) = \int_{{\mathbb T}} t^{-n} \widehat w(dt)
$$
associated with an $\cL(\cG', \cG'')$-valued measure $\widehat w$
(equal to the  Fourier transform of $\{w(n)\}_{n \in {\mathbb Z}}$) such that
\begin{equation}  \label{sigma-pos}
\sigma: = \begin{bmatrix} \sigma'' & \widehat w \\ \widehat w^{*}
& \sigma' \end{bmatrix} \text{ is a positive } \cL(\cG', \cG'')-\text{valued measure}
\end{equation}
where we have set
\begin{equation}  \label{sigma-primes}
\sigma' = i_{\cG'}^{*} E_{\cU'}(\cdot)  i_{\cG'}, \quad
\sigma'' = i_{\cG''}^{*} E_{\cU''}(\cdot) i_{\cG''}.
\end{equation}

Conversely, the inverse Fourier transform
$$
w(n): = \int_{{\mathbb T}} t^{-n} \widehat w(dt)
$$
of any $\cL(\cG', \cG'')$-valued measure $\widehat w$ on ${\mathbb T}$
which in addition satisfies \eqref{sigma-pos} is the symbol (associated with the subspaces
$\cG'$ and $\cG''$) for a uniquely determined contractive intertwiner
$Y \colon \cK' \to \cK''$.
\end{theorem}

\begin{proof}
    The forward direction of the theorem is an immediate
 consequence of the results preceding the theorem.
				
For the converse we use the Hellinger model. Suppose that $\widehat
w$ is a vector measure such that
\eqref{sigma-pos} holds, where $\sigma'$ and $\sigma''$ are
 the positive measures given by \eqref{sigma-primes}.	Consider the
Hellinger
space $\cL^{\sigma}$ (see Section \ref{S:Hellinger})
with operator $\cU^{\sigma}$ being multiplication
by the independent variable $t$. We define the embeddings
 $$ i_{\cK'}: \cK'\to \cL^{\sigma} {\text\ and\ \ }
 i_{\cK''}: \cK''\to \cL^{\sigma}
$$
as follows: first
we map $\cK'$ onto $\cL^{\sigma '}$ and $\cK''$ onto $\cL^{\sigma ''}$
by Fourier representations
 $$ k'\mapsto i_{\cG'}^* E'(dt)k',\quad k'\mapsto i_{\cG''}^*
E''(dt)k''.
 $$
Then we embed $\cL^{\sigma '}$ and $\cL^{\sigma ''}$ into
$\cL^{\sigma}$ by
 $$\sigma' p' \mapsto \begin{bmatrix} \widehat w \\
\sigma' \end{bmatrix} p', \quad
 \sigma'' p'' \mapsto \begin{bmatrix} \sigma'' \\ \widehat
	 w^{*} \end{bmatrix} p''
$$
for arbitrary vector trigonometric polynomials $p'$, $p''$.
As it was shown in \cite{Kheifets-Dubov}
$$
 \operatorname{im} i_{{\mathcal K}'} + \operatorname{im}
 i_{{\mathcal K}''} \text{ is dense in } \cL^{\sigma}.
 $$
 Thus we get a minimal AA-unitary coupling of $\cU'$ and $\cU''$.
The symbol of the contractive intertwiner associated with this
AA-unitary coupling
is just the trigo\-nometric-moment sequence of the originally
 given measure $\widehat w$. Since the definition of symbol
 \eqref{symbol} can be rephrased as
 \begin{equation}\label{symbol-03}
	 \langle w_Y(n-m)\ g', g''\rangle = \langle  Y \cU^{\ \prime * n} g',
		 \cU^{\prime\prime * m} g''\rangle,
\end{equation}
and the sets $\{\cU^{\prime * n} g'\}$ and $\{ \cU^{\prime\prime * m}
g''\}$
 ($n$ running over ${\mathbb Z}$, $g'$ over $\cG'$ and $g''$
 over $\cG''$) have dense span in $\cK'$ and $\cK''$, respectively,
 we see that the correspondence between intertwiners $Y$ and
symbols $\{w(n)\}_{n \in
{\mathbb Z}}$ is one-to-one. Moreover, the correspondence between
symbols $\{w(n)\}_{n \in {\mathbb
 Z}}$ and their Fourier transforms $\widehat w$ is also one-to-one.
\end{proof}
				
 \begin{remark} {\em   The proof of Theorem \ref{T:symbol-Fourier}
in fact shows that, given an AA-unitary coupling $(\cU,
i_{\cK'}, i_{\cK''}; \cK)$ of two unitary operators
$(\cU'', \cK'')$ and $(\cU', \cK')$ together with a
choice of $*$-cyclic subspaces $\cG' \subset \cK'$ and
$\cG'' \subset \cK''$, then the AA-unitary coupling $(\cU, i_{\cK'},
i_{\cK''};
\cK)$ is unitarily equivalent to the Hellinger-model AA-unitary
coupling
$$
(\cU^{\sigma}, i_{\cK' \to \cL^{\sigma}}, i_{\cK'' \to
 \cL^{\sigma}}; \cL^{\sigma})
$$
where $\sigma = \sigma_{{\mathfrak S}_{AA}}$ is given
by\eqref{coupling-char-meas}.
 } \end{remark}

 \section{Liftings and unitary extensions of an isometry defined by
the problem data}
\label{S:solutions-extensions}
In this section we discuss, following \cite{AAK68, CS, Arocena1,
Arocena2, KKY, Kh-PhD,
Kh-AIP1, Kh-AIP2, Kheifets-Berkeley, Kheifets-IWOTA96, HarmAIP},
how solutions of the lifting problem can be identified with unitary
extensions of a
certain (partially defined) isometry $V$ which is constructed
directly from the
problem data. Introduce a Hilbert space $\cH_{0}$ by
	\begin{equation}\label{cH0} \cH_{0} = \operatorname{clos}
\begin{bmatrix} \cK''_{-} \\
\cK'_{+}\end{bmatrix}
 \end{equation}
with inner product given by
\begin{equation} \label{H-0-inner-product}
\left\langle \begin{bmatrix} k''_{-} \\ k'_{+} \end{bmatrix},
\begin{bmatrix} \ell''_{-} \\ \ell'_{+} \end{bmatrix}
 \right\rangle_{\cH_{0}} = \left\langle
\begin{bmatrix} I & X \\ X^{*} & I \end{bmatrix}
 \begin{bmatrix} k''_{-} \\ k'_{+} \end{bmatrix},
 \begin{bmatrix} \ell''_{-} \\ \ell'_{+} \end{bmatrix}
\right\rangle_{\cK''_{-} \oplus \cK'_{+}}.
 \end{equation}
Special subspaces of $\cH_{0}$ are of interest:
 \begin{equation}\label{defect-spaces}
	\cD := \operatorname{clos} \begin{bmatrix} \cK''_{-} \\
	 \cU' \cK'_{+} \end{bmatrix} \subset \cH_{0}, \qquad
	\cD_* := \operatorname{clos} \begin{bmatrix} \cU''^*\cK''_{-} \\
	 \cK'_{+} \end{bmatrix} \subset \cH_{0}.
\end{equation}
 Define an operator $V \colon \cD \to  \cD_*$ densely by
 \begin{equation}  \label{defV}
 V=\begin{bmatrix} {\mathcal U}^{\prime \prime *} & 0 \\
 0 & {\mathcal U}^{\prime *} \end{bmatrix}
\colon
\begin{bmatrix}  k''_{-} \\ \cU' k'_{+} \end{bmatrix}
\mapsto \begin{bmatrix}\cU^{\prime \prime *} k''_{-} \\ k'_{+}
\end{bmatrix}.
\end{equation}
 for  $k''_{-} \in \cK''_{-}$  and  $k'_{+} \in \cK'_{+}$.
 By the same computation as in \eqref{Ucoupled}, we see that $V$
extends to define an isometry from $\cD$ onto $\cD_*$.
Notice also that $V$ is completely determined by the problem data.

Let us say that the operator $\cU^*$ on $\cK$ is a {\em minimal
unitary extension} of $V$ if
$\cU^*$ is unitary on $\cK$ and there is an isometric embedding
 $i_{\cH_{0}} \colon \cH_{0} \to \cK$ of $\cH_{0}$ into $\cK$ such
that
\begin{equation}  \label{ext1}
    i_{\cH_{0}} V = \cU^* i_{\cH_{0}}|_{\cD}.
\end{equation}
 and
\begin{equation}  \label{minimal}
  \overline{\operatorname{\text{span}}}_{n \in {\mathbb Z}}
\cU^{n} \operatorname{im} i_{\cH_{0}} = \cK.
 \end{equation}
 In this situation note that we then also have
\begin{equation}  \label{ext2}
i_{\cH_{0}} V^* = \cU i_{\cH_{0}}|_{\cD_*}.
\end{equation}
 Two such minimal unitary extensions $(\cU^*, i_{\cH_{0}}; \cK)$
 and $(\widetilde\cU^*, \widetilde i_{\cH_{0}};\widetilde\cK)$
 are said to be {\em unitarily equivalent} if there is a unitary
operator $\tau \colon \cK \to
\widetilde\cK$ such that
$$
\tau  \cU^* = \widetilde\cU^* \tau, \qquad \tau i_{\cH_{0}} =
\widetilde i_{\cH_{0}}.
$$
Then the connection between minimal unitary extensions of $V$ and
lifts of $X$
 is given by the following.

\begin{theorem}  \label{T:AAcoupling-Unitext}
Suppose that we are given data for a Lifting Problem \ref{P:lift}
as above. Assume that the
subspaces $\cK'_+$ and $\cK''_-$ are $*$-cyclic.
 Let $V \colon \cD \to  \cD_{*}$ be the
 isometry given by \eqref{defV}. Then there exists a canonical
 one-to-one correspondence between equivalence classes of
 minimal AA-unitary couplings $({\mathcal U},
 i_{{\mathcal K}'}, i_{{\mathcal K}''};
{\mathcal K})$ of $(\cU', {\mathcal K}')$ and $(\cU'',
	 {\mathcal K}'')$ such that the contractive
intertwiner $Y= i_{\cK''}^{*} i_{\cK'}$
lifts $X$ on the one hand and equivalence classes of
minimal unitary extensions $(\cU^*, i_{\cH_{0}}; \cK)$
 of $V$ on the other.

 Specifically, if $(\cU, i_{\cK'}, i_{\cK''}; \cK)$ is a
minimal AA-unitary coupling of $(\cU'', \cK'')$ and $(\cU', \cK')$
with associated contractive intertwiner $Y =
i_{\cK''}^{*}i_{\cK''}$ lifting $X$, then the mapping
$$
i_{\cH_{0}}: = \begin{bmatrix} i_{\cK''} & i_{\cK'}
\end{bmatrix}\left|\begin{bmatrix} \cK''_- \\ \cK'_+
\end{bmatrix} \right.
$$
extends to an isometric embedding of $\cH_{0}$ into $\cK$ and
$$
\left(\cU^*, i_{\cH_{0}}; \cK  \right)
 $$
is a minimal unitary extension of $V$. Conversely, if
 $(\cU^*, i_{\cH_{0}}; \cK)$ is a minimal unitary extension
of $V$ and if we define isometric embedding operators
$i_{\cK'} \colon \cK' \to \cK$ and
$i_{\cK''} \colon \cK'' \to \cK$ via the wave operator construction
$$
i_{\cK'}   k' = \operatorname{s-lim}_{n \to \infty}
 \cU^{*n} i_{\cH_{0}} \begin{bmatrix} 0 \\ \cU^{\prime n}
 k' \end{bmatrix}, \qquad
i_{\cK''} k'' = \operatorname{s-lim}_{n \to \infty}
 \cU^{n} i_{\cH_{0}} \begin{bmatrix} \cU^{\prime \prime * n} k'' \\ 0
\end{bmatrix}
$$
defined initially only for
$$
k' \in \bigcup_{m \ge 0} \cU^{\prime * m} \cK'_{+}, \qquad
k'' \in \bigcup_{m \ge 0} \cU^{\prime \prime m} \cK''_{-},
$$
and then extended uniquely to all of $\cK'$ and
$\cK''$ respectively by
continuity, then the collection
$$
(\cU, i_{\cK'}, i_{\cK''}; \cK)
$$
is a minimal AA-unitary coupling of $(\cU'', \cK'')$ and
 $(\cU', \cK')$ with associated contractive intertwiner $Y =
i_{\cK''}^{*} i_{\cK'}$ lifting $X$.
\end{theorem}

\begin{proof}
Suppose that $({\mathcal U}, i_{{\mathcal K}'}, i_{{\mathcal K}''};
{\mathcal K})$
is a minimal AA-unitary coupling of $(\cU', {\mathcal K}')$ and
$(\cU'', {\mathcal K}'')$.
Define the map
 \begin{equation}\label{iH_0}
i_{\cH_{0}}  \colon \begin{bmatrix} k''_{-} \\ k'_{+}
\end{bmatrix} \mapsto i_{\cK''} k''_{-}+ i_{\cK'}k'_{+}.
\end{equation}
Since $i_{{\mathcal K}'}$ and  $i_{{\mathcal K}''}$ are
isometric then  $i_{\cH_{0}}$ is isometric if and only if
$$
\langle i_{\cK''} k''_{-},i_{\cK'}k'_{+}\rangle_{\cK}=
\langle k''_{-},k'_{+}\rangle_{\cH_{0}}:=
\langle k''_{-},X k'_{+}\rangle_{\cK''}.
$$
This in turn means that the intertwiner $Y= i_{\cK''}^{*} i_{\cK'}$
lifts $X$. Now,
\begin{align*}
i_{\cH_{0}} V \begin{bmatrix} k''_{-} \\ \cU^{\prime}
k'_{+}\end{bmatrix}
& =  i_{\cH_{0}} \begin{bmatrix}\cU''^* k''_{-} \\  k'_{+}
\end{bmatrix} =
i_{\cK''} \cU''^* k''_{-}+ i_{\cK'} k'_{+}  \\
 &  = \cU^*(i_{\cK''}  k''_{-}+ i_{\cK'}\cU^\prime  k'_{+})
 = \cU^* i_{\cH_{0}}  \begin{bmatrix} k''_{-} \\ \cU'
k'_{+}\end{bmatrix}.
\end{align*}
This in turn means that (\ref{ext1}) holds. Thus, $\cU^*$ on $\cK$
with embedding $i_{\cH_{0}}$ is a unitary
extension of $V$.
Since, by assumption, $\cK'_{+}$ is $*$-cyclic for
$\cU'$ on $\cK'$, $\cK''_{-}$ is $*$-cyclic for $\cU''$ on $\cK''$,
and since the AA-unitary coupling
$({\mathcal U}, i_{{\mathcal K}'}, i_{{\mathcal K}''};
{\mathcal K})$ is minimal, then
$$\overline{\operatorname{\text{span}}}_{n \in {\mathbb Z}}
\cU^{n} \operatorname{im} i_{\cH_{0}} = \cK .$$
Thus, $(\cU^*, i_{\cH_{0}}; \cK)$ is a minimal unitary
extension of $V$.

Conversely, suppose that $(\cU^*, i_{\cH_{0}}; \cK)$ is a
minimal unitary extension of $V$. We now apply the
construction of the wave operator from \cite{HarmAIP}
(Section 4), which simplifies significantly in our
situation due to the fact that $\cK'_+$ and $\cK''_-$ are embedded
isometrically into $\cH_0$.
For $k' \in \cU^{\prime * n}\cK'_{+}$
and $m \ge n$ note that
$\cU^{*m} i_{\cH_{0}} \sbm{ 0 \\ \cU^{\prime m} k' }$
is  well-defined (since $\cU^{\prime m}k' \in \cK'_{+}$ so
$\sbm{ 0 \\ \cU^{\prime m} k' } \in \cH_{0}$) and independent of
$m$ (since $\cU^*$ is an extension of $V$, see \eqref{ext2}).
Thus, the formula
\begin{equation}  \label{icK'}
i_{\cK'} \colon k' \mapsto \lim_{m \to \infty} \cU^{*m} i_{\cH_{0}}
\begin{bmatrix} 0 \\ \cU^{\prime m} k' \end{bmatrix}
\end{equation}
is a well-defined isometry from
$\displaystyle\bigcup_{n=0}^{\infty} \cU^{\prime
 * n} \cK'_{+}$ into $\cK$.  By assumption,
$\displaystyle\bigcup_{n=0}^{\infty} \cU^{\prime
* n} \cK'_{+}$ is dense in $\cK'$, and hence
 $i_{\cK'}$ extends uniquely by continuity to an isometry (still
denoted by $i_{\cK'}$) from $\cK'$ into $\cK$.  Similarly, the
formula
\begin{equation}  \label{icK''}
i_{\cK''} \colon k'' \mapsto \lim_{m \to \infty} \cU^{m}
 i_{\cH_{0}} \begin{bmatrix} \cU^{\prime \prime * m} k'' \\ 0
\end{bmatrix}
\end{equation}
gives rise to a well-defined isometry from $\cK''$ into $\cK$.
Definitions (\ref{icK'}) and (\ref{icK''}) imply that
$$
\cU i_{\cK'}= i_{\cK'} \cU'\quad {\text{and}}\quad \cU
i_{\cK''}= i_{\cK''} \cU''.
$$
We have thus arrived at an AA-unitary coupling $(\cU, i_{\cK''},
i_{\cK'}; \cK)$ of $(\cU'', \cK'')$ and $(\cU', \cK')$. To
check the minimality of the AA-unitary coupling note
that it follows
from (\ref{icK'}) and (\ref{icK''}) that
$$\operatorname{im} i_{\cK'}=
\overline{\operatorname{\text{span}}}_{n \in {\mathbb Z_{-}}}
\cU^{n}  i_{\cH_{0}}  \begin{bmatrix} 0 \\ {\cK'_{+}}
\end{bmatrix}
$$
and
$$\operatorname{im} i_{\cK''}=
 \overline{\operatorname{\text{span}}}_{n \in {\mathbb Z_{+}}}
\cU^{n}  i_{\cH_{0}}  \begin{bmatrix} {\cK''_{-}} \\ 0
\end{bmatrix}.
$$
Since $i_{\cH_{0}}\begin{bmatrix} 0 \\ {\cK'_{+}} \end{bmatrix}$
 is invariant for $\cU$ and $i_{\cH_{0}}  \begin{bmatrix}
{\cK''_{-}} \\ 0
\end{bmatrix}$ is invariant for $\cU^*$, we conclude that
$$
\operatorname{im} i_{\cK'}+\operatorname{im}
i_{\cK''}\supseteq \operatorname{\text{span}}_{n \in {\mathbb Z}}
 \cU^{n}  i_{\cH_{0}}\begin{bmatrix}
{\cK''_{-}} \\ {\cK'_{+}}
\end{bmatrix}.
$$
 Therefore,
 $$\overline{\operatorname{im} i_{\cK'}+\operatorname{im}
 i_{\cK''}}\supseteq
\overline{\operatorname{\text{span}}}_{n \in {\mathbb Z}}
\cU^{n}  i_{\cH_{0}}={\cK}.
$$
The last equality is due to minimality of the extension.
Thus,
$$\overline{\operatorname{im} i_{\cK'}+\operatorname{im}
i_{\cK''}}={\cK}
$$
and it follows that the AA-unitary coupling is minimal.
Moreover, $Y = i_{\cK''}^{*} i_{\cK'}$ lifts $X$ since
\begin{align*}
& \langle Y k'_{+}, k''_{-}\rangle_{\cK}
 = \langle i_{\cK'} k_{+}', i_{\cK''} k_{-}''
\rangle_{\cK}
 = \left\langle \begin{bmatrix} 0 \\ k'_{+}\end{bmatrix},
\begin{bmatrix} k''_{-} \\ 0 \end{bmatrix}
	 \right\rangle_{\cH_{0}}
= \langle X k'_{+}, k''_{-} \rangle_{K''_{-}}.
\end{align*}

The correspondences between AA-unitary couplings and
unitary extensions
defined above are mutually inverse. Moreover, it is
straightforward from the definitions of the equivalences that
under these correspondences equivalent AA-unitary couplings go to
equivalent unitary extensions and equivalent unitary
extensions go to equivalent AA-unitary couplings.
 This completes the proof of Theorem \ref{T:AAcoupling-Unitext}.
\end{proof}

\section{Structure of unitary extensions} \label{S:unitext}
In the previous section we obtained a correspondence between
contractive intertwining lifts $Y$ of $X$ and minimal unitary
extensions $\cU^{*}$ of a isometry $V$ on a Hilbert space $\cH_{0}$
with domain $\cD$ and codomain $\cD_{*}$.  In this section we
indicate how one can parametrize all such minimal unitary extensions.

We therefore suppose that we are given a Hilbert space $\cH_{0}$,
two subspaces $\cD$ and $\cD_{*}$ of $\cH_{0}$ and an operator $V$
which maps $\cD$ isometrically onto $\cD_{*}$:
$$
	V \colon \cD \to \cD_{*}.
$$
In this situation we say that $V$ is an {\em isometry} on $\cH_{0}$
with {\em
domain} $\cD$ and {\em codomain}  $\cD_{*}$.  We let $\Delta$ and
$\Delta_{*}$ be the respective orthogonal complements
 $$
	\Delta: = \cH_{0} \ominus \cD, \quad \Delta_{*}: = \cH_{0} \ominus
\cD_{*}.
$$
Let $\cU^{*}$ be a minimal unitary extension of $V$ to a Hilbert
space $\cK$, i.e., $\cU$ is unitary on the Hilbert space $\cK$,
$\cK$ contains the space $\cH_{0}$ as a subspace, the smallest
subspace of $\cK$ containing $\cH_{0}$ and reducing for $\cU$ is the
whole space $\cK$ and $\cU^{*}$ when restricted to $\cD \subset
\cH_{0} \subset \cK$ agrees with $V$:  $\cU|_{\cD} = V$.  We set
$\cH_{1}$ equal to $\cK \ominus \cH_{0}$ and write $\cK = \cH_{0}
\oplus \cH_{1}$.  We associate two unitary colligations $U_{1}$ and
$U_{0}$ to the extension $\cU^{*}$ as follows.  Since
$\cU^{*}|_{\cD} = V$ maps $\cD$ onto $\cD_{*}$ and since $\cU^{*}$
is unitary, necessarily $\cU^{*}$ must map $\cK \ominus \cD = \Delta
\oplus \cH_{1}$ onto $\cK \ominus \cD_{*} = \Delta_{*} \oplus
\cH_{1}$.  To define the unitary colligation $U_{1}$, we introduce a
second copy $\widetilde \Delta$ of $\Delta$ and a second copy
$\widetilde \Delta_{*}$ of $\Delta_{*}$ together with unitary
identification maps
\begin{equation}  \label{itildeDeltas}
i_{\widetilde \Delta} \colon \widetilde \Delta \to \Delta \subset
\cH_{0} \subset \cK, \quad
i_{\widetilde \Delta_{*}} \colon \widetilde \Delta_{*} \to
\Delta_{*} \subset \cH_{0} \subset \cK.
\end{equation}
We then define the colligation
\begin{equation}  \label{free-param-col1}
U_{1} : = \begin{bmatrix} A_{1} & B_{1} \\ C_{1} & D_{1}
\end{bmatrix} \colon \begin{bmatrix} \cH_{1} \\ \widetilde \Delta
\end{bmatrix} \to \begin{bmatrix} \cH_{1} \\ \widetilde \Delta_{*}
\end{bmatrix}
\end{equation}
by
\begin{equation}  \label{free-param-col}
 U_{1} = \begin{bmatrix} i_{\cH_{1}}^{*} \\ i_{\widetilde
\Delta_{*}}^{*} \end{bmatrix}  \cU^{*}  \begin{bmatrix} i_{\cH_{1}}
& i_{\widetilde \Delta} \end{bmatrix}
\end{equation}
where $i_{\cH_{1}} \colon \cH_{1} \to \cK = \cH_{0} \oplus \cH_{1}$
is the natural inclusion map.
We define a second colligation
$$
U_{0}  = \begin{bmatrix} A_{0} & B_{0} \\ C_{0} & D_{0}
\end{bmatrix} \colon \begin{bmatrix} \cH_{0} \\ \widetilde \Delta_{*}
\end{bmatrix}
\to \begin{bmatrix} \cH_{0} \\ \widetilde \Delta \end{bmatrix}
$$
 as follows.  Note that the space $\cH_{0}$ has two orthogonal
decompositions
$$
 \cH_{0} = \cD \oplus \Delta = \cD_{*} \oplus \Delta_{*}.
$$
 If we use the first orthogonal decomposition of $\cH_{0}$ on the
domain side and the second orthogonal decomposition of $\cH_{0}$ on
the range side, then we may define an operator $U_{0} \colon \cH_{0}
\oplus \widetilde \Delta_{*} \to \cH_{0} \oplus \widetilde \Delta$
via the $3 \times 3$-block matrix
\begin{equation}  \label{U0-1}
U_{0} = \begin{bmatrix} V & 0 & 0 \\ 0 & 0 & i_{\widetilde
\Delta_{*}} \\ 0 & i_{\widetilde \Delta}^{*} & 0 \end{bmatrix}
\colon \begin{bmatrix} \cD \\ \Delta \\ \widetilde \Delta_{*}
\end{bmatrix} \to \begin{bmatrix} \cD_{*} \\ \Delta_{*} \\ \widetilde
\Delta \end{bmatrix},
 \end{equation}
or, in colligation form,
\begin{equation}  \label{U0-2}
 U_{0} = \begin{bmatrix} A_{0} & B_{0} \\ C_{0} & 0 \end{bmatrix}
\colon \begin{bmatrix} \cH_{0} \\ \widetilde \Delta_{*} \end{bmatrix}
 \to \begin{bmatrix} \cH_{0} \\ \widetilde \Delta \end{bmatrix}
 \end{equation}
where
\begin{align}
&  A_{0}|_{\cD} = V, \quad A_{0}|_{\Delta} = 0,
\quad C_{0}|_{\cD} = 0, \quad
C_{0}|_{\Delta} = i_{\widetilde \Delta}^{*},
 \notag \\
& B_{0} = i_{\widetilde \Delta_{*}} \text{ with } \mathop{\text{im}}
B_{0} = \Delta_{*} \subset {\mathcal H}_{0}. \label{U0-3}
 \end{align}
We note that the colligation $U_{0}$ is defined by the problem data
(i.e., the isometry $V$ with given domain $\cD$ and codomain $\cD_{*}$
in the space $\cH_{0}$) and is independent of the choice of unitary
extension $\cU^{*}$.  As one sweeps all possible unitary extensions
$\cU^{*}$ of $V$, the associated colligation $U_{1}$ can be an
arbitrary colligation of the form \eqref{free-param-col1}, i.e., one
with input space $\widetilde \Delta$ and output space $\widetilde
\Delta_{*}$. Moreover, from the fact the colligation matrix $U_{0} =
\sbm{ A_{0} & B_{0} \\ C_{0} & 0 }$ has a zero for its $(2,2)$-entry,
we see from Theorem \ref{T:FBcon} that the feedback connection
$\cF_{\ell}(U_{0}, U_{1})$ is well-defined for {\em any} colligation
(in particular, for any unitary colligation) of the form
\eqref{free-param-col1}.
Also,  from the very definitions, we see that if
$U_{1}$ is constructed from the unitary extension $\cU^{*}$ as
indicated in \eqref{free-param-col}, then we recover $\cU^{*}$ from
 $U_{0}$ and $U_{1}$ as the feedback connection $\cU^{*} =
\cF_{\ell}(U_{0}, U_{1})$ given by \eqref{FBconnection}.  The
following result gives the converse.

\begin{theorem}\label{T:ext-struct}
The operator $\mathcal U^*$ on $\cK$ is a unitary extension of $V$ to
a Hilbert
space $\mathcal K$ if and only if, upon decomposing $\cK$ as
$\mathcal K = \mathcal H_0\oplus\mathcal H_1$, $\cU^{*}$ can be
written in the form
 $$
 \cU^{*} =\cF_{\ell}(U_{0}, U_{1})
$$
where $U_{0}$ is the universal unitary colligation determined
completely by the problem data as in \eqref{U0-3} and $U_{1}$ is a
free-parameter unitary colligation of the form
\eqref{free-param-col1}.
Moreover, $\cU^{*}$ is a {\em minimal unitary extension of $V$}, i.e.,
the smallest
reducing subspace for $\cU^{*}$ containing $\cH_{0}$ is the whole
space $\cK: = \cH_{0} \oplus \cH_{1}$,  {\em if and only if $U_{1}$
is a simple unitary colligation}, i.e., the smallest reducing
subspace for $A_{1}$ containing $\operatorname{im} B_{1} +
\operatorname{im} C_{1}^{*}$ is the whole space $\cH_{1}$.
\end{theorem}

\begin{proof}
We already showed that every unitary extension $\cU^{*}$ of $V$ has
the form $\cU^{*} = \cF_{\ell}(U_{0}, U_{1})$ where $\cU^{*}$
determines $U_{1}$ according to \eqref{free-param-col}.  Conversely
we now show that every lower feedback connection $\cF_{\ell}(U_{0}, U_{1})$
(with arbitrary unitary colligation $U_1$ of the form \eqref{free-param-col1})
produces a unitary extension $\cU^{*}$ of $V$.  From the formula \eqref{FB}
for the lower feedback connection applied to the case where $D_{0} = 0$,
we see that
\begin{equation}  \label{FBU0U1}
\cF_{\ell}(U_{0}, U_{1}) \begin{bmatrix} h_{0} \\ h_{1}
\end{bmatrix} = \begin{bmatrix} A_{0} + B_{0} D_{1} C_{0} & B_{0}
C_{1} \\ B_{1} C_{0}  & A_{1} \end{bmatrix} \begin{bmatrix} h_{0} \\
h_{1} \end{bmatrix}.
 \end{equation}
Specializing to the case where $h_{0} = d \in \cD \subset \cH_{0}$
and $h_{1} = 0$ and using the formulas \eqref{U0-3} for $A_{0}, B_{0},
C_{0}$, we see that
 $$
 \cF_{\ell}(U_{0}, U_{1}) \begin{bmatrix} d \\ 0 \end{bmatrix} =
\begin{bmatrix} A_{0} d \\ 0 \end{bmatrix} = \begin{bmatrix} Vd \\
 0 \end{bmatrix}
$$
and it follows that $\cF_{\ell}(U_{0}, U_{1})$ is an extension of
$V$.  Moreover, by plugging in the
explicit formulas \eqref{U0-3} for $A_{0}, B_{0},C_{0}$ into
\eqref{FBU0U1}, it is straightforward to verify that
we recover $U_{1}$ from $\cU^{*}: = \cF_{\ell}(U_{0}, U_{1})$ via
the formula \eqref{free-param-col} and that $\cU^{*}$ is unitary
exactly when $U_{1}$ is unitary.

It remains to check:
$\cU^{*}$ is a minimal extension of $V$ if  and only if $U_{1}$
is a simple unitary colligation.
Consider the minimal reducing subspace for $\cU^*$ that contains
$\cH_0$, then its orthogonal complement (which is a subspace of
$\cH_1$) also reduces $\cU^*$. From the definitions one sees that
the latter is the zero subspace
exactly when $U_1$ is simple.
\end{proof}

Since unitary extensions $\cU^{*}$ of $V$ are given via the lower
feedback connection $\cF_{\ell}(U_{0},U_{1})$, we may use the results
of Theorems \ref{T:FB-traj} and \ref{T:backwardFB-traj} to compute
positive and negative powers of $\cU^{*}$.  To simplify notation, we
let
\begin{align}
& {\mathbf S}^{+} =
 \begin{bmatrix} S_{0}^{+} &  S_{2}^{+} \\  S_{1}^{+} & S^{+}
\end{bmatrix} \colon
\begin{bmatrix} \cH_{0} \\ \ell_{\widetilde \Delta_{*}}({\mathbb
Z}_{+})\end{bmatrix}
 \to     \begin{bmatrix} \ell_{\cH_{0}}({\mathbb Z}_{+}) \\
\ell_{\widetilde \Delta}({\mathbb Z}_{+}) \end{bmatrix},  \notag \\
& {\mathbf S}^{-} = \begin{bmatrix} S_{0}^{-} &
 S_{1}^{-} \\ S_{2}^{-} & S^{-}  \end{bmatrix}\colon
 \begin{bmatrix}
\cH_{0} \\ \ell_{\widetilde \Delta}({\mathbb Z}_{-}) \end{bmatrix}
\to \begin{bmatrix} \ell_{\cH_{0}}({\mathbb Z}_{-}) \\
\ell_{\widetilde \Delta_{*}}({\mathbb Z}_{-}) \end{bmatrix}
\label{universal-aug-IO}
 \end{align}
be the forward and backward augmented input-output operators for the
universal colligation $U_{0}$ and we let
\begin{align}
&\mathbf\Omega^{+} = \begin{bmatrix} \Omega_{0}^{+} &
\Omega_{2}^{+} \\ \Omega_{1}^{+} & \Omega^{+} \end{bmatrix} \colon
\begin{bmatrix} \cH_{1} \\ \ell_{\widetilde \Delta}({\mathbb
 Z}_{+}) \end{bmatrix}
\to \begin{bmatrix} \ell_{\cH_{1}}({\mathbb Z}_{+}) \\
\ell_{\widetilde \Delta_{*}}({\mathbb Z}_{+})  \end{bmatrix}, \notag
\\
& \mathbf\Omega^{-} = \begin{bmatrix} \Omega_{0}^{-} &
\Omega_{1}^{-} \\ \Omega_{2}^{-} & \Omega^{-} \end{bmatrix} \colon
\begin{bmatrix} \cH_{1} \\ \ell_{\widetilde \Delta_{*}}({\mathbb
Z}_{-}) \end{bmatrix}
\to \begin{bmatrix} \ell_{\cH_{1}}({\mathbb Z}_{-}) \\
\ell_{\widetilde \Delta}({\mathbb Z}_{-}) \end{bmatrix}
\label{freeparam-aug-IO}
\end{align}
be the forward and backward augmented input-output operators for the
free-param\-eter unitary colligation $U_{1}$.    From the first rows
in
the formulas
 \eqref{Un-explicit} and \eqref{backward-Un-explicit}, we read off
that
$$  \Lambda_{\cH_{0}}(U) =  \begin{bmatrix} \Lambda_{\cH_{0},-}(U)
\\ \Lambda_{\cH_{0},+}(U) \end{bmatrix} \colon
\begin{bmatrix} \cH_{0} \\ \cH_{1}
\end{bmatrix} \to \begin{bmatrix} \ell_{\cH_{0}}({\mathbb Z}_{-}) \\
\ell_{\cH_{0}}({\mathbb Z}_{+}) \end{bmatrix}
$$
is given by
 \begin{equation}  \label{LambdaH0}
 \Lambda_{\cH_{0}}(U) =
\begin{bmatrix}
S_{0}^{-} + S_{1}^{-} (I - \Omega^{-}S^{-})^{-1} \Omega^{-} S_{2}^{-}
& & S_{1}^{-} (I - \Omega^{-} S^{-})^{-1} \Omega_{2}^{-} \\ \\
 S_{0}^{+} + S_{2}^{+} ( I - \Omega^{+} S^{+})^{-1} \Omega^{+}
S_{1}^{+}
 & & S_{2} ^{+} (I - \Omega^{+} S^{+})^{-1} \Omega_{1}^{+}
\end{bmatrix}.
\end{equation}
From the second rows in the formulas \eqref{Un-explicit} and
\eqref{backward-Un-explicit} we read off
that
$$
\Lambda_{\cH_{1}}(U) = \begin{bmatrix} \Lambda_{\cH_{1},-}(U) \\
 \Lambda_{\cH_{1},+}(U) \end{bmatrix} \colon \begin{bmatrix} \cH_{0}
\\ \cH_{1} \end{bmatrix} \to \begin{bmatrix}
\ell_{\cH_{1}}({\mathbb Z}_{-}) \\ \ell_{\cH_{1}}({\mathbb Z}_{+})
\end{bmatrix}
$$
is given by
\begin{equation}  \label{LambdaH1}
\Lambda_{\cH_{1}}(U) = \begin{bmatrix}
\Omega_{1}^{-} (I - S^{-} \Omega^{-})^{-1} S_{2}^{-}
& &\Omega_{0}^{-} + \Omega_{1}^{-}(I - S^{-} \Omega^{-})^{-1} S^{-}
\Omega_{2}^{-}
\\ \\
\Omega^{+}_{2} (I - S^{+} \Omega^{+})^{-1} S_{1}^{+}
& &\Omega^{+}_{0} + \Omega_{2}^{+} (I - S^{+} \Omega^{+})^{-1} S^{+}
\Omega_{1}^{+} \end{bmatrix}.
\end{equation}

\section{Parametrization of symbols of intertwiners}  \label{S:param}

Assume that we are given the data set
\begin{equation}  \label{Lifting-data}
(X,\quad  (\cU', \cK'), \quad  (\cU'', \cK''), \quad \cK'_{+} \subset
\cK',
\quad \cK''_{-} \subset \cK'')
\end{equation}
as in the Lifting Problem \ref{P:lift}.  If we are given $*$-cyclic
subspaces $\cG'$ and $\cG''$ for $\cU'$ and $\cU''$ respectively,
then the sets
$$
\{ \cU^{\prime n} g' \colon n \in {\mathbb Z},\, g' \in \cG'\}, \quad
\{ \cU^{\prime \prime n} g'' \colon n \in {\mathbb Z},\, g'' \in
\cG''\}
$$
have dense span in $\cK'$ and $\cK''$ respectively.  If $Y \in \cL(\cK',
 \cK'')$ is any operator satisfying the intertwining condition $Y
\cU' = \cU''
Y$, then the computation
$$
\langle Y \cU^{\prime n} g', \cU^{\prime \prime m} g''
\rangle_{\cK''} =
\langle \cU^{\prime \prime (n-m)} Y g', g'' \rangle_{\cK''} =
\langle i_{\cG''}^{*} \cU^{\prime \prime (n-m)} Y i_{\cG'}, g', g''
\rangle_{\cG''}
$$
shows that $Y$ is uniquely determined by its
{\em symbol}
$$
\{ Y_{n} = i_{\cG''}^{*} \cU^{\prime \prime * n} Y i_{\cG'}
= i_{\cG''}^{*} Y\cU^{\prime * n}  i_{\cG'} \}_{n
\in {\mathbb Z}}.
$$
Therefore, in principle, to describe all contractive intertwining
lifts $Y$ of a given $X \colon \cK'_{+} \to \cK''_{-}$, it suffices
to describe all the symbols $w_{Y}$ of contractive intertwining
lifts $Y$.  Such a description is given in the next result.

\begin{theorem} \label{T:symbol-param}
 Suppose that we are given data set \eqref{Lifting-data}
for a Lifting Problem \ref{P:lift}.  Let $U_{0} \colon \sbm{ \cH_{0}
\\ \widetilde
 \Delta_{*}} \to \sbm{ \cH_{0} \\ \widetilde \Delta}$ be the
universal unitary colligation constructed from the problem data
 as in  \eqref{U0-1} or \eqref{U0-2} and \eqref{U0-3} with
associated augmented input-output maps ${\mathbf S}^{+}$ and
${\mathbf S}^{-}$ as in \eqref{universal-aug-IO}.
For $U_{1}$ equal to a free-parameter unitary colligation of the
form \eqref{free-param-col1}, let ${\mathbf \Omega}^{+}$ and
${\mathbf \Omega}^{-}$ be the associated augmented input-output
maps as in \eqref{freeparam-aug-IO}.
 Finally let
$\cG'$ and $\cG''$ be a fixed pair of $\cU'$-$*$-cyclic and
 $\cU''$-$*$-cyclic subspaces of $\cK'$ and $\cK''$ and assume that
 $$
\cG' \subset \cK'_+, \quad \cG'' \subset \cK''_-.
$$
Let $i_{\cG'} \colon \cG' \to \cH_{0}$ be the inclusion map of
$\cG'$ into $\cH_{0}$ obtained as the inclusion of $\cG'$ in
$\cK'_{+}$ followed by the inclusion of $\cK'_{+}$ into $\cH_{0}$,
and,
similarly, let $i_{\cG''}$ be the inclusion of $\cG''$ in
$\cH_{0}$ obtained as the inclusion of $\cG''$ in $\cK''_{-}$
followed by the inclusion of $\cK''_{-}$ in $\cH_{0}$.
Let $\cI_{\cG''}^{*} = \operatorname{diag}_{n \in {\mathbb Z}}
 \{i_{\cG''}^{*} \}$ be the coordinate-wise projection of
$\ell_{\cH_{0}}({\mathbb Z})$ onto $\ell_{\cG''}({\mathbb Z})$.
Then the
$\cL(\cG', \cG'')$-valued bilateral sequence $w=\{w(n)\}_{n \in
{\mathbb Z}}$ is the symbol $w = w_{Y}$ (with respect to $\cG'$
and $\cG''$) for a contractive intertwining lift $Y$ of $X$ if and
only if there exists a free-parameter unitary colligation $U_{1}
= \sbm{A_{1} & B_{1} \\ C_{1} & D_{1}} \colon \cH_{1} \oplus
\widetilde \Delta \to \cH_{1} \oplus \widetilde \Delta_{*}$ so
 that $w$ (as an infinite column vector) has the form
 \begin{equation}  \label{symbol-param}
w_{Y} = \begin{bmatrix} {\mathcal I}_{\cG''}^{*} S_{0}^{-} i_{\cG'} \\
{\mathcal I}_{\cG''}^{*} [ S_{0}^{+} + S_{2}^{+} (I - \Omega^{+}
S^{+})^{-1} \Omega^{+} S_{1}^{+}] i_{\cG'} \end{bmatrix}.
\end{equation}
 \end{theorem}

\begin{remark}  \label{R:symbol-param}
 {\em There are various other formulations of the
	  formula \eqref{symbol-param} for the parametrization of lifting
  symbols.
If we define
 \begin{align}
 & s_{0}^{+} = {\mathcal I}_{\cG''}^{*} S_{0}^{+} i_{\cG'}, \quad
  s_{2}^{+} = {\mathcal I}_{\cG''}^{*} S_{2}^{+}, \quad s_{1}^{+} =
 S_{1}^{+} i_{\cG'}, \quad s^{+} = S^{+},
\label{def-Redheffer-param}
 \end{align}
 then the formula \eqref{symbol-param} assumes the form
\begin{equation}  \label{symbol-param'}
 w_{Y} = \begin{bmatrix} s_{0}^{-} \\
 s_{0}^{+} + s_{2}^{+} (I - \Omega^{+} s^{+})^{-1} \Omega^{+}
 s_{1}^{+} \end{bmatrix}
\end{equation}
 where the coefficient matrix $\sbm{s_{0}^{+} & s_{2}^{+} \\
 s_{1}^{+} & s^{+}}$ together with $s_{0}^{-}$ is completely
determined from the problem data while $\Omega^{+}$ is the
 input-output map for the free-parameter unitary colligation
 $U_{1}$.

If we consider $\ell_{\cG''}({\mathbb Z}_{+})$ as embedded in
 $\ell_{\cG''}({\mathbb Z})$ in the natural way, we may rewrite in
 turn the formula \eqref{symbol-param'} in the still more compact form
\begin{equation}  \label{symbol-param''}
 w_{Y} = s_{0} + s_{2} (I - \omega s)^{-1} \omega s_{1}
	\end{equation}
 where we define
 \begin{align}
& (s_{0}g')(m) = \begin{cases}
(s_{0}^{-} g')(m) & \text{for } m
 < 0 \\  (s_{0}^{+} g')(m) &\text{for } m \ge 0, \end{cases}
\notag \\
 & s_{2} = \iota s_{2}^{+} \text{ where } \iota \colon
\ell_{\cG''}({\mathbb Z}_{+}) \to \ell_{\cG''}({\mathbb Z})
\text{ is the natural inclusion,} \notag \\
& s=s^{+}, \quad \omega = \Omega^{+}, \quad s_{1} = s_{1}^{+}.
 \label{def-Redheffer}
	\end{align}
} \end{remark}

\begin{proof}
Theorem \ref{T:AAcoupling-Unitext} gives an
identification between contractive intertwining lifts $\cU^{*}$ and
unitary
extensions of the isometry $V \colon \cD \to \cD_{*}$ on $\cH_{0}$
constructed from the Lifting Problem data while Theorem
\ref{T:ext-struct} in turn gives a Redheffer-type parametrization of
all
such unitary extensions.  Moreover formula \eqref{LambdaH0} tells
us how to compute the powers of $\cU^{*} = \cF_{\ell}(U_{0},U_{1})$
followed by the projection to the subspace $\cH_{0}$.
By definition the symbol $w_{Y}$ is given by
$$
w_{Y}(n) = i_{\cG''}^{*} \cU^{*n} i_{\cG}.
$$
The parametrization result \eqref{symbol-param} now follows by
plugging into \eqref{LambdaH0} once we verify:
 \begin{equation}  \label{verify}
 i_{\cG''}^{*} S_{1}^{-}(I - \Omega^{-} S^{-})^{-1} S_{2}^{-}
i_{\cG'} = 0.
\end{equation}
We assert that in fact
 \begin{equation}\label{verify2}
 S_{2}^{-} i_{\cG'} = 0.
\end{equation}
Indeed, by definition $S_{2}^{-} h_{0} = \{ \widetilde
\delta_{*}(n)\}_{n \in {\mathbb Z}_{-}}$
means that $\widetilde \delta_{*}(n)$ is generated by the recursion
\begin{align*}
&  h_{0}(n) = A_{0}^{*} h_{0}(n+1), \quad h_{0}(0) = h_{0}, \\
& \widetilde \delta_{*}(n) = B_{0}^{*} h_{0}(n) \text{ for }
n = -1,-2, \dots.
\end{align*}
If we set $m = -n$, this means simply that
 $$
 \widetilde \delta_{*}(-m) = B_{0}^{*} A_{0}^{*m} h_{0}.
 $$
 For the case where $h_{0} = i_{\cG'} g' \in i_{\cK'_{+}} \cK'_{+}
\subset \cR$, we then have
$$
A_{0}^{*} h_{0} = V^{*} i_{\cG'} g' = i_{\cK'_{+}} \cU' g' \in \cD_{*}
$$
and, inductively, given that $A_{0}^{*m} h_{0} = i_{\cK'_{+}}
	k'_{+} \in i_{\cK'_{+}}  \cK'_{+} \subset \cD_{*}$, we have
$$
 A_{0}^{*m+1} h_{0} = A_{0}^{*} i_{\cK'_{+}}  k'_{+} = V^{*}
 i_{\cK'_{+}} k'_{+} = i_{\cK'_{+}} \cU' k'_{+} \in i_{\cK'_{+}}
\cK'_{+} \subset \cD_{*}.
$$
As $\cD_{*}$ is orthogonal to the final space $\Delta_{*}$ for the
isometry $i_{\widetilde \Delta_{*}}$, it follows that, for $m = 1, 2,
\dots $,
$$
B_{0}^{*} A_{0}^{*m} i_{\cG'} g' = i_{\widetilde \Delta_{*}}^{*}
A_{0}^{*m} i_{\cG'} g' = 0 \text{ for } m = 1,2, \dots
$$
from which \eqref{verify2} and \eqref{verify} follow.
\end{proof}

\begin{remark}  \label{R:param}
{\em We note that the value of the symbol $w_{Y}(n)$  is independent
of the choice of lift $Y$ for $n \le 0$.  Indeed, for $n \le
0$, $g' \in \cG'$ and $g'' \in \cG''$ (where as always we
are assuming that $\cG' \subset \cK'_{+}$ and $\cG'' \subset
\cK''_{-}$), we have
\begin{align*}
\langle  w_{Y}(n) g', g'' \rangle_{\cG''}  & =
\langle \cU^{\prime \prime *n} Y g', g'' \rangle_{\cK''_{-}}
	= \langle Y g', \cU^{\prime \prime n} g'' \rangle_{\cK''_{-}} \\
 & = \langle Y g', \cU^{\prime \prime n} g''
 \rangle_{\cK''_{-}}
	= \langle X g', \cU^{\prime \prime n} g'' \rangle_{\cK''_{-}}
\end{align*}
since $\cU^{\prime \prime n} g'' \in \cK''_{-}$ for $n \le 0$
whenever $g'' \in \cG'' \subset \cK''_{-}$.
							
Let us consider the special case where we take
 $$\cG': = \cK'_{+} \text{ and } \cG'' =\cE'': = \cK''_{-} \ominus
\cU^{\prime \prime *} \cK''_{-}.
$$
Then $\cE''$ is wandering for $\cU''$ and we may represent
$\cK''$ as the direct-sum decomposition
$$
 \cK'' = \cK''_{-} \oplus \bigoplus_{n=0}^{\infty}
 \cU^{\prime \prime n} (\cU'' \cE'').
$$
Then the Fourier representation operator
$$
\Phi'' \colon k'' \mapsto \{ i_{\cE''}^{*} \cU^{\prime \prime *
 n+1} k'' \}_{n \in {\mathbb Z}_{+}}
$$
 is a coisometry mapping $\cK''$ onto
$\ell^{2}_{\cE''}({\mathbb Z}_{+})$
with initial space equal to $\cK'' \ominus \cK''_{-} =
\bigoplus_{n=0}^{\infty} \cU^{\prime \prime n} (\cU^{\prime \prime}
\cE'')$.

If $Y$ is any lift, then $Y$ is uniquely determined by its
restriction $Y|_{\cK'_{+}}$ to $\cK'_{+}$ by the
wave-operator construction; thus, to solve the Lifting
Problem it suffices to describe all $Y|_{\cK'_{+}} \colon
\cK'_{+} \to \cK''$ rather than all lifts $Y\colon \cK' \to
\cK''$.  Moreover, if we use the Fourier representation
operator $\Phi''$ to identify $\cK'' \ominus \cK''_{-}$ with
$\ell^{2}_{\cE''}({\mathbb Z}_{+})$, then we have an
identification of $\cK''$ with $\sbm{ \cK''_{-} \\
\ell^{2}_{\cE''}({\mathbb Z}_{+})}$.  Then, with this
identification in place,  the restricted lift $ Y|_{\cK'_{+}}$
has a matrix representation of the form
$$
 Y|_{\cK'_{+}} = \begin{bmatrix} X \\ Y_{+} \end{bmatrix}
 \colon \cK'_{+} \to \begin{bmatrix}  \cK''_{-} \\
\ell^{2}_{\cE''}({\mathbb Z}_{+}) \end{bmatrix}.
$$
With this representation we lose no information concerning
 the lift $Y$ despite the fact that in general $\cG'':= \cE''
\subset \cK''_{-}$ may not be $*$-cyclic for $\cU''$.

 If we use the parametrization from \eqref{symbol-param},
 the operator $Y_{+}$ in turn has an infinite column-matrix
 representation given by
 $$
 Y_{+} = \begin{bmatrix} w_{Y}(1)  \\ w_{Y}(2) \\ \vdots \\
 w_{Y}(n+1) \\ \vdots \end{bmatrix} = J_{+}^{*} {\mathcal
I}^{*}_{\cE''}
 (S_{0}^{+} + S_{2}^{+}(I - \Omega^{+} S^{+})^{-1}\Omega^{+}
S_{1}^{+}) i_{\cK'_{+}}
$$
where $J_{+}$ is the shift operator on
$\ell^{2}_{\cE''}({\mathbb Z}_{+})$ and
where $\Omega^{+}$ is the input-output map associated with
the free-parameter unitary colligation $U_{1}$.  Finally, if
we apply the $Z$-transform
$$
\{e''(n)\}_{n \in {\mathbb Z}_{+}} \mapsto
\sum_{n=0}^{\infty} e''(n) \zeta^{n}
$$
to transform $\ell^{2}_{\cE''}({\mathbb Z}_{+})$ to the
Hardy space $H^{2}_{\cE''}$, then the operator $\widehat
Y_{+} \colon \cK'_{+} \to H^{2}_{\cE''}$ induced by $Y_{+}
 \colon \cK'_{+} \to \ell^{2}_{\cE''}({\mathbb Z}_{+})$ is
given by multiplication by the $\cL(\cK'_{+},\cE'')$-valued function
\begin{equation}  \label{FFparam}
\widehat Y_{+}(\zeta) =  \zeta^{-1}[\widehat s_{0}^{+}(\zeta) -
 \widehat s_{0}^{+}(0)] + \zeta^{-1} \widehat s_{2}^{+}(\zeta) (I -
 \omega(\zeta)\widehat  s^{+}(\zeta))^{-1} \omega(\zeta)
\widehat	s_{1}^{+}(\zeta) ] i_{\cK'_{+}}
\end{equation}
where
 \begin{align*}
     & \widehat s_{0}^{+}(\zeta) = i_{\cE''}^{*} \widehat
S_{0}^{+}(\zeta) i_{\cG'}, \,
& \widehat s_{2}(\zeta) = i_{\cE''}^{*}
\widehat S_{2}^{+}(\zeta), \,\quad
& \widehat s_{1}(\zeta) = \widehat S_{1}^{+}(\zeta)
i_{\cG'}, \,
& \widehat s(\zeta) = \widehat S^{+}(\zeta)
\end{align*}
 and where
$$
\begin{bmatrix}  \widehat S_{0}^{+}(\zeta) & \widehat
S^{+}_{2}(\zeta) \\
\widehat S_{1}^{+}(\zeta) & \widehat S^{+}(\zeta) \end{bmatrix} =
\begin{bmatrix} (I - \zeta A_{0})^{-1} & \zeta (I - \zeta
A_{0})^{-1} B_{0} \\ C_{0} (I - \zeta A_{0})^{-1} &
\zeta C_{0} (I - \zeta A)^{-1} B_{0} \end{bmatrix}
$$
is the frequency-domain version of the augmented input-output
map associated with the unitary colligation $U_{0}$
(and hence is completely determined from the problem data)  and where
$$
 \omega(\zeta) = D_{1} + \zeta C_{1} (I - \zeta A_{1})^{-1} B_{1}
$$
is the characteristic function of the free-parameter unitary
colligation $U_{1}$.  Let us use the notation $D_{X}$ for the defect
operator $D_{X}: = (I - X^{*}X)^{1/2}$ of $X$.
Further analysis shows that $\widehat
Y_{+}(\zeta)$ has a factorization $\widehat Y_{+}(\zeta) =
Y_{0+}(\zeta) D_{X}$ where the operator $\Gamma \colon D_{X}k'_{+}
\mapsto  Y_{0+}(\zeta) D_{X} k'_{+}$ defines a contraction
operator from $\cD_{X}: = \overline{\operatorname{Ran}} D_{X}$
(viewed as a space of constant functions) into $H^{2}_{\cE''}$.
Then we have the following form for the parametrization of the lifts:
\begin{equation}   \label{FFparam'}
Y \colon k'_{+} \mapsto \begin{bmatrix} X k'_{+} \\
  Y_{0+}(\zeta) D_{X} k'_{+} \end{bmatrix} \text{ where }
 \widehat Y_{+}(\zeta)  =  Y_{0+}(\zeta) D_{X} \text{ is given by
\eqref{FFparam}}.
\end{equation}
In Sections 6 and 7 of
Chapter XIV in \cite{FFbook}  or Theorem VI.5.1 in \cite{FFGK}, there
are derived formulas for a Redheffer coefficient matrix
\begin{equation}\label{1012}
  \Psi(\zeta) = \begin{bmatrix} \Psi_{11}(\zeta) & \Psi_{12}(\zeta) \\
  \Psi_{21}(\zeta) & \Psi_{22}(\zeta) \end{bmatrix}
\end{equation}
so that the function $Y_{0+}(\zeta)$ is expressed by the formula
\begin{equation}  \label{FFparam-final}
  Y_{0+}(\zeta) = \Psi_{11}(\zeta) + \Psi_{12}(\zeta) (I -
  \omega(\zeta) \Psi_{22}(\zeta))^{-1} \omega(\zeta) \Psi_{21}(\zeta).
\end{equation}
S.~ter Horst (private communication) has verified that, after some
changes of variable, the formula \eqref{FFparam-final} agrees with
\eqref{FFparam'}.

In this formulation of the Lifting Problem, the intertwining property
\eqref{Y=liftX} is
encoded  directly in terms of $\widehat Y_{0+}(\zeta)$ in the form
\begin{equation}  \label{intertwine1}
 Y_{0+}(\zeta) D_{X} \cU_{+} = i_{\cE''}^{*} X + \zeta
 Y_{0+}(\zeta) D_{X}.
\end{equation}
Here the range of $i_{\cE''}^{*}$ is the space $\cE''$ and $\cE''$
is identified as the subspace of constant functions in
$H^{2}_{\cE''}$.  Associated with the data of a Lifting Problem is an
underlying isometry $\rho \colon \cF \to \cE'' \oplus \cD_{X}$
where
$$
  \cF = \overline{\operatorname{Ran}} D_{X} \cU'_{+}
$$
and defined densely by
\begin{equation}   \label{underlyingisom}
    \rho D_{X} \cU'_{+} k'_{+} = \begin{bmatrix} \rho_{1} D_{X} \cU'
    k'_{+} \\ \rho_{2} D_{X} \cU' k'_{+} \end{bmatrix}: =
\begin{bmatrix} i_{\cE''}^{*} X k'_{+}  \\ D_{X} k'_{+} \end{bmatrix}.
\end{equation}
Then the form \eqref{intertwine1} of the intertwining condition can
be expressed directly in terms of the isometry $\rho$ in the form
\begin{equation}   \label{intertwine2}
    \rho_{1} + \zeta \cdot Y_{0+}(\zeta) \rho_{2} =
    Y_{0+}(\zeta)|_{\cF}.
\end{equation}
It is this formulation which has been extended to the context of the
Relaxed Commutant Lifting problem in \cite{FtHK1, FtHK2} and in
addition to a Redheffer parametrization for the set of all solutions
in \cite{tH1, tH2}. For the relaxed
problem, the underlying isometry $\rho$ given by
\eqref{underlyingisom} is in general only a contraction rather
than an isometry.  The Redheffer coefficient matrix \eqref{1012}
is a coisometry from $\cD_{X}\oplus H^{2}_{\cD_{\rho^{*}}}$ to $ H^{2}_{\cE''}\oplus
H^{2}_{\cD_{T'}}$ ($\Psi_{11}$ and $\Psi_{21}$ are multiplication operators).

}\end{remark}

\section{The universal extension}  \label{S:centralext}

Theorem \ref{T:symbol-param} obtained a parametrization of all
symbols of solutions of the lifting problem (and therefore also of
all lifts under the assumption that $\cG'$ and $\cG''$ are
$*$-cyclic) via a Redheffer linear-fractional map acting on a
free-parameter input-output map, or equivalently, a free-parameter
Schur-class function, acting between coefficient spaces $\widetilde
 \Delta$ and $\widetilde \Delta_{*}$.  As has been observed before in
 a variety of contexts (see e.g.~\cite{FFbook, FFGK}), a special role
is played by the lift associated with the free-parameter taken to be
equal to $0$ (the {\em central lift}). In this section we develop the
special properties of the
universal lift from the point of view of the ideas developed here.

The first step is to construct the simple unitary colligation having
characteristic function equal to the zero function.

\begin{theorem} \label{T:0charfunc}
    The essentially unique simple unitary colligation
 \begin{equation}  \label{col0charfunc}
	U_{10} = \begin{bmatrix} A_{10} & B_{10} \\ C_{10} & D_{10}
\end{bmatrix} \colon \begin{bmatrix} \cH_{10} \\ \widetilde \Delta
\end{bmatrix} \to \begin{bmatrix}\cH_{10} \\ \widetilde \Delta_{*}
\end{bmatrix}
\end{equation}
 having  characteristic function equal to the zero function
 $$
\omega_{10}(\lambda) = D_{0} + \lambda C_{0} (I - \lambda
 A_{0})^{-1} B_{0} \equiv 0 \colon \widetilde \Delta \to \widetilde
\Delta_{*}
$$
is constructed as follows:  take
\begin{align}
& \cH_{10} = \begin{bmatrix} \ell^{2}_{\widetilde
\Delta_{*}}({\mathbb Z}_{-}) \notag \\
 \ell^{2}_{\widetilde \Delta}({\mathbb Z}_{+}) \end{bmatrix}, \quad
 A_{10} = \begin{bmatrix} J_{-}^{*} & 0 \\ 0 & J_{+}^{*}
\end{bmatrix}, \\
 & B_{10} = \begin{bmatrix}  i_{\widetilde \Delta}^{(-1)} \\ 0
\end{bmatrix},
\quad
  C_{10} = \begin{bmatrix} 0 & i^{(0)*}_{\widetilde \Delta_{*}}
\end{bmatrix}, \quad D_{10} = 0,
 \label{col0charfunc'}
\end{align}
 where in general
 $J_{-}\colon ( \cdots, x(-2), x(-1)) \mapsto (\cdots, x(-3), x(-2))$
 is the compressed forward shift on $\ell^{2}_{\cX}({\mathbb Z}_{-})$,
 $J_{+} \colon (x(0), x(1), x(2), \dots) \mapsto (0, x(0),x(1),
\dots)$
 is the forward shift on $\ell^{2}_{\cX}({\mathbb Z}_{+})$ (with
coefficient space
 $\cX$ clear from the context), where $i_{\widetilde \Delta}^{(-1)}
\colon x \mapsto
(\dots, 0, x)$ is the natural injection of $\widetilde \Delta$ into
the subspace of elements of $\ell^{2}_{\widetilde \Delta}({\mathbb
Z}_{-})$ supported
on the singleton $\{-1\}$, and $i^{(0)}_{\widetilde \Delta_{*}}
\colon x \mapsto (x,0,0, \dots)$ is the natural injection of
$\widetilde \Delta_{*}$ into the subspace of elements of
$\ell^{2}_{\widetilde \Delta_{*}}({\mathbb Z}_{+})$ supported on the
singleton $\{0\}$.
\end{theorem}

\begin{proof}  This is a straightforward verification which we leave
to the reader.
\end{proof}

We have seen in Theorem \ref{T:ext-struct} that the operator
$\cU^{*}$ on $\cK$   extends the isometry $V$ on $\cH_{0}$ having
domain $ \cD \subset \cH_{0}$ and range $\cD_{*} \subset \cH_{0}$ if
and only if $\cU^{*}$ has a representation of the form
$$
 \cU^{*} = \cF_{\ell}(U_{0}, U_{1})
 $$
where $U_{0}$ is the universal colligation given by \eqref{U0-1} or
equivalently by \eqref{U0-2} and \eqref{U0-3} and where $U_{1} \colon
 \sbm{ \cH_{1} \\ \widetilde \Delta } \to \sbm{ \cH_{1} \\ \widetilde
\Delta_{*}}$ is a free-parameter unitary colligation, and, moreover,
 $\cU^{*}$ is a minimal unitary extension of $V$ if and only if
$U_{1}$ is a simple unitary colligation.  We now consider the
particular case where we take $U_{1}$ equal to the simple unitary
 colligation with zero characteristic function $U_{10}$ given as in
Theorem \ref{T:0charfunc}.  We refer to the resulting minimal unitary
extension
$\cU_{0}^{*} : = \cF_{\ell}(U_{0}, U_{10})$ as the {\em central
unitary extension}.  An application of the general formula
\eqref{FBU0U1} then gives
 \begin{align}
	 \cU_{0}^{*} & = \cF_{\ell}(U_{0}, U_{10}) =
\begin{bmatrix} A_{0} + B_{0} D_{10} C_{0}  & B_{0} C_{10}  \\
	 B_{10} C_{0} & A_{10} \end{bmatrix} \notag \\
& = \begin{bmatrix} A_{0} & 0 & i_{\widetilde \Delta_{*}}
 i_{\widetilde \Delta_{*}}^{(0) *} \\
i_{\widetilde \Delta}^{(-1)} i_{\widetilde \Delta}^{*} &
J_{-}^{*} & 0 \\ 0 & 0 & J_{+}^{*} \end{bmatrix} \text{ on }
\cK_{0}:= \begin{bmatrix} \cH_{0} \\ \ell^{2}_{\widetilde
 \Delta}({\mathbb Z}_{-}) \\ \ell^{2}_{\widetilde
\Delta_{*}}({\mathbb Z}_{+}) \end{bmatrix}.
 \label{centralext*}
 \end{align}
with adjoint given by
 \begin{equation}  \label{centralext}
 \cU_{0} = \begin{bmatrix} A_{0}^{*} & i_{\widetilde \Delta}
 i_{\widetilde \Delta}^{(-1)*} & 0 \\
 0 & J_{-} &  0 \\ i_{\widetilde \Delta_{*}}^{(0)} i_{\widetilde
\Delta_{*}}^{*} & 0 & J_{+} \end{bmatrix}.
\end{equation}

To analyze the finer structure of the universal extension
 $(\cU_{0}, \cK_{0})$ given by \eqref{centralext}, let us
 define embedding operators
 $$
 i_{\widetilde \Delta,0} \colon \widetilde \Delta \to
 \cK_{0},\quad i_{\widetilde \Delta_{*},0} \colon \widetilde
\Delta_{*} \to
\cK_{0}, \quad i_{\cK''_{-},0} \colon \cK''_{-} \to \cK_{0}, \quad
i_{\cK'_{+},0} \colon \cK'_{+} \to \cK_{0}
$$
 by
$$
 i_{\widetilde \Delta,0} = \begin{bmatrix} 0 \\ i_{\widetilde
\Delta}^{(-1)} \\ 0 \end{bmatrix}, \quad
 i_{\widetilde \Delta_{*},0} = \begin{bmatrix} 0 \\ 0 \\
 i_{\widetilde \Delta_{*}}^{(0)} \end{bmatrix}, \quad
i_{\cK''_{-},0} = \begin{bmatrix} i_{\cK''_{-} \to \cH_{0}} \\
 0 \\ 0 \end{bmatrix}, \quad i_{\cK'_{+},0} = \begin{bmatrix}
 i_{\cK'_{+} \to \cH_{0}} \\ 0 \\ 0 \end{bmatrix}.
$$
Then the collection
 \begin{equation}  \label{S0}
{\mathfrak S}_{0} = ( \cU_{0}, \quad \begin{bmatrix} i_{\widetilde
\Delta,0} &
i_{\widetilde \Delta_{*},0} &
i_{\cK''_{-},0}  & i_{\cK'_{+},0} \end{bmatrix}; \quad
\cK_{0}, \quad \widetilde \Delta \oplus \widetilde \Delta_{*} \oplus
\cK''_{-} \oplus \cK'_{+} )
\end{equation}
is a scattering system in the sense of Section \ref{S:systems}
(see \eqref{scatsys}).  Moreover, the operators $i_{\widetilde
\Delta,0}$,
$i_{\widetilde \Delta_{*},0}$, $i_{\cK''_{-},0}$, $i_{\cK'_{+},0}$
have unique respective isometric extensions
$$
\vec i_{\widetilde \Delta,0} \colon \ell^{2}_{\widetilde
\Delta}({\mathbb Z}) \to \cK_{0}, \quad
\vec i_{\widetilde \Delta_{*},0} \colon \ell^{2}_{\widetilde
\Delta_{*}}({\mathbb Z})
\to \cK_{0}, \quad
i_{\cK'',0}  \colon \cK'' \to \cK_{0}, \quad
i_{\cK',0} \colon \cK' \to \cK_{0}
$$
which satisfy the respective intertwining conditions
$$
\vec i_{\widetilde \Delta,0} J = \cU_{0} \vec i_{\widetilde
\Delta,0},	\quad
\vec i_{\widetilde \Delta_{*},0} J = \cU_{0} \vec i_{\widetilde
\Delta_{*},0}, \quad i_{\cK'',0} \cU'' = \cU_{0} i_{\cK'',0}, \quad
i_{\cK',0} \cU' = \cU_{0} i_{\cK',0}
$$
where here we set $J$ equal to the bilateral shift operator on any
space of the form $\ell^{2}_{\cX}({\mathbb Z})$ (the coefficient
space $\cX$ determined by the context).
Then the collection
\begin{equation}  \label{SAA0}
{\mathfrak S}_{AA,0} = (\cU_{0}, \quad  \vec i_{\widetilde\Delta,0},
\quad
\vec i_{\widetilde \Delta_{*},0}, \quad    i_{\cK'',0}, \quad
i_{\cK',0}; \quad \cK_{0} )
\end{equation}
can be viewed as a {\em four-fold  AA-unitary coupling} of the four
unitary operators
$$
(J, \ell^{2}_{\widetilde \Delta}({\mathbb Z})), \quad
(J, \ell^{2}_{\widetilde \Delta_{*}}({\mathbb Z})), \quad (\cU'',
\cK''),
\quad (\cU', \cK')
$$
which has certain additional properties.  The next theorem
identifies some of these additional properties.

\begin{theorem}  \label{T:scatsys-properties}
The scattering system \eqref{S0} and its
 extension to the four-fold AA-unitary coupling \eqref{SAA0}
associated with the universal extension \eqref{centralext}
$\cU_{0}$ for a  Lifting Problem have the following
properties:
\begin{enumerate}
	\item The density conditions
	\begin{equation}  \label{10.6}
		 \im i_{\cK',0} + \im i_{\cK'',0} \text{ is dense in
		 }\cK_{0},
 \end{equation}
 \begin{align}
&  \overline{\operatorname{span}} \{ \im i_{\cK',0},\, \im
  i_{\cK''_{-},0} \} = \cK_{0} \ominus \vec
  i_{\widetilde \Delta_{*},0} (\ell^{2}_{\widetilde
  \Delta_{*}}({\mathbb Z}_{+})),  \notag \\
 & \overline{\operatorname{span}} \{ \im i_{\cK'_{+},0},\, \im
     i_{\cK'',0} \}  = \cK_{0} \ominus \vec
  i_{\widetilde \Delta,0} (\ell^{2}_{\widetilde \Delta}({\mathbb
Z}_{-})),
 \label{MFM-ortho}
\end{align}
and
 \begin{align}
& \overline{\operatorname{span}} \{ \im i_{\cK',0},\, \im
i_{\cK''_{-},0} \} \cap
 \overline{\operatorname{span}} \{ \im i_{\cK'_{+},0},\, \im
i_{\cK'',0} \}
= \overline{\operatorname{span}} \{ \im i_{\cK'_{+},0}, \, \im
i_{\cK''_{-},0} \}
\label{MFM-ortho'}
 \end{align}
 hold.
 \item The orthogonality conditions
\begin{align}
 &  \vec i_{\widetilde \Delta_{*},0} (\ell^{2}_{\widetilde
 \Delta_{*}}({\mathbb Z}_{+})) \perp \im i_{\cK''_{-},0},
\quad
 \vec i_{\widetilde \Delta,0}(\ell^{2}_{\widetilde
 \Delta}({\mathbb Z}_{-})) \perp \im i_{\cK'_{+},0},
 \notag \\
	 & \vec i_{\widetilde \Delta,0}(\ell^{2}_{\widetilde
\Delta}({\mathbb Z}_{-})) \perp \vec i_{\widetilde
\Delta_{*},0}(\ell^{2}_{\widetilde \Delta_{*}}({\mathbb Z}_{+}))
\label{analytic}
\end{align}
and
\begin{equation}  \label{orthogonal}
    \vec i_{\widetilde \Delta,0}(\ell^{2}_{\widetilde
     \Delta}({\mathbb Z}_{-})) \perp \im i_{\cK''_{-},0}, \quad
\vec i_{\widetilde \Delta_{*},0}(\ell^{2}_{\widetilde
\Delta_{*}}({\mathbb Z}_{+})) \perp \im i_{\cK'_{+},0}
\end{equation}
 hold.
\item The subspace identities
 \begin{equation}  \label{10.7}
 \cU_{0}^{*}  \im i_{\widetilde \Delta_{*},0}
=i_{\cH_{0},0} \Delta_{*}, \quad
 \cU_{0} \im i_{\widetilde \Delta, 0}  =i_{\cH_{0}, 0} \Delta
\end{equation}
hold
\end{enumerate}
\end{theorem}

\begin{proof}
For simplicity let us use the bold notation
$$
\bcH_{0} = \im i_{\cH_{0},0}
$$
to indicate the subspace $\cH_{0}$ when viewed as a
subspace of $\cK_{0}$.

Property \eqref{10.6} is a consequence of the fact that the
universal free-parameter unitary colligation $U_{10}$ given by
 \eqref{col0charfunc} and \eqref{col0charfunc'} is simple and
 hence (by Theorem \ref{T:ext-struct}) the unitary operator
$\cU^{*} = \cF_{\ell}(U_{0}, U_{10})$ is a minimal unitary
extension of $V$.

To check conditions \eqref{MFM-ortho}, we use the orthogonal
decomposition of $\cK_{0}$ (see \eqref{centralext*})
\begin{equation}  \label{cK0decom}
 \cK_{0} = \vec i_{\widetilde
\Delta,0}(\ell^{2}_{\widetilde\Delta}({\mathbb Z}_{-}))
 \oplus \im  i_{\cH_{0},0} \oplus \vec i_{\widetilde
 \Delta_{*,0}}(\ell^{2}_{\widetilde \Delta_{*}}({\mathbb Z}_{+})).
 \end{equation}
From the formula for $\cU_{0}^{*}$ in \eqref{centralext*}, it
is easily checked that the smallest $\cU_{0}$-invariant
subspace $\bcH_{0+}$ containing $\bcH_{0}$ is
$$
\bcH_{0+} = \bcH_{0} \oplus \vec i_{\widetilde \Delta_{*},0}
(\ell^{2}_{\widetilde \Delta_{*}}({\mathbb Z}_{+}))
=  \cK_{0} \ominus \im \vec i_{\widetilde \Delta,
0}(\ell^{2}_{\widetilde \Delta}({\mathbb Z}_{-})).
$$
On the other hand, by the construction this smallest
$\cU_{0}$-invariant subspace can also be identified as
$\bcH_{0+} = \overline{\textup{span}}\{ \im i_{\cK'',0}, \, \im
i_{\cK'_{+},0}\}$.
Combining these observations gives the first part of
\eqref{MFM-ortho}.
The second part follows similarly by identifying the smallest
$\cU_{0}^{*}$-invariant
subspace of $\bcH_{0-}$ containing $\bcH_{0}$ as
$\bcH_{0-} = \cK_{0} \ominus \vec i_{\widetilde \Delta_{*},0}
(\ell^{2}_{\widetilde \Delta_{*}}({\mathbb Z}_{+}))$ on the one hand
and also as
$\bcH_{0-} = \overline{\textup{span}} \{ \im i_{\cK''_{-},0}, \, \im
i_{\cK',0} \}$ on the other.
To prove \eqref{MFM-ortho'}, note from the above discussion that
\begin{align*}
\bcH_{0+} & = \bcH_{0} \oplus \vec i_{\widetilde \Delta_{*},0}
(\ell^{2}_{\widetilde \Delta_{*}}({\mathbb Z}_{+})), \\
\bcH_{0-} & = \bcH_{0} \oplus \vec i_{\widetilde \Delta,0}
(\ell^{2}_{\widetilde \Delta}({\mathbb Z}_{-})).
 \end{align*}
As $ \vec i_{\widetilde \Delta_{*},0}
(\ell^{2}_{\widetilde \Delta_{*}}({\mathbb  Z}_{+}))$ and
$\vec i_{\widetilde \Delta,0} (\ell^{2}_{\widetilde \Delta}({\mathbb
Z}_{-}))$
are orthogonal to each other, it follows that
$\bcH_{0+} \cap \bcH_{0-} = \bcH_{0}$, i.e., \eqref{MFM-ortho'} holds.

The orthogonality conditions \eqref{analytic} and
\eqref{orthogonal} are clear from \eqref{cK0decom}.
In fact, the orthogonality conditions \eqref{orthogonal}  hold in the
stronger form
\begin{equation}  \label{orthogonal-strong}
 \im \vec i_{\widetilde \Delta_{*},0} \perp \im i_{\cK',0}, \quad
 \im \vec i_{\widetilde \Delta,0} \perp \im i_{\cK'',0}.
\end{equation}
To see this, note that
$\cK'_{+}$ is invariant under $\cU'$ and $i_{\cK'_{+},0} \cU'
= \cU_{0} i_{\cK'_{+},0}$ and hence $\im i_{\cK'_{+},0}$ is
invariant under $\cU_{0}$ and the first of the orthogonality
 conditions \eqref{orthogonal} implies that $\cU_{0}^{*n} \vec
i_{\widetilde \Delta_{*},0}(\ell^{2}_{\widetilde \Delta_{*}}({\mathbb
Z}_{+}))$
is orthogonal to $\im i_{\cK'_{+},0}$. As the subspace
$$
 \cup_{n=0}^{\infty} \cU_{0}^{*n} \vec i_{\widetilde \Delta_{*},0}
 (\ell^{2}_{\widetilde \Delta_{*}}({\mathbb Z}_{+})) =
\cup_{n=0}^{\infty} \vec i_{\widetilde \Delta,0} (J^{*n}
 \ell^{2}_{\widetilde \Delta_{*}}({\mathbb Z}_{+}))
 $$
is dense in $\im \vec i_{\widetilde \Delta_{*},0}$,
we conclude that $\im \vec i_{\widetilde \Delta_{*},0}$ is orthogonal
to
$\im i_{\cK'_{+},0}$.  As $\im \vec i_{\widetilde \Delta_{*},0}$ is
reducing for $\cU_{0}$, we conclude that in fact $\im \vec
i_{\widetilde \Delta_{*},0}$
is orthogonal to the smallest $\cU_{0}$-reducing subspace
containing $\im i_{\cK'_{+},0}$, i.e., to $\im i_{\cK',0}$, and the
first of conditions \eqref{orthogonal-strong} follows.  The second
orthogonality condition in \eqref{orthogonal-strong} follows similarly
 from the observation that $\im i_{\cK''_{-}}$ is invariant under
$\cU^{\prime \prime *}$.

The subspace identities \eqref{10.7} can be read off from the
definitions, in particular, the definition of $\cU_{0}$ \eqref{U0-1}.
\end{proof}

\begin{remark} \label{R:zeros}  {\em One can easily verify that
the orthogonality conditions \eqref{analytic} and
\eqref{orthogonal} can be expressed  in more succinct fashion as
\begin{align}
& i_{\cK''_{-},0}^{*} \cU_{0}^{*n} i_{\widetilde
 \Delta_{*},0} = 0 \text{ for } n \le 0, \label{zero-s2} \\
 & i_{\widetilde \Delta,0}^{*} \cU_{0}^{*n} i_{\cK'_{+},0} = 0
 \text{ for } n < 0, \label{zero-s1} \\
&i_{\widetilde \Delta,0}^{*} \cU_{0}^{*n} i_{\widetilde \Delta_{*},0}
= 0
\text{ for } n<0, \label{zero-s}
\end{align}
and
\begin{align}
 &i_{\widetilde \Delta,0}^{*} \cU_{0}^{*n} i_{\cK''_{-},0} = 0
\text{ for } n \ge 0, \label{zero13pre} \\
& i_{\widetilde \Delta_{*},0}^{*} \cU_{0}^{*n}
i_{\cK'_{+},0} = 0 \text{ for } n < 0.
\label{zero24pre}
\end{align}
Since actually the stronger relations \eqref{orthogonal-strong}
hold, the conditions \eqref{zero13pre} and \eqref{zero24pre} actually
hold for all $n \in {\mathbb Z}$:
\begin{align}
 &i_{\widetilde \Delta,0}^{*} \cU_{0}^{*n} i_{\cK''_{-},0} = 0
 \text{ for all } n \in {\mathbb Z}, \label{zero13} \\
 & i_{\widetilde \Delta_{*},0}^{*} \cU_{0}^{*n}
i_{\cK'_{+},0} = 0 \text{ for all } n \in {\mathbb Z},
	\label{zero24}
 \end{align}
 respectively. }\end{remark}

It is of interest that conversely the properties \eqref{10.6},
\eqref{MFM-ortho},
\eqref{analytic}, \eqref{orthogonal} and \eqref{10.7}  can be used to
characterize the universal extension $\cU_{0}$ associated with
a  Lifting Problem.  We present two versions of such a result.

\begin{theorem} \label{T:converse1}
Suppose that $(\cU'', \cK'')$ and $(\cU', \cK')$
are unitary operators and that $\cK''_{-} \subset \cK''$ and
$\cK'_{+} \subset \cK'$ are $*$-cyclic subspaces with
$\cK''_{-}$ and $\cK'_{+}$ invariant under $\cU^{\prime \prime*}$
and $\cU'$ respectively. Suppose also that $\widetilde \Delta$ and
$\widetilde \Delta_{*}$ are two coefficient Hilbert spaces and
that we are given a scattering system of the form
$$
{\mathfrak S}_{0} = (\cU_{0}, \quad \begin{bmatrix} i_{\widetilde
 \Delta,0} & i_{\widetilde \Delta_{*},0} & i_{\cK''_{-},0} &
i_{\cK'_{+},0} \end{bmatrix}; \quad \cK_{0}, \quad \widetilde \Delta
\oplus
\widetilde \Delta_{*} \oplus \cK''_{-} \oplus \cK'_{+})
$$
 where $ i_{\widetilde \Delta,0}$, $i_{\widetilde \Delta_{*},0}$,
 $i_{\cK''_{-},0}$ and $i_{\cK'_{+},0}$ are isometric embedding
operators of the respective spaces $\widetilde \Delta$, $\widetilde
\Delta_{*}$, $\cK''_{-}$
 and $\cK'_{+}$ into $\cK_{0}$.  We assume also that there is a
four-fold AA-unitary coupling
$$
{\mathfrak S}_{AA,0} = \left(\cU_{0},  \quad \vec
i_{\widetilde \Delta,0},\quad  \vec i_{\widetilde \Delta_{*},0},
 \quad
i_{\cK'',0}, \quad i_{\cK',0}; \quad \cK_{0} \right)
$$
of the four unitary operators
$$ (J, \ell^{2}_{\widetilde \Delta}({\mathbb Z})), \quad
(J, \ell^{2}_{\widetilde \Delta_{*}}({\mathbb Z})),  \quad
(\cU'',\cK''),
\quad (\cU', \cK')
$$
 which extends ${\mathfrak S}_{0}$ in the sense that
\begin{equation}  \label{SAA0extends}
i_{\widetilde \Delta,0} = \vec i_{\widetilde \Delta,0}
\circ i_{\widetilde \Delta}^{(-1)}, \quad
i_{\widetilde \Delta_{*},0} = \vec i_{\widetilde
\Delta_{*},0} \circ i_{\widetilde \Delta_{*}}^{(0)}, \quad
i_{\cK'',0}|_{\cK''_{-}} = i_{\cK''_{-},0}, \quad
i_{\cK',0}|_{\cK'_{+}} = i_{\cK'_{+}}.
\end{equation}
Define subspaces $\cH_{0}$, $\cD$, $\cD_{*}$, $\Delta$ and
$\Delta_{*}$ of $\cK_{0}$ according to
\begin{align}
&  \cH_{0}  = \operatorname{clos} \operatorname{ im } \begin{bmatrix}
i_{\cK''_{-},0} & i_{\cK'_{+},0} \end{bmatrix}, \quad
\cD  = \operatorname{clos} \left( \im i_{\cK''_{-},0} +
i_{\cK'_{+},0}(\cU' \cK'_{+}) \right), \notag \\
& \cD_{*}  = \operatorname{clos} \left(
i_{\cK''_{-},0}(\cU^{\prime \prime *} \cK''_{-}) + \im
 i_{\cK'_{+},0} \right), \quad
\Delta = \cH_{0} \ominus \cD, \quad \Delta_{*} =
	\cH_{0} \ominus \cD_{*}
 \label{subspaces}
\end{align}
	and let $i_{\cH_{0},0} \colon \cH_{0} \to \cK_{0}$ be the
isometric inclusion map.
Suppose also that either one of the the following additional
conditions holds:
\begin{enumerate}
	\item
Conditions \eqref{10.6}, \eqref{analytic}, \eqref{orthogonal}, and
\eqref{10.7} all hold, or
\item Conditions \eqref{MFM-ortho}, \eqref{MFM-ortho'},
\eqref{analytic}, \eqref{orthogonal}, and the following weaker form
of \eqref{10.7}
\begin{equation}  \label{10.7weak}
 \im i_{\widetilde \Delta_{*},0} \perp \im  \cU  i_{\widetilde
\Delta,0}
\end{equation}
	hold.
\end{enumerate}
Then ${\mathfrak  S}_{0}$ and ${\mathfrak S}_{AA,0}$ are equal to the
scattering
 system and the four-fold AA-unitary coupling associated with
 the universal extension $\cU_{0}^{*}$ from some Lifting Problem.
\end{theorem}

\begin{remark} \label{R:converse1}  {\em From the first version of
Theorem \ref{T:converse1}, we see that if \eqref{10.6},
\eqref{analytic}, \eqref{orthogonal} and \eqref{10.7}
 hold, then also \eqref{orthogonal-strong}, \eqref{MFM-ortho}
and \eqref{MFM-ortho'} hold.
From the second version, we see that if
\eqref{MFM-ortho}, \eqref{MFM-ortho'}, \eqref{analytic},
\eqref{orthogonal}
and \eqref{10.7weak} hold, then also \eqref{orthogonal-strong},
\eqref{10.6} and \eqref{10.7} hold.
}\end{remark}

In the proofs below it is convenient to use the bold notation
$$
\bcH_{0} = \im i_{\cH_{0}}, \, \bcD = i_{\cH_{0},0}(\cD), \,
\bcD_{*} = i_{\cH_{0},0}(\cD_{*}), \, \bDelta = i_{\cH_{0}}
(\Delta), \, \bDelta_{*} = i_{\cH_{0},0}(\Delta_{*})
$$
for the subspaces introduced in \eqref{subspaces} when  viewed as
subspaces
of $\cK_{0}$ rather than of $\cH_{0}$, as well as the additional
simplifications
$$
\bcG_{-} = \vec i_{\widetilde \Delta,0}(\ell^{2}_{\widetilde
\Delta}({\mathbb Z}_{-}))
\subset \cK_{0}, \quad \bcG_{*+} = \vec i_{\widetilde \Delta_{*},0}
(\ell^{2}_{\widetilde \Delta_{*}}({\mathbb Z}_{+}))
\subset \cK_{0}.
$$

\begin{proof}[Proof of version 1:]
The combined effect of the hypotheses \eqref{analytic} and
\eqref{orthogonal} is that the three subspaces
$\bcH_{0}, \bcG_{-},  \bcG_{*+}$ are pairwise orthogonal.  Therefore
the span of these
subspaces $\cK_{00}$ has an orthogonal decomposition
\begin{equation}  \label{cK00}
\cK_{00} = \bcH_{0} \oplus \bcG_{-} \oplus \bcG_{*+}.
\end{equation}
From the definitions we see that $\bcH_{0}$ has a two-fold orthogonal
decomposition as
\begin{equation}   \label{two-fold}
	\bcH_{0} = \bcD \oplus \bDelta = \bcD_{*} \oplus \bDelta_{*}.
\end{equation}
								
Due to the intertwinings
 $$
 \cU_{0} i_{\cK'',0} = i_{\cK'',0} \cU'', \quad \cU_{0} i_{\cK'} =
i_{\cK'} \cU'
$$
 one can see that
\begin{equation}   \label{cK00-1}
	 \cU_{0}^{*} ( \bcD) = \bcD_{*}
\end{equation}
and in fact we have the alternate characterizations of $\bcD$ and
$\bcD_{*}$:
\begin{equation}  \label{cDchar}
 \bcD= \{ h \in \bcH_{0} \colon \cU_{0}^{*} h \in \bcH_{0}\}, \quad
\bcD_{*} = \{ h_{*} \in \bcH_{0} \colon \cU_{0} h_{8} \in \bcH_{0} \}.
\end{equation}
From the hypothesis \eqref{10.7} we know that
\begin{equation}   \label{cK00-2}
 \cU_{0}^{*} \left( \im i_{\widetilde \Delta_{*},0}\right) =
\bDelta_{*}, \quad
\cU_{0}^{*} \bDelta = \im i_{\widetilde \Delta,0}
\end{equation}
and, from the intertwinings $\cU_{0}^{*} \vec
i_{\widetilde\Delta_{*},0}
= \vec i_{\widetilde \Delta_{*},0} J^{*}$ and
 $\cU_{0}^{*} \vec i_{\widetilde \Delta,0} =  \vec i_{\widetilde
\Delta,0}J^{*}$,
 we know that
\begin{equation}   \label{cK00-3}
\bcG_{-} = \cU_{0}^{*} \bcG_{-} \oplus \im \vec i_{\widetilde
\Delta,0}, \quad
\bcG_{*+} = \im \vec i_{\widetilde \Delta_{*},0} \oplus \cU_{0}
\bcG_{*+}.
\end{equation}
From the orthogonal decompositions \eqref{cK00} and
\eqref{two-fold} for $\cK_{0}$ and $\bcH_{0}$ combined with
\eqref{cK00-1}, \eqref{cK00-2} and \eqref{cK00-3} we see that
 $\cK_{00}$ is reducing for $\cU_{0}$.
From hypothesis \eqref{10.6} we conclude that in fact $\cK_{00}
=\cK_{0}$
and the decomposition \eqref{cK00} applies with $\cK_{0}$ in place of
$\cK_{00}$, i.e.,
we have
\begin{equation} \label{cK0decom'}
\cK_{0} = \bcG_{-} \oplus \bcH_{0} \oplus \bcG_{*+}.
\end{equation}
 From \eqref{cK00-1} we see that we may define an isometry $V$ on
 $\cH_{0}$ with domain $\cD$ and range $\cD_{*}$ by
$$
V d = d_{*} \text{ if } \cU_{0}^{*} i_{\cH_{0},0}d =
i_{\cH_{0},0}d_{*} \text{ for } d \in \cD,\, d_{*} \in \cD_{*}.
$$
It is now straightforward to check that necessarily $\cU_{0}^{*}$ is
a universal
extension of the isometry $V$. Furthermore, one can check that $V$
is the isometry constructed from the  Lifting Problem data
$$
X = i_{\cK''_{-},0}^{*} i_{\cK'_{+},0}, \quad (\cU'', \cK''),
 \quad (\cU', \cK'), \quad \cK'_{+} \subset \cK', \quad \cK''_{-}
\subset \cK''.
 $$
This completes the proof of the first version of Theorem
\ref{T:converse1}.
							     \end{proof}

\begin{proof}[Proof of version 2:]  Using hypotheses \eqref{analytic}
and
\eqref{orthogonal} as in the proof of version 1, we form the subspace
$\cK_{00}$ as in  \eqref{cK00}.
If we define $\widetilde \bcH_{0}$ by
$$
\widetilde \bcH_{0} : = \cK_{0} \ominus [ \bcG_{-} \oplus \bcG_{*+}],
$$
then by definition we have
\begin{equation}   \label{K0'}
 \cK_{0} = \widetilde \bcH_{0} \oplus \bcG_{-} \oplus \bcG_{*+}.
\end{equation}
For convenience let us introduce the temporary notation
	\begin{align}
 \bcH_{0-} & = \overline{\textup{span}} \{ \im i_{\cK',0},\, \im
i_{\cK''_{-},0} \},  \label{H0-} \\
\bcH_{0+} & = \overline{\textup{span}} \{ \im i_{\cK'_{+},0},\, \im
i_{\cK'',0} \}. \label{H0+}
\end{align}
Note that $\bcH_{0-}$ is the smallest $\cU_{0}^{*}$-invariant
subspace of $\cK_{0}$
containing $\bcH_{0}$  and that $\bcH_{0+}$ is the smallest
$\cU_0$-invariant subspace of $\cK_{0}$
containing $\bcH_{0}$.  Hypothesis \eqref{MFM-ortho} now takes the
form
\begin{align}
	\cK_{0} & = \bcH_{0-} \oplus \bcG_{*+}, \label{K0-} \\
 \cK_{0} & = \bcH_{0+} \oplus \bcG_{-}.
\label{K0+}
\end{align}
Combining \eqref{K0-} and \eqref{K0+} with \eqref{K0'} gives
$$
 \bcH_{0-}  = \widetilde \bcH_{0} \oplus \bcG_{-}, \quad
\bcH_{0+}  = \widetilde \bcH_{0} \oplus \bcG_{*+}.
$$
Since $\bcG_{-}$ is orthogonal to $\bcG_{*+}$ in $\cK_{0}$,  we then
get
$$
\bcH_{0-} \cap \bcH_{0+} = \widetilde \bcH_{0}.
$$
We may now invoke \eqref{MFM-ortho'} to conclude that $\bcH_{0}
=\widetilde \bcH_{0}$
and hence also $\cK_{00} = \cK_{0}$ and $\cK_{0}$ has the orthogonal
decomposition \eqref{cK00}
(with $\cK_{0}$ in place of $\cK_{00}$), i.e. \eqref{cK0decom'} holds.

From the third condition in \eqref{analytic} combined with
\eqref{10.7weak}, we see that in fact
 $$
\cU_{0}^{*} \im i_{\widetilde \Delta_{*},0} \perp \bcG_{-},
\quad
\cU \im i_{\widetilde \Delta,0} \perp \bcG_{*+}.
$$
But also
$$
\cU^{*} \im i_{\widetilde \Delta_{*},0} \perp \bcG_{*+}, \quad
\cU \im i_{\widetilde \Delta,0} \perp  \bcG_{-}.
$$
Hence we have
$$
\cU^{*} \im i_{\widetilde \Delta_{*},0} \perp \bcG_{-} \oplus
\bcG_{*+}, \quad
\cU \im i_{\widetilde \Delta,0} \perp
\bcG_{-} \oplus   \bcG_{*+}.
$$
From the orthogonal decomposition for $\cK_{0}$ \eqref{cK0decom}, we
conclude that
\begin{equation} \label{conclude1}
 \cU_{0}^{*} \im i_{\widetilde \Delta_{*},0} \subset \bcH_{0},
\quad \cU_{0} \im i_{\widetilde \Delta, 0} \subset \bcH_{0}.
\end{equation}

As in the proof of version (1), we see that $\bcD$ and $\bcD_{*}$ have
the characterizations \eqref{cDchar} and $\bcH_{0}$ has the two
orthogonal decompositions \eqref{two-fold}.  By combining these
observations with the decomposition \eqref{cK00} for $\cK_{0}$ and
the fact that $\cU_{0}$ is unitary on $\cK_{00}$, we see that the
containments \eqref{conclude1} actually force
$$
\cU_{0}^{*} \im i_{\widetilde \Delta_{*},0} = \bDelta_{*},
  \quad \cU_{0} \im i_{\widetilde \Delta, 0} = \bDelta,
$$
i.e., \eqref{10.7} holds.  From \eqref{10.7} combined with the
already proved decomposition \eqref{cK0decom'} for $\cK_{0}$ we see
that \eqref{10.6} holds as well.

It now follows from the already proved version (1) of Theorem
\ref{T:converse1}
that $(\cU_{0}, \cK_{0})$ is the central lift with
associated central scattering system ${\mathfrak S}_{0}$ and
four-fold AA-unitary coupling ${\mathfrak S}_{AA,0}$ coming from a
Lifting Problem as asserted.
	\end{proof}

Theorem \ref{T:symbol-param} gives a parametrization of the set of
all symbols $w_{Y}$ (with respect to a choice of two scale subspaces
$\cG'' \subset \cK''_{-}$ and $\cG' \subset \cK'_{+}$) via a
Redheffer-type linear-fractional-transformation \eqref{symbol-param'}
 $$
w_{Y} = s_{0} + s_{1} (I - \omega s)^{-1} \omega s_{2}
$$
where $\omega \colon \ell_{\widetilde \Delta}({\mathbb Z}_{+}) \to
\ell_{\widetilde \Delta_{*}}({\mathbb Z}_{+})$ is the input-output
map for a free-parameter unitary colligation, and where the Redheffer
coefficient matrix (see \eqref{def-Redheffer})
$$
\begin{bmatrix} s_{0} & s_{2} \\ s_{1} & s \end{bmatrix} \colon
\begin{bmatrix} \cG' \\ \ell_{\widetilde \Delta_{*}}({\mathbb
Z}_{+}) \end{bmatrix} \to \begin{bmatrix}
 \ell_{\cG''}({\mathbb Z}) \\ \ell_{\widetilde \Delta}({\mathbb Z})
\end{bmatrix}
$$
is completely determined from the  Lifting-Problem data.  Note that,
 if elements of the space $\ell_{\cG'' \oplus \widetilde
\Delta}({\mathbb Z})$ are expressed as infinite column vectors,
 the first column of the Redheffer coefficient matrix
$$
 \begin{bmatrix} s_{0} \\ s_{1} \end{bmatrix} \colon \cG' \to
\begin{bmatrix} \ell_{\cG''}({\mathbb Z}) \\
 \ell_{\widetilde \Delta}({\mathbb Z}) \end{bmatrix} =:
\ell_{\cG''
 \oplus \widetilde \Delta}({\mathbb Z})
$$
can be expressed naturally as a column matrix
$$
 \begin{bmatrix} s_{0} \\ s_{1} \end{bmatrix} = \operatorname{col}_{n
\in {\mathbb Z}} \begin{bmatrix} s_{0}(n) \\ s_{1}(n) \end{bmatrix}.
 $$
If we also view elements of $\ell_{\widetilde \Delta_{*}}({\mathbb
Z}_{+})$ as infinite column vectors, then the second column of the
Redheffer coefficient matrix
 $$
 \begin{bmatrix} s_{2} \\ s \end{bmatrix} \colon \ell_{\widetilde
	   \Delta_{*}}({\mathbb Z}_{+}) \to \ell_{\cG'' \oplus \widetilde
   \Delta}({\mathbb Z})
 $$
 can be expressed as an infinite matrix which has Toeplitz structure:
 \begin{equation}  \label{Toeplitz}
\begin{bmatrix} s_{2} \\ s \end{bmatrix}_{n,m} = \begin{bmatrix}
s_{2} \\ s \end{bmatrix}_{n-m,0} = : \begin{bmatrix}
s_{2}(n-m) \\  s(n-m) \end{bmatrix} \text{ for } n \in
{\mathbb Z}, \, m \in {\mathbb Z}_{+}.
 \end{equation}
Let us define the {\em Redheffer coefficient-matrix symbol} to be
simply the operator sequence
 \begin{equation}  \label{Redheffer-symbol}
\left\{ \begin{bmatrix} s_{0}(n) & s_{2}(n) \\ s_{1}(n) & s(n)
\end{bmatrix} \right\}_{n \in {\mathbb Z}}.
 \end{equation}
 The following result shows how the Redheffer coefficient-matrix
symbol can be
 expressed directly in terms of the universal extension $\cU_{0}$.  To
this end we introduce the notation
$$
 i_{\cG',0} \colon \cG' \to \cK_{0}, \quad i_{\cG'',0} \colon \cG''
 \to \cK_{0}
 $$
for the inclusion of $\cG'$ in $\cK_{0}$ obtained as the composition
$i_{\cG',0} = i_{\cH_{0},0} i_{\cG'}$ of the inclusion of $\cG'$
in $\cH_{0}$ followed by the inclusion of $\cH_{0}$ in $\cK_{0}$,
 and similarly $i_{\cG'',0} = i_{\cH_{0},0} i_{\cG''}$.

\begin{theorem}  \label{T:Redheffer-central}
 The Redheffer coefficient-matrix symbol  $\left\{\sbm{ s_{0}(n) &
s_{2}(n) \\
s_{1}(n) & s(n) } \right\}_{n \in {\mathbb Z}}$
for a  Lifting Problem can be
 recovered directly from the central extension $\cU_{0}$ (see
\eqref{U0-1} or
\eqref{U0-2} and \eqref{U0-3}) according to the formula
\begin{equation}  \label{Redheffer-central}
 \begin{bmatrix}
s_{0}(n) &  s_{2}(n) \\ s_{1}(n) &  s(n)
 \end{bmatrix} = \begin{bmatrix} i_{\cG'',0}^{*} \\
 i_{\widetilde \Delta,0}^{*} \cU_{0}^{*} \end{bmatrix} \cU_{0}^{*n}
	 \begin{bmatrix} i_{\cG',0} & i_{\widetilde
  \Delta_{*},0} \end{bmatrix}.
 \end{equation}
 Moreover,
 \begin{equation}  \label{zeros}
 s_{2}(m)=0 \text{ for } m \le 0, \quad   s_{1}(m) = 0
  \text{ and } s(m)
= 0 \text{ for } m < 0,
\end{equation}
and also
\begin{equation}  \label{zeros'}
s(0) = 0.
\end{equation}
\end{theorem}

\begin{proof}
We first check that the formula \eqref{Redheffer-central} is correct
for $n< 0$.  By using the definitions \eqref{def-Redheffer-param}
to unravel formula \eqref{def-Redheffer}, we read off
 that, for $n < 0$,
 $$
 s_{0}(n) = i_{\cG'',0}^{*}\cU^{*n}i_{\cG',0}, \quad
 s_{1}(n) = 0, \quad s_{2}(n) = 0, \quad  s(n)
 = 0.
$$
The first formula matches with the upper left corner of
\eqref{Redheffer-central} for $n<0$. As $\cG'' \subset \cK''_{-}$
and $\cG' \subset \cK'_{+}$ by assumption, the other three blocks
match up for $n<0$ as a consequence of the identities
\eqref{zero-s2}, \eqref{zero-s1} and \eqref{zero-s}.

 We next verify \eqref{Redheffer-central} for $n \ge 0$.  From the
Toeplitz structure \eqref{Toeplitz} we see that the validity of
\eqref{Redheffer-central} for $n \ge 0$ is equivalent to showing
that
\begin{equation}   \label{toshow1}
 \begin{bmatrix} s_{0}^{+} & s_{2}^{+} \\ s_{1}^{+} & s^{+}
\end{bmatrix} \colon \begin{bmatrix} g' \\ i_{\widetilde
 \Delta_{*},0} \widetilde \delta_{*} \end{bmatrix} \mapsto
 \left\{ \begin{bmatrix} i^{*}_{\cG'',0} \\ i^{*}_{\widetilde
\Delta ,0} \cU_{0}^{*n} \end{bmatrix} \cU_{0}^{*}
\begin{bmatrix} i_{\cG',0} & i_{\widetilde \Delta_{*},0} \end{bmatrix}
\begin{bmatrix} g' \\ i_{\widetilde
\Delta_{*},0} \widetilde \delta_{*} \end{bmatrix}
 \right\}_{n \in {\mathbb Z}_{+}}.
\end{equation}
From the definition \eqref{def-Redheffer-param} we have
$$
\begin{bmatrix} s_{0}^{+} & s_{2}^{+} \\ s_{1}^{+} & s^{+}
\end{bmatrix}
	= \begin{bmatrix} \cI_{\cG''} & 0 \\ 0 & I \end{bmatrix}
\begin{bmatrix} S_{0}^{+} & S_{2}^{+} \\ S_{1}^{+} & S^{+}
\end{bmatrix}
\begin{bmatrix} i_{\cG'} & 0 \\ 0 & I \end{bmatrix}.
$$
Combining this with \eqref{toshow1}, we see that it suffices to
show that
 \begin{equation}  \label{toshow2}
\begin{bmatrix} S_{0}^{+} & S_{2}^{+}i_{\widetilde
 \Delta_{*}^{(0)}} \\ S_{1}^{+} & S^{+}
i_{\widetilde \Delta_{*}^{(0)}} \end{bmatrix} \colon
 \begin{bmatrix} h_{0} \\ \widetilde \delta_{*} \end{bmatrix}
\mapsto \left\{
 \begin{bmatrix} i_{\cH_{0},0}^{*} \\ i_{\widetilde
\Delta,0}^{*} \cU_{0}^{*} \end{bmatrix} \cU_{0}^{*n}
 \begin{bmatrix} i_{\cG',0} & i_{\widetilde
 \Delta_{*},0} \end{bmatrix}
 \begin{bmatrix} h_{0} \\ \widetilde \delta_{*} \end{bmatrix}
 \right\}_{n \in {\mathbb Z}_{+}}.
 \end{equation}
 From the definition of
 $\sbm{S_{0}^{+} & S_{2}^{+} \\ S_{1}^{+} & S^{+}}$ as the
forward-time augmented input-output map for the unitary colligation
 $U_{0}$, we know that
$$
 \begin{bmatrix} S_{0}^{+} & S_{2}^{+} i_{\widetilde \Delta_{*}^{(0)}}
 \\ S_{1}^{+} & S^{+} i_{\widetilde \Delta_{*}^{(0)}}
\end{bmatrix} \begin{bmatrix} h_{0} \\ \widetilde \delta_{*}
\end{bmatrix} = \left\{ \begin{bmatrix} h_{0}(n) \\ \widetilde
 \delta(n) \end{bmatrix} \right\}_{n \in {\mathbb Z}_{+}}
 $$
 means that
 \begin{equation}  \label{know1}
 h_{0}(0) = h_{0},  \quad
 \begin{bmatrix} h_{0}(1) \\ \widetilde \delta(0) \end{bmatrix} =
 U_{0} \begin{bmatrix} h_{0} \\ \widetilde \delta_{*}
\end{bmatrix}, \quad
 \begin{bmatrix} h_{0}(n+1) \\ \widetilde \delta(n) \end{bmatrix} =
	 U_{0} \begin{bmatrix} h_{0}(n) \\ 0 \end{bmatrix} \text{ for } n
	 > 0.
 \end{equation}
Given $h_{0} \in \cH_{0}$ and $\widetilde \delta_{*} \in \widetilde
\Delta_{*}$,
 let us define $h_{0}(n) \in \cH_{0}$ and $\delta(n) \in \widetilde
 \Delta$ by
 \begin{equation}   \label{define1}
 \begin{bmatrix} h_{0}(n) \\ \delta(n) \end{bmatrix} =
 \begin{bmatrix} i_{\cH_{0},0}^{*} \\ i_{\widetilde \Delta,0}^{*}
\cU_{0}^{*} \end{bmatrix} \cU_{0}^{*n} k
\text{ where } k = i_{\cH_{0},0} h_{0} + i_{\widetilde \Delta_{*},0}
\widetilde \delta_{*}.
\end{equation}
Then \eqref{Redheffer-central} follows if we can show that
 $\{h_{0}(n), \widetilde \delta(n)\}_{n \in {\mathbb Z}_{+}}$ so
defined satisfies \eqref{know1}.

 The first  equality in \eqref{know1} is immediate from the fact that
 $\operatorname{im} i_{\cH_{0,0}}$ is orthogonal to $\operatorname{im}
i_{\widetilde \Delta_{*},0}$ in $\cK_{0}$.

The second equality in \eqref{know1} is an easy consequence of the
general identity
 \begin{equation}  \label{U0-calU0}
 U_{0} = \begin{bmatrix} i_{\cH_{0},0}^{*} \\ i_{\widetilde
\Delta,0}^{*} \end{bmatrix} \cU_{0}^{*} \begin{bmatrix}
i_{\cH_{0}, 0} & i_{\widetilde \Delta_{*},0} \end{bmatrix},
\end{equation}
connecting $U_{0}$ and $\cU_{0}$.  This identity in turn is
an easy consequence of  the formula \eqref{centralext*} for
$\cU_{0}^{*}$ and is
an analogue of the formula \eqref{free-param-col} connecting
$U_{1}$ and $\cU$ in a more general context.

The third equality in \eqref{know1} can also be seen as a
 consequence of \eqref{U0-calU0} as follows.  For $n>0$ we compute
 \begin{align}
U_{0} \begin{bmatrix}  h_{0}(n) \\ 0 \end{bmatrix}
 & = U_{0} \begin{bmatrix} i_{\cH_{0},0}^{*} \cU_{0}^{*n} k \\ 0
\end{bmatrix} \notag \\
 & = \begin{bmatrix} i_{\cH_{0},0}^{*} \\ i_{\widetilde \Delta,0}^{*}
\end{bmatrix}
 \cU_{0}^{*} i_{\cH_{0},0} i_{\cH_{0},0}^{*} \cU_{0}^{*n} k
\text{ (by \eqref{U0-calU0}) }\notag \\
 & = \begin{bmatrix} i_{\cH_{0},0}^{*} \\ i_{\widetilde \Delta,0}^{*}
\end{bmatrix}
\cU_{0}^{*} P_{\cH_{0}} \cU_{0}^{*n} k
\label{compute1}
\end{align}
where $k = i_{\cH_{0},0} h_{0} + i_{\widetilde \Delta_{*},0}
 \widetilde \delta_{*}$.
 and where $P_{\cH_{0}}$ is the orthogonal projection of $\cK_{0}$
 onto $\operatorname{im} i_{\cH_{0},0}$.  To verify the third
equality in \eqref{know1} it remains only to show that the
projection $P_{\cH_{0}}$ is removable in the last expression in
 \eqref{compute1}.  To this end, use the orthogonal decomposition
$$
 \cK_{0} = \vec i_{\widetilde \Delta,0} (\ell^{2}_{\widetilde
 \Delta}({\mathbb Z}_{-})) \oplus \operatorname{im} i_{\cH_{0},0}
\oplus \vec i_{\widetilde \Delta_{*},0}(\ell^{2}_{\widetilde
\Delta_{*},0}({\mathbb Z}_{+})).
$$
 Note that $\cU_{0}^{*n} k \perp \vec i_{\widetilde
 \Delta_{*},0}(\ell^{2}_{\widetilde \Delta_{*}}({\mathbb Z}_{+}))$,
i.e.,
 \begin{equation}  \label{condition1}
\cU_{0}^{*n} k \in \vec i_{\widetilde \Delta,0}(\ell^{2}_{\widetilde
\Delta}
({\mathbb Z}_{-}))
  \oplus \operatorname{im} i_{\cH_{0},0}
\end{equation}
  for $n > 0$. Moreover it is easily checked
\begin{equation}  \label{condition2}
\cU_{0}^{*} \vec i_{\widetilde \Delta,0}(\ell^{2}_{\widetilde
\Delta}({\mathbb Z}_{-})) \perp \operatorname{im}
	i_{\widetilde \Delta,0} \oplus \operatorname{im} i_{\cH_{0},0}.
 \end{equation}
From conditions \eqref{condition1} and \eqref{condition2} we see
 that indeed the projection $P_{\cH_{0}}$ is removable in
 \eqref{compute1} and the third equation in \eqref{know1} follows as
  required.

 Now that the validity of \eqref{Redheffer-central} is
 established, we see that $s_{2}(0) = 0$ as a consequence of
 \eqref{zero-s2} for the case $n=0$.

It remains to verify that $s(0) =i_{\widetilde \Delta,0}^{*}
 \cU_{0} i_{\widetilde \Delta_{*},0} = 0$, or equivalently,
 $$
\im i_{\widetilde \Delta,0} \perp \cU_{0}^{*} \im i_{\widetilde
 \Delta_{*},0}.
 $$
This can be seen as a direct consequence of the definition of
  $\cU_{0}^{*}$ in \eqref{centralext*}.
 \end{proof}

 \begin{remark} \label{R:LPscat} {\em In the proof of Theorem
 \ref{T:Redheffer-central} it is shown that, given that
 $\cU_{0}$ and $U_{0}$ are related as in \eqref{U0-calU0}, then
\eqref{define1} implies \eqref{know1}.  This observation can be
seen as a special case of the following general result.}
 Given a unitary operator $\cU$ on
  $\cK$, a unitary colligation $U$ of
 the form
 $$
 U = \begin{bmatrix} A & B \\ C & D \end{bmatrix} \colon
 \begin{bmatrix} \cH \\ \cE \end{bmatrix} \to \begin{bmatrix}
  \cH \\ \cE_{*} \end{bmatrix}.
$$
such that
\begin{equation}  \label{U-calU}
U = \begin{bmatrix} i_{\cH}^{*} \\ i_{\cE_{*}}^{*}
\end{bmatrix} \cU^{*} \begin{bmatrix} i_{\cH} & i_{\cE} \end{bmatrix}.
 \end{equation}
 where
 $$
 i_{\cH} \colon \cH \to \cK, \quad i_{\cE} \colon \cE \to \cK, \quad
 i_{\cE_{*}} \colon \cE_{*} \to \cK
 $$
 are isometric embedding operators with
\begin{align*}
 & \im i_{\cH} \perp \im i_{\cE} \text{ so } \begin{bmatrix}
  i_{\cH} & i_{\cE} \end{bmatrix}  \colon \begin{bmatrix} \cH \\
  \cE \end{bmatrix} \to \cK \text{ is isometric,} \\
	 & \im i_{\cH} \perp \im i_{\cE_{*}} \text{ so } \begin{bmatrix}
	   i_{\cH} & i_{\cE_{*}} \end{bmatrix}  \colon \begin{bmatrix} \cH \\
	   \cE_{*} \end{bmatrix} \to \cK \text{ is isometric,}
 \end{align*}
 then,
for any $k \in \cK$, if $(\vec e, \vec h,
	\vec e_{*})$ of $\ell_{\cE}({\mathbb Z}) \times
\ell_{\cH}({\mathbb Z}) \times \ell_{\cE_{*}}({\mathbb	Z})$ is given
by
	\begin{align}
e(n) & = i_{\cE}^{*} \cU^{*n} k, \notag \\
 h(n) & = i_{\cH}^{*} \cU^{*n} k, \notag \\
  e_{*}(n) & = i_{\cE_{*}}^{*} \cU^{*n+1} k,
 \label{calUorbit}
 \end{align}
 then $(\vec e, \vec h, \vec e_{*})$ is a $U$-system
 trajectory, i.e., the system equations
\begin{equation}  \label{syseq}
	\begin{bmatrix} h(n+1) \\ e_{*}(n) \end{bmatrix} = U
   \begin{bmatrix} h(n) \\ e(n) \end{bmatrix}
\end{equation}
hold for all $n \in {\mathbb Z}$.
 {\em Under these assumptions there is no a priori way to
characterize which system trajectories $(\vec e, \vec h,
 \vec e_{*})$ arise from a $k \in \cK$ via formula
 \eqref{calUorbit}.  If we impose the additional
 structure:
\begin{enumerate}
 \item[]
 $\im i_{\cE}$ and $\im
 i_{\cE_{*}}$ are wandering subspaces for $\cU$, so there
 exist uniquely determined isometric embedding operators
 $$
 \vec i_{\cE} \colon \ell^{2}_{\cE}({\mathbb Z}) \to \cK,
 \quad
\vec i_{\cE_{*}} \colon \ell^{2}_{\cE_{*}}({\mathbb Z}) \to
  \cK
 $$
 which extend $i_{\cE}$ and $i_{\cE_{*}}$ in the sense that
 $$
 i_{\cE} = \vec i_{\cE} \circ i_{\cE}^{(0)}, \quad
 i_{\cE_{*}} = \vec i_{\cE_{*}} \circ i_{\cE_{*}}^{(-1)},
$$
  and $\cK$ has the orthogonal decomposition
 \begin{equation}  \label{ortho-decom}
 \cK = \im i_{\cH} \oplus \vec
 i_{\cE_{*}}(\ell^{2}_{\cE_{*}}({\mathbb Z}_{-})) \oplus \vec
 i_{\cE}(\ell^{2}_{\cE}({\mathbb Z}_{+})),
\end{equation}
 \end{enumerate}
 then one can characterize the system trajectories of the
form \eqref{calUorbit} as exactly those of finite-energy
in the sense that}
 \begin{equation}  \label{FE1}
\vec e \in \ell^{2}_{\cE}({\mathbb Z}) \text{ and } \vec
 e_{*} \in \ell^{2}_{\cE_{*}}({\mathbb Z}),
 \end{equation}
	 {\em or, equivalently, in the sense that}
	  \begin{equation}  \label{FE2}
	 \vec e|_{{\mathbb Z}_{+}} \in \ell^{2}_{\cE}({\mathbb
	Z}_{+}) \text{ and } \vec e_{*}|_{{\mathbb Z}_{-}} \in
	 \ell^{2}_{\cE_{*}}({\mathbb Z}_{-}).
	\end{equation}
 With the additional wandering-subspace assumption and
orthogonal-decomposition
 assumption \eqref{ortho-decom} given above in
 place, then the map $k \mapsto (e(n), h(n), e_{*}(n))$ defined by
\eqref{calUorbit} gives a one-to-one correspondence between
  elements $k$ of $\cK$ and finite-energy $U$-system trajectories
$(\vec e, \vec h, \vec e_{*})$. {\em This last statement
 is essentially Lemma 2.3 in \cite{BCU} and is the main
ingredient in the coordinate-free approach in embedding
a unitary colligation into a (discrete-time)
 Lax-Phillips scattering system.  The reader can check
 that the situation in Theorem \ref{T:Redheffer-central}
meets all these assumptions (with $\widetilde \Delta$ in
place of $\cE_{*}$ and $\widetilde \Delta_{*}$ in place of
 $\cE$); the computation in the
 proof of Theorem \ref{T:Redheffer-central} exhibits the $k
 = i_{\cH_{0}} h_{0} + i_{\widetilde \Delta_{*},0}
  \widetilde \delta_{*}$ corresponding to the
 finite-energy system
 trajectory supported on ${\mathbb Z}_{+}$ with initial
  condition  $h_{0}$ and impulse input supported at time $n=0$ equal
 to $\widetilde \delta_{*}$.
 We invite the reader to consult \cite{BSV} for an
extension of these ideas to a several-variable context.
}\end{remark}

\section{The characteristic measure of the universal scattering
system and associated Hellinger-space models} \label{S:charmeas}

In the sequel we assume that the subspaces $\cG' \subset \cK'_{+}$
and $\cG'' \subset \cK''_{-}$ are chosen to be
$$
 \cG' = \cK'_{+}, \quad \cG'' = \cK''_{-}.
$$
Given the scattering system ${\mathfrak S}_{0}$ \eqref{S0} arising
from the central extension $\cU_{0}$ associated with the data set
\begin{equation}  \label{CLdata}
X, \quad (\cU', \cK'), \quad  (\cU'', \cK''), \quad  \cK'_{+}
\subset \cK', \quad \cK''_{-}
\subset \cK''
\end{equation}
for a  Lifting Problem,  following the discussion in
Section \ref{S:systems}  (see \eqref{charmeas}) with a minor
adjustment, we   define the
{\em central characteristic measure} $\widehat \Sigma$  for the
Lifting-Problem data
set \eqref{CLdata} by
\begin{equation}  \label{central-charmeas}
    \widehat \Sigma_{0}(dt) =
    \begin{bmatrix} i_{\widetilde \Delta,0}^{ *} \cU_{0}^{*}\\
	i_{\widetilde \Delta_{*},0}^{ *} \\ i_{\cK''_{-},0}^{*}
	\\ i_{\cK'_{+},0}^{*} \end{bmatrix} E_{\cU_{0}}(dt)
	\begin{bmatrix} \cU_{0} i_{\widetilde \Delta,0} &
	    i_{\widetilde \Delta_{*},0} & i_{\cK''_{-},0} &
	    i_{\cK'_{+},0} \end{bmatrix}
\end{equation}
where $E_{\cU_{0}}(dt)$ is the spectral measure for the unitary
operator $\cU_{0}$.  The next theorem lists some special properties
of the central characteristic measure $\widehat \Sigma_{0}(dt)$.

\begin{theorem}  \label{T:central-charmeas}
    Suppose that $\widehat \Sigma_{0}(dt)$ is the central
    characteristic measure associated with a Lifting-Problem data
    set \eqref{CLdata} as in \eqref{central-charmeas}.  Then the
    following properties hold:
    \begin{enumerate}
	\item $\widehat \Sigma_{0}$ is a positive operator measure
	of the form
	\begin{equation}  \label{Sigma0form}
	\widehat \Sigma_{0} = \begin{bmatrix} mI_{\widetilde
	\Delta} & \widehat s & 0 & \widehat s_{1} \\
	\widehat s^{*} & m I_{\widetilde \Delta_{*}} & \widehat
	s_{2}^{*} & 0 \\ 0 & \widehat s_{2} & \sigma'' & \widehat
	s_{0} \\ \widehat s_{1}^{*} & 0 & \widehat s_{0}^{*} &
	\sigma' \end{bmatrix}
	\end{equation}
	where $m$ is Lebesgue measure on the unit circle ${\mathbb
	T}$.

 \item The Redheffer coefficient-matrix symbol
	$\left\{\sbm{ s_{0}(m) & s_{2}(m) \\ s_{1}(m) &
	s(m) } \right\}_{m \in {\mathbb Z}}$  for the case $\cG' =
	\cK'_{+}$, $\cG'' = \cK''_{-}$ (see Theorem
	\ref{T:symbol-param}) is the moment sequence of the
	corresponding measures $\widehat s_{0}$, $\widehat s_{2}$,
	$\widehat s_{1}$ and $\widehat s$ appearing in $\widehat
	\Sigma_{0}$:
	 \begin{equation}  \label{12.4}
	    \begin{bmatrix} s_{0}(m) & s_{2}(m) \\ s_{1}(m) & s(m)
\end{bmatrix}
		= \begin{bmatrix} i_{\cK''_{-},0}^{*} \\
		i_{\widetilde \Delta,0}^{*} \cU_{0}^{*} \end{bmatrix}  \cU_{0}^{*m}
		\begin{bmatrix} i_{\widetilde \Delta_{*},0}  &
		    i_{\cK'_{+},0} \end{bmatrix} =
		    \int_{{\mathbb T}} t^{-m} \begin{bmatrix}
		    \widehat s_{0} & \widehat s_{2} \\ \widehat
		    s_{1} & \widehat s \end{bmatrix} (dt).
      \end{equation}

 \item
 The measures $\widehat s$, $\widehat s_{2}$,
       $\widehat s_{1}$ are analytic operator-valued measures in
       the sense that
       \begin{equation}  \label{analytic-meas}
	   \int_{{\mathbb T}} t^{-m} \widehat s(dt) = 0 \text{ and }
	   \int_{{\mathbb T}} t^{-m} \widehat s_{1}(dt) = 0 \text{
	   for } m < 0, \quad
	   \int_{{\mathbb T}} t^{-m} \widehat s_{2}(dt) = 0 \text{
	   for } m \le 0
       \end{equation}
      and moreover
       \begin{equation}  \label{analytic-meas'}
	   \int_{{\mathbb T}} \widehat s(dt) = 0.
	   \end{equation}
	\end{enumerate}
 \end{theorem}

 \begin{proof}
   By the spectral theorem we see that the moment sequence of
$\widehat \Sigma_{0}$ is
   given by
 $$
 \widehat \Sigma_{0,m} = \begin{bmatrix} i_{\widetilde
 \Delta,0}^{*} \cU_{0}^{*}
 \\ i_{\widetilde \Delta_{*},0}^{*} \\ i_{\cK''_{-},0}^{*} \\
 i_{\cK'_{+},0}^{*} \end{bmatrix} \cU_{0}^{*m}
 \begin{bmatrix} \cU_{0} i_{\widetilde \Delta,0}
   & i_{\widetilde \Delta_{*},0} & i_{\cK''_{-},0} &
   i_{\cK'_{+},0}^{*} \end{bmatrix}.
 $$
 In particular,
 $$
 [ \widehat \Sigma_{0,m} ]_{1,1} = i_{\widetilde \Delta,0}^{*}
 \cU_{0}^{*m} i_{\widetilde \Delta,0} = \delta_{m,0} I_{\widetilde
 \Delta}
 $$
 (where $\delta_{m,0}$ is the Kronecker delta symbol equal to $1$
 for $m = 0$ and $0$ for $m \ne 0$) since $\im i_{\widetilde
 \Delta,0}$ is wandering for $\cU_{0}$.  We conclude that the
 $(1,1)$-entry  $[\widehat \Sigma_{0}]_{1,1}$ of $\widehat
 \Sigma_{0}$ is indeed $m I_{\widetilde \Delta}$ where $m$ is
 Lebesgue measure.  Similarly, the $(2,2)$-entry $[\widehat
 \Sigma_{0}]_{2,2}$ is equal to $m I_{\widetilde \Delta_{*}}$.

 From \eqref{Redheffer-central} we read off  that the Redheffer
 coefficient-matrix symbol \eqref{Redheffer-symbol} (for the case
 here where $\cG' = \cK'_{+}$ and $\cG'' = \cK''_{-}$) is given by
 \begin{align*}
     \begin{bmatrix} s_{0}(m) & s_{2}(m) \\ s_{1}(m) & s(m)
     \end{bmatrix} & = \begin{bmatrix} i_{\cK''_{-}}^{*} \\
     i_{\widetilde \Delta,0}^{*}  \cU_{0}^{*}\end{bmatrix}
\cU_{0}^{*m}
     \begin{bmatrix} i_{\cK'_{+}} & i_{\widetilde \Delta_{*},0}
     \end{bmatrix} \\
      & = \begin{bmatrix} \left[ \widehat \Sigma_{0,m}\right]_{3,4} &
     \left[ \widehat \Sigma_{0,m}\right]_{3,2} \\ \left[ \widehat
     \Sigma_{0,m} \right]_{1,4} & \left[ \widehat
     \Sigma_{0,m}\right]_{1,2} \end{bmatrix} \\
     & = \begin{bmatrix} \widehat s_{0,m} & \widehat s_{2,m} \\
     \widehat s_{1,m} & \widehat s_{m} \end{bmatrix} \\
     & =  \int_{{\mathbb T}} t^{-m} \begin{bmatrix}
		    \widehat s_{0} & \widehat s_{2} \\ \widehat
		    s_{1} & \widehat s \end{bmatrix} (dt)
 \end{align*}
 from which \eqref{12.4} follows immediately.
 The analyticity properties \eqref{analytic-meas} and
 \eqref{analytic-meas'} can now be seen as a
 consequence of the analyticity properties \eqref{zeros} and
 \eqref{zeros'} of the
 Redheffer coefficient-matrix symbol, or, equivalently, as a
 consequence of the orthogonality relations \eqref{analytic} and
 the fact that $\cU_{0}^{*} \im i_{\widetilde \Delta_{*},0} \perp
 \im i_{\widetilde \Delta,0}$.

 It remains only to verify that $[\widehat \Sigma_{0}]_{1,3} = 0$
 and $[\widehat \Sigma_{0}]_{2,4} = 0$.  In terms of moments, we
 must verify that
 $$
 i_{\widetilde \Delta,0}^{*} \cU_{0}^{*m+1} i_{\cK''_{-},0} = 0 \text{
 and } i_{\widetilde \Delta_{*},0}^{*} \cU_{0}^{*m} i_{\cK'_{+}} = 0
 \text{ for all } m \in {\mathbb Z}.
 $$
 But these conditions are the strengthened versions \eqref{zero13} and
 \eqref{zero24} of the
 orthogonality relations \eqref{orthogonal} in Theorem
 \ref{T:scatsys-properties}.  This concludes the proof of Theorem
 \ref{T:central-charmeas}.
     \end{proof}

     We now use the central characteristic measure to provide a
     Hellinger space model for the central extension $\cU_{0}$ and
     the associated (slightly adjusted) scattering space
     ${\mathfrak S}_{0}$ and four-fold AA-unitary coupling
     ${\mathfrak S}_{0,AA}$ as follows.  Consider the central
     scattering system with adjusted scale operator (still denoted
     ${\mathfrak S}_{0}$)
     $$
     \left(\cU_{0}, \quad \begin{bmatrix} \cU_{0} i_{\widetilde
     \Delta,0} & i_{\widetilde \Delta_{*},0} &
     i_{\cK''_{-},0} & i_{\cK'_{+},0} \end{bmatrix};
     \quad \cK_{0}, \quad \widetilde \Delta \oplus \widetilde
     \Delta_{*} \oplus \cK''_{-} \oplus \cK'_{+} \right).
     $$
   As $\im \begin{bmatrix} i_{\cK''_{-},0} & i_{\cK'_{+},0}
 \end{bmatrix}$ is $*$-cyclic for $\cU_{0}$, certainly $\im
 \begin{bmatrix} \cU_{0} i_{\widetilde
     \Delta,0} & i_{\widetilde \Delta_{*},0} & i_{\cK''_{-},0} &
     i_{\cK'_{+},0} \end{bmatrix}$ is
     $*$-cyclic for $\cU_{0}$ and the Fourier representation
     operator (see \eqref{Fourierrep})
     $$
      \cF_{0} \colon k_{0} \mapsto
      \begin{bmatrix} i_{\widetilde \Delta,0}^{ *} \cU_{0}^{*}\\
		i_{\widetilde \Delta_{*},0}^{ *} \\ i_{\cK''_{-},0}^{*}
		\\ i_{\cK'_{+},0}^{*} \end{bmatrix} E_{\cU_{0}}(dt)
		k_{0}
      $$
      maps $\cK_{0}$ unitarily onto the Hellinger space
      $\cL^{\widehat \Sigma_{0}}$ (see
      \cite{Kheifets-Dubov} for complete details).  Inside
      $\cL^{\widehat \Sigma_{0}}$ we consider the following
      subspaces:
      \begin{align}
	&  \bcK'':= \cF_{0}(\im i_{\cK'',0}) = \widehat \Sigma_{0}
	  \begin{bmatrix} 0 \\ 0 \\ \sigma^{\prime \prime [-1]}
	      \cL^{\sigma''} \\ 0 \end{bmatrix}, \quad
      \bcK':= \cF_{0}(\im i_{\cK',0}) = \widehat \Sigma_{0}
	      \begin{bmatrix} 0 \\ 0 \\ 0 \\ \sigma^{\prime [-1]}
		  \cL^{\sigma'} \end{bmatrix}, \notag \\
& \bcK''_{-}: = \cF_{0}(\im i_{\cK_{-}'', 0}) = \widehat \Sigma_{0}
\begin{bmatrix} 0 \\ 0 \\ \cK''_{-} \\ 0 \end{bmatrix},  \quad
  \bcK'_{+}: = \cF_{0}(\im \cK'_{+}) = \widehat \Sigma_{0}
  \begin{bmatrix} 0 \\ 0 \\ 0 \\ \cK'_{+} \end{bmatrix}, \notag \\
      & \bcH_{0}: = \cF_{0}(\im i_{\cH_{0},0}) =
      \operatorname{clos} \widehat \Sigma_{0} \begin{bmatrix} 0 \\
      0 \\ \cK''_{-} \\ \cK'_{+} \end{bmatrix}, \quad
 \bcD: = \cF_{0}(i_{\cH_{0},0}(\cD)) = \operatorname{clos}
\widehat \Sigma_{0} \begin{bmatrix} 0 \\ 0 \\ \cK''_{-} \\ \cU'
\cK'_{+} \end{bmatrix}, \notag \\
 &  \bcD_{*}: = \cF_{0}(i_{\cH_{0},0} (\cD_{*})) = \operatorname{clos}
\widehat \Sigma_{0} \begin{bmatrix} 0 \\ 0 \\ \cU^{\prime \prime *}
\cK''_{-} \\
\cK'_{+} \end{bmatrix}, \notag \\
& \widetilde \bDelta_{*}^{(0)}: =\cF_{0}(\im i_{\widetilde
\Delta_{*},0}) = \widehat \Sigma_{0}
\begin{bmatrix} 0 \\ \widetilde \Delta_{*} \\ 0 \\ 0 \end{bmatrix},
    \quad
    \widetilde \bDelta^{(-1)}:  = \cF_{0}(\im i_{\widetilde
    \Delta,0}) = t^{-1} \widehat \Sigma_{0} \begin{bmatrix}
    \widetilde \Delta \\ 0 \\ 0 \\ 0 \end{bmatrix}.
    \label{boldspaces}
      \end{align}
 As a translation of the conditions \eqref{10.7} we see that then we
 also have
 \begin{align}
  \widehat \Sigma_{0} \begin{bmatrix} \widetilde \Delta \\ 0 \\ 0
  \\ 0 \end{bmatrix} & = \textup{clos } \widehat \Sigma_{0}
\begin{bmatrix} 0
  \\ 0 \\ \cK''_{-} \\ \cK'_{+} \end{bmatrix} \ominus \textup{clos }
  \widehat \Sigma_{0} \begin{bmatrix} 0 \\ 0 \\ \cK''_{-} \\
   \cU' \cK'_{+} \end{bmatrix}, \notag \\
   t^{-1} \widehat \Sigma_{0} \begin{bmatrix} 0 \\ \widetilde
  \Delta_{*} \\ 0 \\ 0 \end{bmatrix} & = \textup{clos } \widehat
\Sigma_{0}
  \begin{bmatrix} 0 \\ 0 \\ \cK''_{-} \\ \cK'_{+} \end{bmatrix}
      \ominus \textup{clos } \widehat \Sigma_{0} \begin{bmatrix} 0
      \\ 0 \\ \cU^{\prime \prime *} \cK''_{-} \\ \cK'_{+}
\end{bmatrix}
      \label{10.7'}
 \end{align}
and as a consequence of \eqref{cK0decom} we have
\begin{equation}  \label{cK0decom-func}
    \cL^{\widehat \Sigma_{0}} = \widehat \Sigma_{0} \begin{bmatrix}
    H^{2 \perp}_{\widetilde \Delta} \\ 0 \\ 0 \\ 0 \end{bmatrix}
    \oplus \operatorname{clos} \widehat \Sigma_{0}\begin{bmatrix} 0
\\ 0 \\ \cK''_{-}
    \\ \cK'_{+} \end{bmatrix} \oplus \widehat \Sigma_{0}
    \begin{bmatrix} 0 \\ H^{2}_{\widetilde \Delta_{*}}  \\ 0 \\ 0
	\end{bmatrix}.
\end{equation}
In addition we have the following result which generalizes a result
of Adamjan-Arov-Kre\u{\i}n \cite{AAK68}.

\begin{theorem}  \label{T:Schur-com}  If $\widehat \Sigma_{0}$ as
    in \eqref{central-charmeas} is the central characteristic
    measure for a Lifting Problem, then
    \begin{equation}\label{10.6'}
    \begin{bmatrix} m I_{\widetilde
    \Delta} & \widehat s  \\
	  \widehat s^{\ *} & m I_{\widetilde \Delta_{*}}\end{bmatrix}=
	  \begin{bmatrix}0 &
	  \widehat s_{1}\\
	  \widehat s_{2}^{\ *} & 0 \end{bmatrix}
    \begin{bmatrix}
	  \sigma'' & \widehat s_{0} \\
	  \widehat s_{0}^{\ *} & \sigma '
	 \end{bmatrix}^{[-1]}
	 \begin{bmatrix}
	 0\ & \widehat s_{2}\\
    \widehat s_{1}^{\ *} & 0 \end{bmatrix}
    \end{equation}
    in the sense of measure Schur-complements as in
    \cite{Kheifets-Dubov}.
   \end{theorem}

    \begin{proof}  We are given that $\im i_{\cK''_{-}} + \im
	i_{\cK'_{+}}$ is $*$-cyclic for $\cU_{0}$ in $\cK_{0}$.
      By transforming this condition to the space $\cL^{\widehat
      \Sigma_{0}}$ under the unitary
      transformation $\cF_{0}$, we see that the space $\bcK'_{+} +
      \bcK''_{-}$ is weak-$*$ dense in $\cL^{\widehat \Sigma_{0}}$.
      As a consequence of statement \eqref{L1=L} in Theorem
      \ref{T:Hellinger-decom}, we see that this is equivalent
      to the property that the Schur complement of the block $\sbm{
      \sigma'' & \widehat s_{0} \\ \widehat s_{0}^{*} & \sigma'}$
      in the matrix-measure $\widehat \Sigma_{0}$ is zero, i.e.,
      condition \eqref{10.6'} holds.
\end{proof}

The next result says that, conversely, the properties
\eqref{analytic-meas}, \eqref{analytic-meas'}, \eqref{10.7'} and
\eqref{10.6'} can be used to characterize central characteristic
measures for a Lifting Problem.

\begin{theorem}  \label{T:converse2}  Suppose that we are given
    unitary operators $(\cU'', \cK'')$ and $(\cU', \cK')$ together
    with $*$-cyclic subspaces $\cK''_{-} \subset \cK''$ and
    $\cK'_{+} \subset \cK'$ invariant under $\cU^{\prime \prime *}$
    and $\cU'$ respectively.  Suppose that $\widetilde \Delta_{*}$
    and $\widetilde \Delta$ are two coefficient Hilbert spaces and
    that $\widehat \Sigma_{0}$ is an $\cL(\widetilde \Delta \oplus
    \widetilde \Delta_{*} \oplus \cK''_{-} \oplus \cK'_{+})$-valued
    measure of the form \eqref{Sigma0form} satisfying the following
    properties:
    \begin{enumerate}

	\item The measures $\sigma''$ and $\sigma'$ are given by
	\begin{equation}  \label{sigma-primes'}
	\sigma''(dt) = i_{\cK''_{-} \to \cK''}^{*}E_{\cU''}(dt)
	i_{\cK''_{-} \to \cK''}, \quad
	\sigma'(dt) = i_{\cK'_{+} \to \cK'}^{*} E_{\cU'}(dt)
	i_{\cK'_{+} \to \cK'}
	\end{equation}
	where $i_{\cK'' \to \cK''}$ and $i_{\cK'_{+} \to \cK'}$ are
	the inclusions of $\cK''_{-}$ into $\cK''$ and of
	$\cK'_{+}$ into $\cK'$ respectively, and where $E_{\cU''}$
	and $E_{\cU'}$ are the spectral measures for $\cU''$ and
	$\cU'$ respectively.

	\item The measure Schur-complement condition \eqref{10.6'}
	is satisfied.

	\item The measures $\widehat s$, $\widehat s_{1}$ and
	$\widehat s_{2}$ satisfy the analyticity conditions
	\eqref{analytic-meas}.

	\item Conditions \eqref{10.7'} hold.
  \end{enumerate}
  Then there is a Lifting Problem so that
  $(M_{t}, \cL^{\widehat \Sigma_{0}})$ is the associated central lift,
  the collection
    \begin{equation}  \label{model-scatsys}
	\bfrakS_{0} =
  (M_{t}, \quad  M_{\widehat \Sigma_{0}}; \quad \cL^{\widehat
  \Sigma_{0}}, \quad \widetilde \Delta \oplus \widetilde \Delta_{*}
  \oplus \cK''_{-} \oplus \cK'_{+})
  \end{equation}
  is the associated central scattering system, and
  $$
  \begin{bmatrix} s_{0}(m) & s_{2}(m) \\ s_{1}(m) & s(m)
  \end{bmatrix} : = \int_{{\mathbb T}} t^{-m} \begin{bmatrix}
  \widehat s_{0} & \widehat s_{2} \\ \widehat s_{1} & \widehat s
  \end{bmatrix} (dt)
  $$
  is the associated Redheffer coefficient-matrix symbol.
  Here $M_{t}$ is the operator of multiplication
  by the coordinate function $t$ on $\cL^{\widehat \Sigma_{0}}$ and
  $$
  M_{\widehat \Sigma_{0}} \colon \widetilde \Delta \oplus
  \widetilde \Delta_{*} \oplus \cK''_{-} \oplus \cK'_{+} \to
  \cL^{\widehat \Sigma_{0}}
  $$
  is the model scale operator of multiplication on the left by the
  measure $\widehat \Sigma_{0}$.
  \end{theorem}

  \begin{remark}  \label{R:Tconverse2}  {\em Part of the content of
      Theorem \ref{T:converse2} is that condition
      \eqref{analytic-meas'} is a consequence of the hypotheses
      given in the statement of the theorem, in particular, of the
       condition \eqref{10.7'}.
      } \end{remark}

  \begin{proof}
      Let $\widehat \Sigma_{0}$ be as in the statement of the
      theorem.  If we ignore the middle expressions involving
      subspaces of $\cK_{0}$ and the Fourier representation
      operator $\cF_{0}$, we may use formulas \eqref{boldspaces} to
      define subspaces $\bcK''$, $\bcK'$, $\bcK''_{-}$,
      $\bcK'_{+}$, $\bcH_{0}$, $\bcD$, $\bcD_{*}$, $\widetilde
      \bDelta_{*}^{(0)}$, $\widetilde \bDelta_{*}^{(0)}$,
      $\widetilde \bDelta^{(-1)}$ of $\cL^{\widehat \Sigma_{0}}$.

      We wish to apply the first version of Theorem \ref{T:converse1}
to the case where
      ${\mathfrak S}_{0}$ is equal to the Hellinger-model
      scattering system $\bfrakS_{0}$. We first verify that the
      model scattering system $\bfrakS_{0}$ \eqref{model-scatsys}
      extends to a model four-fold AA-unitary coupling
      \begin{equation}  \label{model-AA}
	  \bfrakS_{AA,0} = ( M_{t},\quad \vec \bi_{\widetilde
	  \Delta}, \quad \vec \bi_{\widetilde \Delta_{*}}, \quad
	  \bi_{\cK''}, \quad \bi_{\cK'}; \quad \cL^{\widehat
	  \Sigma_{0}})
       \end{equation}
       of the unitary operators
       \begin{equation} \label{4unitaries}
       (J, \ell^{2}_{\widetilde \Delta}({\mathbb Z})), \quad
       (J, \ell^{2}_{\widetilde \Delta_{*}}({\mathbb Z})), \quad
       (\cU'', \cK''), \quad (\cU', \cK')
       \end{equation}
       as required in Theorem \ref{T:converse1}.  For $\vec
       \delta \in \ell^{2}_{\widetilde \Delta}({\mathbb
       Z})$ of the form $\vec  \delta = J^{n}
       i_{\widetilde \Delta}^{(0)} \widetilde \delta$ for some $n
       \in {\mathbb Z}$ and $\widetilde \delta \in \widetilde
       \Delta$, we define
       $ \vec \bi_{\widetilde \Delta} \vec \delta = t^{n}
       \widehat \Sigma_{0}(dt)
       \sbm{  \widetilde \delta \\ 0 \\
       0 \\ 0}$.
       Since the $(1,1)$-entry of $\widehat \Sigma_{0}$ is $m \cdot
       I_{\widetilde \Delta}$, it follows that $\vec
       \bi_{\widetilde \Delta}$ extends to a well-defined isometry
       mapping $\ell^{2}_{\widetilde \Delta}({\mathbb Z})$ into
       $\cL^{\widehat \Sigma_{0}}$ with the additional property that
       $\bi_{\widetilde \Delta}:= t^{-1} \widehat \Sigma_{0}
       \sbm{ I \\ 0 \\ 0 \\ 0 } = \vec
	   \bi_{\widetilde \Delta} \circ i_{\widetilde
	   \Delta}^{(-1)} \colon \widetilde \Delta \to \widetilde
	   \bDelta^{(-1)}$.
       Similarly, the formula
       $
       \vec \bi_{\widetilde \Delta_{*}} \colon J^{n} i_{\widetilde
       \Delta_{*}}^{(0)} \widetilde \delta_{*} \mapsto t^{n}
       \widehat \Sigma_{0}(dt) \sbm{ 0 \\ \widetilde
       \delta_{*} \\ 0 \\ 0 }$
       extends to an isometric embedding of $\ell^{2}_{\widetilde
       \Delta_{*}}({\mathbb Z})$ into $\cL^{\widehat \Sigma_{0}}$
       with the extension property
       $\bi_{\widetilde \Delta_{*}} := \widehat \Sigma_{0}
       \sbm{ 0 \\ I \\ 0 \\ 0 } = \vec
	   \bi_{\widetilde \Delta_{*}} \circ i_{\widetilde
	   \Delta_{*}}^{(0)} \colon \widetilde \Delta_{*} \to
	   \widetilde \bDelta_{*}^{(0)}$.
       Moreover, the definition of $\sigma'$ and $\sigma''$ via
       \eqref{sigma-primes'} implies that $M_{t} M_{\sigma'} =
       M_{\sigma'} \cU'$ in $\cL(\cK'_{+}, \cL^{\sigma'})$ and
       $M_{t^{-1}} M_{\sigma''} = M_{\sigma''} \cU^{\prime \prime
       *}$ in $\cL(\cK''_{-}, \cL^{\sigma''})$.  This implies that
       we can use the wave-operator construction to construct
       isometric embeddings
       $$
       \bi_{\cK'} \colon \cK' \to \cL^{\widehat \Sigma_{0}}, \quad
       \bi_{\cK''} \colon \cK'' \to \cL^{\widehat \Sigma_{0}}
       $$
       with the extension properties
       \begin{align*}
	   & \bi_{\cK'}|_{\cK'_{+}} = \widehat \Sigma_{0}
	   \sbm{ 0 \\ 0 \\ 0 \\ I_{\cK'_{+}} }
	   =:\bi_{\cK'_{+}} \colon \cK'_{+} \to
	   \bcK'_{+},  \quad
	    \bi_{\cK''}|_{\cK''_{-}} = \widehat \Sigma_{0}
	   \sbm{ 0 \\ 0 \\ I_{\cK''_{-}} \\ 0 } =:\bi_{\cK''_{-}} \colon
\cK''_{-} \to
	   \bcK''_{-}.
       \end{align*}
       A consequence of the wave-operator construction is that
       $$
       \im \bi_{\cK''} = \bcK'', \quad \im \bi_{\cK'} = \bcK'.
       $$
       We also note the subspace identifications
       $$
       \vec \bi_{\widetilde \Delta}(\ell^{2}_{\widetilde
       \Delta}({\mathbb Z}_{-})) = \widehat \Sigma_{0}
       \sbm{ H^{2 \perp}_{\widetilde \Delta} \\ 0 \\ 0 \\
	   0 }, \quad
	   \vec \bi_{\widetilde \Delta_{*}} (\ell^{2}_{\widetilde
	   \Delta_{*}}({\mathbb Z}_{+})) = \widehat \Sigma_{0}
	   \sbm{ 0 \\ H^{2}_{\widetilde \Delta} \\ 0 \\ 0}.
      $$
It is now clear that \eqref{model-AA} is a four-fold AA-unitary
coupling of the four unitary operators \eqref{4unitaries} which
extends the scattering system $\bfrakS_{0}$ as required in Theorem
\ref{T:converse1}.

To apply the first version of Theorem \ref{T:converse1}, it remains
to check the
conditions \eqref{10.6}, \eqref{analytic}, \eqref{orthogonal} and
\eqref{10.7}.
As noted in the proof of
      Theorem \ref{T:Schur-com}, the vanishing of the measure
      Schur-complement \eqref{10.6'} is the functional-model version
of the
      condition \eqref{10.6}.  The analyticity conditions
      \eqref{analytic-meas} imply that
      $$\bcK'_{-} \perp \widehat \Sigma_{0} \sbm{
      H^{2 \perp}_{\widetilde \Delta} \\ 0 \\ 0 \\ 0 }, \quad
      \bcK'_{+} \perp \widehat \Sigma_{0} \sbm{ 0 \\
      H^{2}_{\widetilde \Delta_{*}} \\ 0 \\ 0 }, \quad
      \bcK''_{-} \perp \bcK'_{+}
      $$
     where $\bcK'_{+}$ and $\bcK''_{-}$ are
      defined as in \eqref{boldspaces}; one can check that these
      conditions are just the functional-model version of the
      orthogonality conditions \eqref{analytic}.  The presence of
      the $0$ in the $(1,3)$ and $(2,4)$ locations of $\widehat
      \Sigma_{0}$ is equivalent to the orthogonality conditions
      $$
      \widehat \Sigma_{0} \sbm{ L^{2}_{\widetilde
      \Delta} \\ 0 \\ 0 \\ 0 } \perp \bcK'', \quad
      \widehat \Sigma_{0} \sbm{ 0 \\ L^{2}_{\widetilde \Delta_{*}}
      \\ 0 \\ 0 } \perp 	\bcK'
      $$
      which is the functional-model equivalent of condition
      \eqref{orthogonal} in the stronger form
      \eqref{orthogonal-strong}.  As we have already noted, the
      assumption \eqref{10.7'} is the functional-model equivalent
      of conditions \eqref{10.7}.

      By Theorem \ref{T:converse1} (first version) we conclude that
$(M_{t}, \cL^{\widehat
      \Sigma_{0}})$ is the central extension, $\bfrakS_{0}$ is the
      central scattering system, and $\bfrakS_{AA,0}$ is the
      four-fold AA-unitary coupling associated with the
      Lifting Problem with data set
      \begin{equation}  \label{model-data}
	   X = \bi_{\cK''}^{*} \bi_{\cK'} = \widehat
	  s_{0}({\mathbb T}), \quad (\cU'', \cK''), \quad
	  (\cU', \cK'), \quad \cK''_{-} \subset \cK'', \quad
	  \cK'_{+} \subset \cK'.
      \end{equation}

      Finally, application of the
      formula \eqref{Redheffer-central} to the present setting
      tells us that the associated Redheffer coefficient-matrix
      symbol is given by
      \begin{align*}
	  \begin{bmatrix} s_{0}(m) & s_{2}(m) \\ s_{1}(m) & s(m)
	  \end{bmatrix}& =
	  \int_{{\mathbb T}} \begin{bmatrix} 0 & 0 & I & 0 \\ I & 0
	  & 0 & 0 \end{bmatrix} t^{-m} \widehat \Sigma_{0}(dt)
	  \begin{bmatrix} 0 & 0 \\ 0 & I \\ 0 & 0 \\  I & 0
	  \end{bmatrix} \\
	  & = \int_{{\mathbb T}} t^{-m} \begin{bmatrix} \widehat
	  s_{0} & \widehat s_{2} \\ \widehat s_{1} & \widehat s
      \end{bmatrix} (dt).
      \end{align*}
      This concludes the proof of Theorem \ref{T:converse2}.
      \end{proof}

      \begin{remark}\label{1013}{\em
      S.~ter Horst \cite{tH-pre} has obtained a
       solution of the inverse Relaxed
       Commutant Lifting (RCL) problem. Although, RCL is a generalization
       of the Commutant Lifting considered in the present paper,
       the setting of the inverse problem of \cite{tH-pre} is quite
       different from our inverse problem even in the context of the
       classical Commutant Lifting.
       We discuss the difference in this remark.

       If we use the formulation in Remark \ref{R:param}, we see that
       lifts $Y$ have a representation of the form
       $$
        Y = \begin{bmatrix} X \\ \Gamma_{H} D_{X} \end{bmatrix}
  $$
  where $\Gamma_{H}$ is the operator of multiplication by the
  function $H(\zeta): = Y_{0+}(\zeta)$ as in formula \eqref{FFparam'}.  The set of
  such functions $H(\zeta)$ is in turn parametrized by the Redheffer
  linear-fractional map $H(\zeta) = \Phi_{11}(\zeta) + \Phi_{12}(\zeta)
  (I - V(\zeta) \Phi_{22}(\zeta))^{-1} \Phi_{21}(\zeta)$ where $V$ is
  a free-parameter operator-valued Schur-class function of
  appropriate size and $\Phi = \left[ \begin{smallmatrix} \Phi_{11} &
  \Phi_{12} \\ \Phi_{21} & \Phi_{22} \end{smallmatrix} \right]$ is
  the Redheffer coefficient matrix. The inverse problem as formulated
  in \cite{tH-pre} is to characterize which Redheffer coefficient
  matrices arise in this way.  The result is that {\em any}
  coefficient matrix $\Phi \colon {\mathbb D} \to \cL(\cU_{1} \oplus
  \cU_{2}, \cY_{1} \oplus \cY_{2})$ such that the multiplication
  operator
  $$  \begin{bmatrix} M_{\Phi_{11}} & \Gamma_{\Phi_{12}} \\
  M_{\Phi_{21}} & \Gamma_{\Phi_{22}} \end{bmatrix} \colon
  \begin{bmatrix} H^{2}_{\cU_{1}} \\ \cU_{2} \end{bmatrix} \to
      \begin{bmatrix} H^{2}_{\cY_{1}} \\ H^{2}_{\cY_{2}} \end{bmatrix}
  $$
  is coisometeric arises in this way.  In this formulation of the inverse problem,
  there is flexibility in the construction of the underlying
  contraction $\rho$, i.e., a contraction
       $\rho = \left[ \begin{smallmatrix} \rho_{1} \\ \rho_{2} \end{smallmatrix}
       \right]$ so that the analogue of \eqref{intertwine2} is satisfied
       for all $Y_{0+} = \cF_{\Psi}[\omega]$, as well as in the
       choice of a compatible RCL data set.

 Our Theorems \ref{T:converse1} and  \ref{T:converse2} can also be
 considered as solutions of inverse problems.  In Theorem
 \ref{T:converse2} we specify not only the Redheffer coefficient
 matrix but also the measures $\sigma'$ and $\sigma''$ giving
 Hellinger-space models for the unitary operators $\cU'$ and $\cU''$
 as well as the subspaces $\cK'_+\subset \cK'$ and $\cK''_-\subset \cK''$
 and the intertwining operator $X$, i.e., the whole data set for the
 Lifting Problem.  It is then automatic that the given Redheffer
 coefficient matrix parametrizes some subset of the set of all
 solutions of the associated Lifting Problem for this data set.  The
 only remaining issue is to characterize the structure required which
 guarantees that the image of the Redheffer linear-fractional map
 gives rise to the set of {\em all} solutions of the given Lifting
 Problem.  Theorem \ref{T:converse1} is just a coordinate-free
 version of the same result: one specifies a four-fold AA-unitary coupling
 and seeks to characterize when it is the four-fold AA-unitary coupling
 coming from a Lifting Problem.  Another version of the inverse
 theorem is Theorem \ref{T:converse3} below where one is given just
 the Redheffer coefficient matrix along with the measures $\sigma''$
 and $\sigma'$.
 }\end{remark}

      We close this section with some observations concerning the
      absolute continuity of the measures $\widehat s_{1}$,
      $\widehat s_{2}$ and $\widehat s$ inside the central
      characteristic measure coming from a  Lifting Problem; these
      results generalize similar results for the case of the Nehari
      problem in \cite{AAK68}.

  \begin{theorem}  \label{T:AC}
   Suppose that $\widehat \Sigma_{0}$ is a positive strong
operator-measure of the form \eqref{Sigma0form}.  Then
the measure $\widehat s$ is uniformly absolutely
continuous with respect to Le\-besgue measure and the
 measures $\widehat s_{1}$ and $\widehat s_{2}$ are
 strongly absolutely continuous with respect to Lebesgue
measure.
\end{theorem}

\begin{proof}
    The positivity of $\widehat \Sigma_{0}$ (see \eqref{Sigma0form})
implies that
$$
 \|\widehat s(B)\| \le m(B)
 $$
 for every Borel set $B\subseteq\mathbb T$; this in turn means that
 $\widehat s$ is uniformly absolutely continuous with
respect to $m$ (see Section 3 in
\cite{Kheifets-Dubov} for more details about absolute continuity,
in particular, definitions (3.1), (3.5), (3.23), (3.24) and
 formulas (3.6), (3.26) appearing there). The positivity of $\widehat
 \Sigma_{0}$
 also implies that
 $$
 \Vert\widehat s_1(B)k'_+\Vert\le\sqrt{m(B)}\ \sqrt{\langle
 \sigma'(B)k'_+, k'_+\rangle}.
 $$
 Then (see (3.1) of Section 3 in \cite{Kheifets-Dubov} for
 definition of the variation of a vector measure and also Theorem
 3.4 ibidem for more details)
 $$
\text{var\ }_{\widehat s_1 k'_+}(B)\le \sqrt{m(B)}\ \sqrt{\langle
 \sigma'(B)k'_+, k'_+\rangle}.
 $$
This means that the measure $\widehat s_1 k'_+$ is absolutely
 continuous with respect to Lebesgue measure.
 \end{proof}

 \begin{remark}  \label{R:AC}  {\em
     Here we give a couple of alternate routes
to results on the absolute continuity of $\widehat s$, $\widehat
s_{1}$ and
$\widehat s_{2}$ which handle various special cases.
	
If we in addition impose
the analyticity conditions \eqref{analytic-meas} on
 the entries $\widehat s_{1}$, $\widehat s_{2}$ and
 $\widehat s$, then at least weak absolute continuity
of $\widehat s$, $\widehat s_{1}$ and $\widehat s_{2}$
 with respect to Lebesgue measure is an immediate
 consequence of the F.~and M.~Riesz theorem (see
 e.g.~\cite[page 47]{Hoffman}).
	
 Alternatively, if we assume all the conditions in Theorem
\ref{T:converse2} so that
$\widehat \Sigma_{0}$ is the universal characteristic
measure coming from a  Lifting Problem, by
using the connection \eqref{Redheffer-central}
between $\widehat \Sigma_{0}$ and the Redheffer
coefficient-matrix symbol together with the explicit
formula (see \eqref{def-Redheffer-param} and
\eqref{def-Redheffer}) for the Redheffer coefficient
 matrix, we see that
$$
\widehat s(dt) = s(t) \cdot m(dt), \quad
\widehat s_{1}(t) = s_{1}(t) \cdot m(dt), \quad
\widehat s_{2}(dt) = s_{2}(t)  \cdot m(dt)
 $$
where $s(t)$ is the boundary-value function for the
Schur-class function
$$
s(\zeta) = \zeta (i_{\cK''_{-} \to \cH_{0}})^{*}
C_{0} (I - \zeta A_{0})^{-1}B_{0} i_{\cK'_{+} \to
\cH_{0}}
 $$
 and $s_{1}(t)$ and $s_{2}(t)$ are the boundary-value
 functions for the strong operator-valued $H^{2}$-functions
\begin{align*}
 & s_{1}(\zeta) = (i_{\cK''_{-} \to \cH_{0}})^{*}
C_{0}(I - \zeta A_{0})^{-1} i_{\cK'_{+} \to
\cH_{0}} \\
 & s_{2}(\zeta) = (i_{\cK''_{-}\to \cH_{0}})^{*}
 (I - \zeta A_{0})^{-1} B_{0} i_{\cK'_{+} \to \cH_{0}}
 \end{align*}
 where $U_{0} = \sbm{A_{0} & B_{0} \\ C_{0} & 0 }$ is
the universal colligation constructed from the
 Lifting Problem data (see \eqref{U0-1}, \eqref{U0-2},
 \eqref{U0-3}).
 }\end{remark}

 \section{A more compact Hellinger-space model and maximal
 factorable minorants}  \label{S:compactHel}

 Suppose that $\widehat \Sigma_{0}$ is the central
 characteristic measure \eqref{central-charmeas} coming from a
  Lifting Problem.  As we have seen, the subspace
 \begin{equation}  \label{dense-space}
 \widehat \Sigma_{0} \begin{bmatrix} 0 \\ 0 \\ \cK''_{-}[t] \\
 \cK'_{+}[t^{-1}] \end{bmatrix}
 \end{equation}
 (where $\cK''_{-}[t]$ is the space of analytic trigonometric
 polynomials with coefficients in $\cK''_{-}$ and similarly
 $\cK'_{+}[t^{-1}]$ is the space of conjugate-analytic
 trigonometric polynomials with coefficients from $\cK'_{+}$)
 is dense in $\cL^{\widehat \Sigma_{0}}$ as a consequence of
 the condition \eqref{10.6} (or, in measure-theoretic terms,
 of \eqref{10.6'}).  By applying Theorem
 \ref{T:Hellinger-decom} with $\widehat \sigma_{0}: = \sbm{
 \sigma'' & \widehat s_{0} \\ \widehat s_{0}^{*} & \sigma' }$
 in place of $\sigma_{11}$, we see that the closure of the
 space \eqref{dense-space} can be identified with
 $\cL^{\widehat \sigma_{0}}$. Moreover, as long as the condition
\eqref{10.6'} is
 in force, as a consequence of the observation \eqref{L1=L} we
 see that the map
 \begin{equation}  \label{sigma0Sigma0}
 U_{\widehat  \sigma_{0}, \widehat \Sigma_{0}}\colon\widehat
\sigma_{0}
 \begin{bmatrix} p'' \\ p' \end{bmatrix}
 \mapsto \widehat \Sigma_{0} \begin{bmatrix} 0 \\ 0 \\ p'' \\
 p' \end{bmatrix}
 \end{equation}
 (where $\sbm{p'' \\p' } \in \sbm{ \cK''_{-} \\ \cK'_{+}}[t,
 t^{-1}]$ is a trigonometric polynomial with coefficients in
 $\sbm{ \cK''_{-} \\ \cK'_{+}}$) extends to define a unitary
 transformation from $\cL^{\widehat \sigma_{0}}$ onto
 $\cL^{\widehat \Sigma_{0}}$.  Alternatively, we may construct
 directly a second model central scattering system with
 ambient space $\cL^{\widehat \sigma_{0}}$ as follows.

 Define maps
 $$
 \bi_{\cK'_{+},0} \colon \cK'_{+} \to \cL^{\widehat
 \sigma_{0}}, \quad
 \bi_{\cK''_{-},0} \colon \cK''_{-} \to \cL^{\widehat \sigma_{0}}
 $$
 by
 $$
 \bi_{\cK'_{+},0} \colon k'_{+} \mapsto \widehat \sigma_{0}
 \begin{bmatrix} 0 \\ k'_{+} \end{bmatrix}, \quad
     \bi_{\cK''_{-},0} \colon k''_{-} \mapsto \widehat \sigma_{0}
     \begin{bmatrix} k''_{-} \\ 0 \end{bmatrix}
 $$
 and extend them to
 $$
 \bi_{\cK',0} \colon \cK' \to \cL^{\widehat \sigma_{0}}, \quad
 \bi_{\cK'',0} \colon  \cK'' \to \cL^{\widehat \sigma_{0}}
 $$
 by
 $$
 \bi_{\cK',0} \colon \cU^{\prime * n} k'_{+} \mapsto \widehat
 \sigma_{0} \begin{bmatrix} 0 \\ t^{-n} k'_{+} \end{bmatrix},
 \quad
 \bi_{\cK'',0} \colon \cU^{\prime \prime n} k''_{-} \mapsto
 \widehat \sigma_{0} \begin{bmatrix} t^{n} k''_{-} \\ 0
    \end{bmatrix}
    $$
    where $k'_{+} \in \cK'_{+}$ and $k''_{-} \in \cK''_{-}$. We also
    define subspaces
    $$
    \widetilde \bDelta_{0}^{(-1)}: = \begin{bmatrix} 0 \\ t^{-1}
    \widehat s_{1}^{*} \widetilde \Delta \end{bmatrix}, \quad
    \widetilde \bDelta_{*0}^{(0)}: = \begin{bmatrix} \widehat s_{2}
    \widetilde \Delta_{*} \\ 0  \end{bmatrix}
    $$
    and define isometric embedding operators
    $$
    \bi_{\widetilde \Delta,0} \colon \widetilde \Delta \to \widetilde
    \bDelta_{0}^{(-1)}, \quad
    \bi_{\widetilde \Delta_{*},0} \colon \widetilde \Delta_{*} \to
    \widetilde \bDelta_{0}^{(0)}
    $$
    by
    $$
    \bi_{\widetilde \Delta,0} \colon \widetilde \delta \mapsto
    \begin{bmatrix} 0 \\ t^{-1} \widehat s_{1}^{*} \widetilde \delta
 \end{bmatrix}, \quad
 \bi_{\widetilde \Delta_{*},0} \colon \widetilde \delta_{*}
 \mapsto \begin{bmatrix} \widehat s_{2} \widetilde \delta_{*}
 \\ 0 \end{bmatrix}
   $$
  with extensions
  $$
  \vec \bi_{\widetilde \Delta,0} \colon \ell^{2}_{\widetilde
  \Delta}({\mathbb Z}) \to \cL^{\widehat \sigma_{0}}, \quad
  \vec \bi_{\widetilde \Delta_{*},0} \colon \ell^{2}_{\widetilde
  \Delta_{*}}({\mathbb Z}) \to \cL^{\widehat \sigma_{0}}
  $$
  given by
  \begin{align*}
  & \vec \bi_{\widetilde \Delta,0} \colon \{ \widetilde \delta(
n)\}_{n
  \in {\mathbb Z}} \mapsto \begin{bmatrix} 0 \\ \widehat s_{1}^{*}
\left(
  \sum_{n \in {\mathbb Z}} \widetilde \delta(n) t^{n-1} \right)
  \end{bmatrix}, \\
   & \vec \bi_{\widetilde \Delta_{*},0} \colon \{ \widetilde
  \delta_{*}(n)\}_{n \in {\mathbb Z}} \mapsto \begin{bmatrix}
  \widehat s_{2} \left( \sum_{n \in {\mathbb Z}} \widetilde
\delta_{*}(n) t^{n}
  \right) \\ 0 \end{bmatrix}.
  \end{align*}
  The fact that all these maps are isometries is a consequence of the
  identity \eqref{10.6'}.  One can check (with occasional use of
  various tools for manipulation of Hellinger spaces from
  \cite{Kheifets-Dubov} to complete the details) that
  $$
  \bfrakS_{00}:= (M_{t}, \quad
  \begin{bmatrix} \bi_{\widetilde \Delta,0} & \bi_{\widetilde
      \Delta_{*},0} & \bi_{\cK''_{-},0} & \bi_{\cK'_{+},0}
  \end{bmatrix}; \quad \cL^{\widehat \sigma_{0}}, \quad
  \widetilde \Delta \oplus \widetilde \Delta_{*} \oplus \cK''_{-}
  \oplus \cK'_{+} )
  $$
  is also a central scattering system associated with the same
   Lifting-Problem data set \eqref{CLdata} which extends (in the
  sense of Theorem \ref{T:converse1}) to the four-fold AA-unitary
  coupling
  $$
  \bfrakS_{AA,00} : = ( M_{t}, \quad  \vec \bi_{\widetilde \Delta,0},
  \quad \vec \bi_{\widetilde \Delta_{*},0}, \quad \bi_{\cK'',0} \quad,
  \bi_{\cK',0}; \quad \cL^{\widehat \sigma_{0}})
  $$
  of the unitary operators \eqref{4unitaries}.
  We mention that the property \eqref{MFM-ortho} assumes the
  following form for the $\cL^{\widehat \sigma_{0}}$ functional model:
  \begin{align}
     & \operatorname{clos} \widehat \sigma_{0} \begin{bmatrix}
      \cK''_{-}[t] \\ \cK'_{+} \end{bmatrix} = \im \bi_{\cH_{0},0}
      \oplus \begin{bmatrix} \widehat s_{2} H^{2}_{\widetilde
      \Delta_{*}} \\ 0 \end{bmatrix} = \cL^{\widehat \sigma_{0}}
      \ominus \begin{bmatrix} 0 \\ \widehat s_{1}^* H^{2
      \perp}_{\widetilde \Delta} \end{bmatrix}, \notag \\
      & \operatorname{clos} \widehat \sigma_{0} \begin{bmatrix}
      \cK''_{-} \\ \cK'_{+}[t^{-1}] \end{bmatrix} = \im
\bi_{\cH_{0},0}
      \oplus \begin{bmatrix} 0 \\ \widehat s_{1}^*
H^{2\perp}_{\widetilde \Delta}
  \end{bmatrix} = \cL^{\widehat \sigma_{0}} \ominus \begin{bmatrix}
  \widehat s_{2} H^{2}_{\widetilde \Delta_{*}} \\ 0 \end{bmatrix}.
\label{MFM-ortho-funcmodel}
    \end{align}

  We note that construction of the space $\cL^{\widehat
  \sigma_{0}}$ does not explicitly make use of $\widehat s_{1}$,
  $\widehat s_{2}$, $\widehat s$ which appear in $\widehat
  \Sigma_{0}$; hence it must be the case that $\widehat s_{1}$,
  $\widehat s_{2}$ and $\widehat s$ are already somehow encoded in
  the ingredients $\sigma'$, $\sigma''$ and $\widehat s_0$ of
$\widehat \Sigma_{0}$.  The goal of this final section is to make this
idea precise.  The first step is the following result concerning maximal
  factorable minorants.

  \begin{theorem}  \label{T:MFM}
      Suppose that $\widehat \Sigma_{0}$ of the form
      \eqref{Sigma0form} is the central
      characteristic measure \eqref{central-charmeas}
      for some Lifting Problem.   Then the following hold:
      \begin{enumerate}
   \item The domination property
   \begin{equation}  \label{dom1}
       \sigma' - \widehat s_{0}^{*} \sigma^{\prime \prime
       [-1]} \widehat s_{0} \ge \widehat s_{1}^{*} \frac{1}{m}
       \widehat s_{1}
    \end{equation}
holds between $\widehat s_{0}$ and $\widehat s_{1}$.  Moreover,
$\widehat s_{1}^{*} \frac{1}{m} \widehat s_{1}$ is a {\em right
maximal
  factorable minorant} of $\sigma' - \widehat s_{0}^{*} \sigma^{\prime
  \prime [-1]} \widehat s_{0}$ in the following sense:  if
 $r_{1}^*$ is a strong
    $\cL(\cK'_{+}, \cN)$-valued conjugate-analytic
    measure such that
    \begin{equation}\label{factorable1}
     \sigma' - s_{0}^*\sigma^{\prime \prime [-1]} s_{0} \ge
    r_{1}^*\frac{1}{m}
    r_{1},
    \end{equation}
      then there is a contractive strongly conjugate-analytic
    $\cL(\cN , \widetilde\Delta)$-valued function
    $\theta_1^{*}$ such that
    \begin{equation}  \label{factor1}
    r_{1}^{*} =  \widehat s_{1}^{*} \theta_{1}^{*}
    \end{equation}
  (where the equality between measures holds in the  strong sense).

  \item The domination property
  \begin{equation}  \label{dom2}
	 \sigma'' - \widehat s_{0} \sigma^{\prime
	 [-1]} \widehat s_{0}^{*} \ge (t^{-1} \widehat s_{2}) \frac{1}{m}
	 (t^{-1}\widehat s_{2})^{*}
      \end{equation}
  holds between $\widehat s_{0}$ and $\widehat s_{2}$.  Moreover,
  $(t^{-1}\widehat s_{2}) \frac{1}{m}(t^{-1} \widehat s_{2})^{*}$ is
a {\em left maximal
  factorable minorant} of $\sigma'' - \widehat s_{0}\sigma^{ \prime
[-1]}
  \widehat s_{0}^{*}$ in the following sense:
  if
   $r_{2}$ is a strong
      $\cL(\cN_{*}, \cK''_{-})$-valued analytic
      measure such that
      \begin{equation}\label{factorable2}
\sigma'' - s_{0}\sigma^{\prime [-1]} s_{0}^{*} \ge
      r_{2}\frac{1}{m}
      r_{2}^{*},
      \end{equation}
 then there is a contractive strongly analytic
      $\cL(\cN_{*} , \widetilde\Delta_{*})$-valued function
      $\theta_{2}$ such that
      \begin{equation}  \label{factor2}
      r_{2} =  (t^{-1}\widehat s_{2}) \theta_{2}.
      \end{equation}
      \end{enumerate}
      Thus, given  a measure $\widehat s_{0}$ so that $\sbm{\sigma'' &
\widehat s_{0} \\ \widehat s_{0}^{*} & \sigma'} \ge 0$, the
remaining nonzero entries $\widehat s_{1}$ and $\widehat s_{2}$ of
a central measure $\widehat \Sigma_{0}$ for a lifting problem are
uniquely determined up to unitary-constant left/right-factor
normalizations of
$\widehat s_{1}$ and
$\widehat s_{2}$.
      \end{theorem}

      \begin{proof}  We prove only the first statement as the second
   is completely analogous.  From the positivity of $\widehat
   \Sigma_{0}$ we deduce the positivity of any of the  $3 \times
   3$ principal submatrices, in particular:
   \begin{equation}  \label{3by3principal}
     \begin{bmatrix} mI_{\widetilde \Delta} & 0 & \widehat
	 s_{1} \\ 0 & \sigma'' & \widehat s_{0} \\ \widehat
	 s_{1}^{*} & \widehat s_{0}^{*} & \sigma'
     \end{bmatrix} \ge 0.
  \end{equation}
  By a standard Schur-complement argument, this is equivalent
  to
  $$
    \sigma' - \begin{bmatrix} \widehat s_{1}^{*} & \widehat
    s_{0}^{*} \end{bmatrix} \begin{bmatrix} m I_{\widetilde
    \Delta} & 0 \\ 0 & \sigma'' \end{bmatrix} ^{[-1]}
    \begin{bmatrix} \widehat s_{1} \\ \widehat s_{0}
    \end{bmatrix} \ge 0
  $$
  from which we get \eqref{dom1}.

  Next, suppose that $r_{1}$ is as in the hypotheses of the
  theorem.  By the same Schur-complement argument one can see
  that \eqref{factorable1} is equivalent to
  $$
  \widehat \sigma_{0} : = \begin{bmatrix} \sigma'' & \widehat
  s_{0} \\ \widehat s_{0}^{*} & \sigma' \end{bmatrix} \ge
  \begin{bmatrix} 0 \\ r_{1}^{*} \end{bmatrix} \frac{1}{m}
      \begin{bmatrix} 0 & r_{} \end{bmatrix}.
  $$
  Another Schur-complement argument converts this to
  $$ \begin{bmatrix} m I_{\cN} & 0 & r_{1} \\ 0 & \sigma'' &
      \widehat s_{0} \\ r_{1}^{*} & \widehat s_{0}^{*} &
      \sigma' \end{bmatrix} \ge 0.
   $$
   By the definition of Hellinger spaces (see Section
   \ref{S:Hellinger} and especially \eqref{majorize}), it
   now follows that $\sbm{ 0 \\ r_{1}^{*} } n \in
   \cL^{\widehat \sigma_{0}}$ for each $n \in \cN$.  Due to
   the conjugate-analyticity of $r_{1}^{*}n$ we then have that
   \begin{equation}  \label{have1}
   t^{-1} \begin{bmatrix} 0 \\ r_{1}^{*} n \end{bmatrix} \perp
   \widehat \sigma_{0} \begin{bmatrix} \cK''_{-}[t] \\
   \cK'_{+} \end{bmatrix}.
   \end{equation}
   As a consequences of \eqref{MFM-ortho-funcmodel},
   the closure of the space $\widehat
   \sigma_{0} \sbm{\cK''_{-}[t] \\ \cK'_{+}}$ can be
   identified exactly as the orthogonal complement of the
   space $\sbm{0 \\ \widehat s_{1}^{*} H^{2 \perp}_{\widetilde
   \Delta}}$ in $\cL^{\widehat \sigma_{0}}$.  We conclude that
   there is an element $g_{-,n}$ of $H^{2 \perp}_{\widetilde
   \Delta}$ so that
   $$
   t^{-1} \begin{bmatrix} 0 \\ r_{1}^{*} n \end{bmatrix} =
   \begin{bmatrix} 0 \\ \widehat s_{1}^{*} \end{bmatrix} g_{-,n}.
    $$
    As the correspondence $n \mapsto g_{-,n}$ is linear, it
    follows that there is a strongly conjugate-analytic
    function $\theta_{1}^{*}$ so that $ g_{-,n} =
    t^{-1} \theta_{1}^{*} n$.   From \eqref{have1} we now arrive at
    \eqref{factor1}.  Taking the Schur complement of $\sbm{
    \sigma'' & \widehat s_{0} \\ \widehat s_{0}^{*} &
    \sigma'}$ in \eqref{3by3principal} then gives
    $$
	m I_{\cN} - \begin{bmatrix} 0 & r_{1} \end{bmatrix}
       \begin{bmatrix} \sigma'' & \widehat s_{0} \\ \widehat
s_{0}^{*} &
	\sigma' \end{bmatrix} ^{[-1]} \begin{bmatrix} 0 \\
	r_{1}^{*} \end{bmatrix} =
	m I_{\cN} - \theta_{1} \begin{bmatrix} 0 & \widehat
	s_{1} \end{bmatrix} \begin{bmatrix} \sigma'' &
	\widehat s_{0} \\ \widehat s_{1}^{*} & \sigma'
    \end{bmatrix} ^{[-1]} \begin{bmatrix} 0 \\ \widehat
    s_{1}^{*} \end{bmatrix} \theta_{1}^{*} \ge 0.
    $$
But a consequence of \eqref{10.6'} is that
   $$
   \begin{bmatrix} 0 & \widehat s_{1} \end{bmatrix}
\begin{bmatrix} \sigma'' & \widehat s_{0} \\ \widehat
    s_{0}^{*} & \sigma' \end{bmatrix}^{[-1]} \begin{bmatrix} 0
    \\ \widehat s_{1} \end{bmatrix} = m I_{\widetilde \Delta}
  $$
  and it follows that $I_{\cN} - \theta_{1} \theta_{1}^{*} \ge 0$ and
  $\theta$ is contractive as required.  This completes the proof of
  Theorem \ref{T:MFM}.
  \end{proof}

  In a similar vein one can obtain another positivity property for
  which the maximal factorable minorant is zero.

  \begin{theorem} \label{T:MFM'}
      Suppose that $\widehat \Sigma_{0}$ of the form
      \eqref{Sigma0form} is the central characteristic measure
      \eqref{central-charmeas} for some Lifting Problem.  Let
      $$
\widehat S = \begin{bmatrix} \widehat s_{0} & \widehat s_{2}
\\ \widehat s_{1} & \widehat s \end{bmatrix}
      $$
      be the Fourier transform of the associated Redheffer
coefficient-matrix symbol.  Then
      the following hold true:
      \begin{enumerate}
   \item
   \begin{equation}  \label{pos1}
       \begin{bmatrix} \sigma'' & 0 \\ 0 & m I_{\widetilde
	   \Delta_{*}} \end{bmatrix} - \widehat S^{*}
	   \begin{bmatrix} \sigma' & 0 \\ 0 & m I_{\widetilde
	       \Delta_{*}} \end{bmatrix} ^{[-1]} \widehat S
	       \ge 0.
 \end{equation}
 Moreover,  $0$ is the right maximal factorable minorant for
 \eqref{pos1} in the following sense: if $\Phi^{*}$ is a strongly
conjugate-analytic
 $\cL\left(\cN, \sbm{ \cK'_{+} \\ \widetilde
\Delta_{*}}\right)$-valued
 measure so that
 \begin{equation}  \label{FacMin}
 \begin{bmatrix} \sigma'' & 0 \\ 0 & m I_{\widetilde
     \Delta_{*}} \end{bmatrix} - \widehat S^{*}
	   \begin{bmatrix} \sigma' & 0 \\ 0 & m I_{\widetilde
	       \Delta_{*}} \end{bmatrix} ^{[-1]} \widehat S
	       \ge \Phi^{*} \frac{1}{m} \Phi,
 \end{equation}
 then $\Phi = 0$.

 \item
 \begin{equation}  \label{pos2}
     \begin{bmatrix} \sigma'' & 0 \\ 0 & mI_{\widetilde
     \Delta_{*}} \end{bmatrix} - \widehat S \begin{bmatrix}
     \sigma'' & 0 \\ 0 & m I_{\widetilde \Delta} \end{bmatrix}
     ^{[-1]} \widehat S^{*} \ge 0.
  \end{equation}
  Moreover, $0$ is the left maximal factorable minorant for
  \eqref{pos2} in the following sense:  if
  If $\Phi_{*}$ is a strongly analytic $\cL\left(\cN_{*},
  \sbm{\cK''_{-} \\ \widetilde \Delta } \right)$-valued measure
  such that
  $$
  \begin{bmatrix} \sigma'' & 0 \\ 0 & mI_{\widetilde  \Delta_{*}}
\end{bmatrix}
 - \widehat S \begin{bmatrix}  \sigma'' & 0 \\ 0 & m I_{\widetilde
\Delta}
   \end{bmatrix}^{[-1]} \widehat S^{*} \ge \Phi_{*} \frac{1}{m}
	     \Phi_{*}^{*},
 $$
 then $\Phi_{*} = 0$.
  \end{enumerate}

  \end{theorem}

  \begin{proof}
      We prove only part (1) as part (2) is similar.
      After interchanging the second and third rows and then the
      second and third columns in \eqref{Sigma0form}, we get
      $$
     \widetilde\Sigma'_0:=\begin{bmatrix}
     \begin{bmatrix} mI_{\widetilde \Delta_{}} & 0 \\ 0 &
    \sigma''\end{bmatrix} & \widehat S' \\
     \widehat S^{\prime *} & \begin{bmatrix} m I_{\widetilde
\Delta_{*}} & 0 \\ 0 &
     \sigma' \end{bmatrix} \end{bmatrix} : =
     \begin{bmatrix}
    \begin{matrix} m I_{\widetilde \Delta_{}} & 0 \\ 0 &
      \sigma''\end{matrix} &
      \begin{matrix} \widehat s & \widehat s_{1} \\ \widehat s_{2} &
\widehat s_{0}
      \end{matrix} \\
      \begin{matrix} \widehat  s^{*} & \widehat s_{2}^{*} \\ \widehat
s_{1}^{*} &
      \widehat s_{0}^{*} \end{matrix}
       & \begin{matrix} m I_{\widetilde \Delta_{*}} & 0 \\ 0 &
      \sigma' \end{matrix} \end{bmatrix} \ge 0,
     $$
     If we then interchange the first two rows and then the first two
     columns and then the last two rows followed by the last two
     columns, we arrive at
     $$
      \widetilde \Sigma_{0} =
     \begin{bmatrix} \begin{bmatrix} \sigma'' & 0 \\ 0 & m
  I_{\widetilde \Delta} \end{bmatrix} &  \widehat S \\ \widehat S^{*}
&
  \begin{bmatrix} \sigma' & 0  \\ 0 & m I_{\widetilde \Delta_{*}}
\end{bmatrix}
     \end{bmatrix} : =
   \begin{bmatrix} \sigma'' & 0 & \widehat s_{0} & \widehat s_{2}
  \\ 0 & m I_{\widetilde \Delta} & \widehat s_{1} & \widehat s
  \\
  \widehat s_{0}^{*} & \widehat s_{1}^{*} & \sigma' & 0 \\
  \widehat s_{2}^{*} & \widehat s^{*} & 0 & m I_{\widetilde
  \Delta_{*}} \end{bmatrix}.
   $$
   Clearly, the positivity of $\widetilde \Sigma_{0}$ is equivalent
   to the positivity of $\widehat \Sigma_{0}$.  By  a standard
   Schur-complement argument, the positivity of $\widetilde
   \Sigma_{0}$ in turn is equivalent to \eqref{pos1}.  Clearly, $0$
   is a left maximal factorable minorant for \eqref{pos1} if and
   only if $0$ is a maximal factorable minorant for the Schur
   complement in $\widetilde \Sigma'_{0}$:
   \begin{equation}  \label{pos1'}
\begin{bmatrix} m I_{\widetilde \Delta_{*}} & 0 \\ 0 & \sigma'
    \end{bmatrix} - \widehat S^{\prime *} \begin{bmatrix} m
    I_{\widetilde \Delta} & 0 \\ 0 & \sigma''
\end{bmatrix}^{[-1]}\widehat S' \ge 0.
    \end{equation}
    By a Schur-complement argument, if in fact \eqref{FacMin} holds
    for a strongly conjugate-analytic $\Phi^{*}$, then
    $$
    \begin{bmatrix}
 m I_{\cN} & \begin{bmatrix} 0 & 0 \end{bmatrix}  & \Phi \\
 \begin{bmatrix} 0 \\ 0 \end{bmatrix} & \begin{bmatrix}
 mI_{\widetilde \Delta} & 0 \\ 0 & \sigma'' \end{bmatrix} &
 \widehat S'  \\
 \Phi^{*} & \widehat S^{\prime *} & \begin{bmatrix}
 mI_{\widetilde \Delta_{*}} & 0 \\ 0 & \sigma' \end{bmatrix}
 \end{bmatrix} \ge 0.
   $$
   By the definition of the Hellinger space, this in turn implies
   that $\sbm{ \sbm{0 \\ 0 } \\ \Phi^{*} } n \in \cL^{\widetilde
   \Sigma'_{0}}$ for every $n \in \cN$.  Since by assumption
   $\Phi^{*}$ is strongly conjugate-analytic, then, as in the proof
   of Theorem \ref{T:MFM}, we have
   $$
   t^{-1} \begin{bmatrix} \begin{bmatrix} 0 \\ 0 \end{bmatrix} \\
   \Phi^{*}n \end{bmatrix} \perp \begin{bmatrix} \begin{bmatrix}
   mI_{\widetilde \Delta} & 0 \\ 0 & \sigma'' \end{bmatrix} &
   \widehat S' \\ \widehat S^{\prime *} & \begin{bmatrix}
   mI_{\widetilde \Delta_{*}} & 0 \\ 0 & \sigma' \end{bmatrix}
\end{bmatrix} \begin{bmatrix} L^{2}_{\widetilde \Delta} \\
\cK''_{-}[t] \\ H^{2}_{\widetilde \Delta_{*}} \\ \cK'_{+}
\end{bmatrix}.
$$
But by \eqref{cK0decom-func}, we know that the linear manifold
$$
\begin{bmatrix} \begin{bmatrix} m I_{\widetilde \Delta} & 0 \\ 0 &
    \sigma'' \end{bmatrix} & \widehat S' \\ \widehat S^{\prime *} &
    \begin{bmatrix} m I_{\widetilde \Delta_{*}} & 0 \\ 0 & \sigma'
 \end{bmatrix} \end{bmatrix} \begin{bmatrix} H^{2
 \perp}_{\widetilde \Delta} \\ \cK''_{-}  \\ H^{2}_{\widetilde
 \Delta_{*}} \\ \cK'_{+} \end{bmatrix}
$$
is already dense in $\cL^{\widetilde \Sigma'_{0}}$.  Therefore
necessarily $\Phi = 0$, completing the promised proof of part (1) of
the Theorem.
  \end{proof}

  We next show that Theorem \ref{T:MFM} leads to the following
characterization of
  central characteristic measures; this version is more intrinsically
  function-theoretic than the characterization given by Theorem
\ref{T:converse2}.

  \begin{theorem} \label{T:converse3}  Suppose that $\widehat
\Sigma_{0}$
     is a strong $\cL(\widetilde \Delta \oplus \widetilde \Delta_{*}
      \oplus \cK''_{-} \oplus \cK'_{+})$-valued measure of the form
      \eqref{Sigma0form}.  Suppose that  $\widehat \sigma_{0} : =
\sbm{ \sigma'' & \widehat
s_{0} \\ \widehat s_{0}^{*} & \sigma' }$ is a positive
      $\cL(\cK''_{-} \oplus \cK'_{+})$-valued measure which determines
      the remaining entries $\widehat s_{1}$,
      $\widehat s_{2}$, $\widehat s$ in $\widehat \Sigma_{0}$
      according to the following procedure:
\begin{enumerate}
   \item $\widehat s_{1}$ is an analytic measure such that
   $\widehat s_{1}^{*} \frac{1}{m} \widehat s_{1}$ is a
   maximal right factorable minorant for $\sigma' - \widehat
   s_{0}^{*} \sigma^{\prime \prime [-1]} \widehat s_{0}$.

   \item $\widehat s_{2}$ is an analytic measure with
   $\widehat s_{2}({\mathbb T}) = 0$ such that
   $(t^{-1}\widehat s_{2}) \frac{1}{m} (t^{-1} \widehat s_{2})^{*}$
is a
   maximal left factorable minorant for $\sigma'' - \widehat
   s_{0} \sigma^{\prime [-1]} \widehat s_{0}^{*}$.

   \item The formula
   \begin{equation}  \label{shat}
   \begin{bmatrix} m I_{\widetilde \Delta} & \widehat s \\
       \widehat s^{*} & m I_{\widetilde \Delta_{*}}
   \end{bmatrix} =
   \begin{bmatrix} 0 &\widehat s_{1} \\ \widehat s_{2}^{*} & 0
       \end{bmatrix} \begin{bmatrix} \sigma'' & \widehat s_{0}
       \\ \widehat s_{0}^{*} & \sigma' \end{bmatrix} ^{[-1]}
       \begin{bmatrix} 0 & \widehat s_{2} \\  \widehat
	   s_{1}^{*} & 0 \end{bmatrix}
   \end{equation}
   holds true for the analytic measure $\widehat s$ which in
   addition satisfies
   \begin{equation}  \label{hats0=0}
     \widehat s({\mathbb T}) = \int_{{\mathbb T}} \widehat
     s(dt) = 0.
   \end{equation}
    \end{enumerate}
  Assume in addition that
  \begin{equation}  \label{MFM-ortho''}
 \left( \cL^{\widehat \sigma_{0}}-\textup{clos } \widehat \sigma_{0}
  \begin{bmatrix}  \cK''_{-}[t] \\ \cK'_{+}
  \end{bmatrix}\right) \cap
  \left(\cL^{\widehat \sigma_{0}}-\textup{clos } \widehat \sigma_{0}
 \begin{bmatrix} \cK''_{-} \\ \cK'_{+}[t^{-1}] \end{bmatrix}
     \right)
     = \cL^{\widehat \sigma_{0}}-\textup{clos } \widehat
     \sigma_{0} \begin{bmatrix} \cK''_{-} \\ \cK'_{+}
 \end{bmatrix}.
  \end{equation}
  Then $\widehat \Sigma_{0}$ is the central characteristic
  measure arising from some Lifting Problem.
  \end{theorem}

  \begin{proof}
     Assume that $\widehat \Sigma_{0}$ has the form
      \eqref{Sigma0form} with $\widehat s_{1}$, $\widehat s_{2}$ and
      $\widehat s_{0}$ determined from $\widehat \sigma_{0} = \sbm{
      \sigma'' & \widehat s_{0} \\ \widehat s_{0}^{*} & \sigma'}$ as
      in the statement of the theorem. We wish to apply the second
      version of Theorem \ref{T:converse1} to conclude that $\widehat
      \Sigma^{0}$ is the central characteristic measure for a Lifting
      Problem.

      From the hypothesis \eqref{shat} combined with the observation
      \eqref{L1=L}, we see that the first hypothesis \eqref{10.6}
      required for application of the second version of Theorem
      \ref{T:converse1} holds.

      The maximal-factorable-minorant properties of $\widehat s_{1}$
      and $\widehat s_{2}$ imply that
      \begin{align}
  & \operatorname{clos} \widehat \sigma_{0} \begin{bmatrix}
   \cK''_{-}[t] \\ \cK'_{+} \end{bmatrix} = \cL^{\widehat
   \sigma_{0}} \ominus \begin{bmatrix} 0 \\ \widehat s_{1}
   H^{2 \perp}_{\widetilde \Delta_{*}} \end{bmatrix}, \notag \\
 &  \operatorname{clos} \widehat \sigma_{0} \begin{bmatrix}
   \cK''_{-} \\ \cK'_{+}[t^{-1}] \end{bmatrix} = \cL^{\widehat
   \sigma_{0}} \ominus \begin{bmatrix} \widehat s_{2} H^{2}
   _{\widetilde \Delta} \\ 0 \end{bmatrix}.
   \label{MFM1}
 \end{align}
   Property \eqref{shat} implies that the map $U_{\widehat
   \sigma_{0}, \widehat \Sigma_{0}}$ given by \eqref{sigma0Sigma0}
   is unitary from $\cL^{\widehat \sigma_{0}}$ onto $\cL^{\widehat
   \Sigma_{0}}$ One can check that
   $$
U_{\widehat \sigma_{0}, \widehat \Sigma_{0}} \colon
\sbm{ 0 \\ \widehat s_{1} H^{2 \perp}_{\widetilde
    \Delta_{*}} } \mapsto \widehat \Sigma_{0}
    \sbm{ H^{2 \perp}_{\widetilde \Delta} \\ 0 \\ 0
	\\ 0 }, \quad
    U_{\widehat \sigma_{0}, \widehat \Sigma_{0}} \colon
    \sbm{ \widehat s_{2} H^{2}_{\widetilde \Delta}
	\\ 0 } \mapsto \widehat \Sigma_{0}
	\sbm{ 0 \\ H^{2}_{\widetilde \Delta_{*}} \\
	    0 \\ 0 }.
 $$
 Hence \eqref{MFM1} transforms to the equivalent condition in
 $\cL^{\Sigma_{0}}$:
 \begin{align}
   &  \operatorname{clos} \widehat \Sigma_{0} \sbm{ 0 \\ 0
     \\ \cK''_{-}[t] \\ \cK'_{+} } = \cL^{\widehat
     \Sigma_{0}} \ominus \widehat \Sigma_{0} \sbm{ H^{2
     \perp}_{\widetilde \Delta} \\ 0 \\ 0 \\ 0 }, \notag
     \\
     & \operatorname{clos} \widehat \Sigma_{0} \sbm{ 0 \\ 0
     \\ \cK''_{-} \\ \cK'_{+}[t^{-1}] } = \cL^{\widehat
     \Sigma_{0}} \ominus \widehat \Sigma_{0} \sbm{ 0 \\
     H^{2}_{\widetilde \Delta_{*}} \\ 0 \\ 0 }.
     \label{MFM2}
  \end{align}
  Conditions \eqref{MFM2} are just the functional-model equivalent of
  conditions \eqref{MFM-ortho}.

  It is easily checked that condition \eqref{MFM-ortho''} is just the
  translation of \eqref{MFM-ortho'} to this functional-model setting.

  Conditions \eqref{analytic} for this
  case can be read off from the analyticity of $t^{-1} \widehat
  s_{2}$, $\widehat s_{1}$ and $\widehat s$, respectively.
  Similarly, conditions \eqref{orthogonal} can be read off from the
  zero appearing in the $(1,3)$ and $(2,4)$ entries of $\widehat
  \Sigma_{0}$, respectively.
  Condition \eqref{hats0=0} can be seen to be equivalent to
  \eqref{10.7weak}. The second version of Theorem
  \ref{T:converse1} applies to lead us to the desired result.
  \end{proof}

  Combining Theorem \ref{T:converse3} with Theorem \ref{T:MFM} leads
  to the following corollary.  A result of this type was obtained by
  Adamjan-Arov-Kre\u{\i}n in the context of the Nehari problem in
  \cite{AAK71}, see also \cite{Kh-exposed}.

  \begin{corollary}  \label{C:MFM}  Suppose that we are given the
      data set
   $$ (\cU', \cK'), \quad (\cU'', \cK''), \quad \cK'_{+} \subset
   \cK', \quad \cK''_{-} \subset \cK''
  $$
  for a Lifting Problem, that we set
  $$
  \sigma'(dt) = (i_{\cK'_{+} \to \cK'})^{*} E_{\cU'}(dt) i_{\cK'_{+}
  \to \cK'}, \quad
  \sigma''(dt) = (i_{\cK''_{-} \to \cK''})^{*} E_{\cU''}(dt)
  i_{\cK''_{-}  \to \cK''}
  $$
  and that we let $\widehat s_{0}$ be an $\cL(\cK'_{+},
  \cK''_{-})$-valued measure so that
  $$
   \sigma_{0}: = \sbm{ \sigma'' & \widehat s_{0} \\ \widehat
   s_{0}^{*} & \sigma' } \ge 0.
  $$
  Construct measures $\widehat s_{1}$ and $\widehat s_{2}$ as maximal
  factorable minorants as in \eqref{factorable1}---\eqref{factor2} in
  Theorem \ref{T:MFM} and define $\widehat s$ as in \eqref{10.6'}.
  Then $\widehat s_{0}$ is the central-measure symbol for the Lifting
  Problem associated with the operator $X = \widehat s_{0}({\mathbb
  T}) \in \cL(\cK'_{+}, \cK''_{-})$ if and only if $\widehat s$ is
  analytic, $\widehat s({\mathbb T}) = 0$ and condition
  \eqref{MFM-ortho''} holds.  In this case $\widehat \Sigma^{0}$ as
  in \eqref{Sigma0form} is the associated central characteristic
  measure.
\end{corollary}

\section{The classical Nehari problem}  \label{S:Nehari}
    In this section we use the classical Nehari problem as an
    illustration of our results for the general Lifting Problem.
   The classical Nehari problem
   (time-domain version) can be stated as follows: {\em  given
   a sequence $\{ \gamma_{n}\}_{n =-1, -2, \dots }$ of complex
   numbers indexed by the negative integers, characterize all
   continuations
   $\{ \gamma_{n}\}_{n=0,1,2, \dots}$ so
   that the bi-infinite Toeplitz matrix $\Gamma_{e} =
   [\gamma_{i-j}]_{i,j \in {\mathbb Z}}$ has $\|\Gamma_{e}\| \le 1$
   as an operator on $\ell^{2}({\mathbb Z})$}. This can be put
   in the form of a Lifting Problem with data set
   \begin{align*}
  & \cK' = \cK'' = \ell^{2}({\mathbb Z}), \, \cU' = \cU'' = J
   \text{ (the bilateral shift), } \cK'_{+} = \ell^{2}({\mathbb
   Z}_{+}), \, \cK''_{-} = \ell^{2}({\mathbb Z}_{-}), \\
   & X = [ \gamma_{i-j} ]_{i \in {\mathbb Z}_{-}, j \in {\mathbb
Z}_{+}}
   \colon \ell^{2}({\mathbb Z}_{+}) \to \ell^{2}({\mathbb Z}_{-}).
   \end{align*}
   The equivalent frequency-domain version is: {\em  given the
   complex sequence $\{\gamma_{n}\}_{n \in {\mathbb Z}_{-}}$,
   characterize all
   $L^{\infty}$-functions $\varphi$ on the unit circle ${\mathbb T}$
   with Fourier series representation $\varphi(t) =
   \sum_{n=-\infty}^{\infty} \varphi_{n} t^{n}$ so that}
   $$
   \| \varphi\|_{\infty} \le 1 \text{ and } \varphi_{n} = \gamma_{n}
   \text{ for } n = -1, -2, -3, \dots.
   $$

   If we follow the conventions of Remark \ref{R:param}, then
   the symbol $\{w_{Y}\}_{n \in {\mathbb Z}}$ for a solution $Y$ of
   the time-domain problem is connected with the corresponding
   solution $\varphi$ of the frequency-domain problem according to
   the formula
   $$
    \varphi(t) = \sum_{n = -\infty}^{\infty} w_{Y}(n+1) t^{n},
   $$
   i.e., if we let $\widehat w_{Y}(t) = \sum_{n=-\infty}^{\infty}
   w_{Y}(n) t^{n}$, then $\varphi(t) = t^{-1} \widehat w_{Y}(t)$.
   The Redheffer parametrization for the set of all solutions of the
   frequency-domain problem then has the form
   $$
   \varphi(t) = t^{-1} s_{0}(t) + t^{-1} s_{2}(t)(1 - \omega(t)
   s(t))^{-1} \omega(t) s_{1}(t)
   $$
   for a free-parameter Schur-class function  $\omega$,
   where $s_{0},s_{1},s_{2},s$ are as in \eqref{def-Redheffer}.  If
   we set
   \begin{equation} \label{adjust}
    \ts_{0}(t) = t^{-1}s_{0}(t), \quad \ts_{2}(t) = t^{-1} s_{2}(t),
\quad
   \ts(t) = s(t), \quad \ts_{1}(t) = s_{1}(t),
   \end{equation}
   then the formula
   \begin{equation}  \label{param-Nehari}
   \varphi(t) = \ts_{0}(t) + \ts_{2}(t) (1 - \omega(t) \ts(t))^{-1}
   \omega(t) \ts_{1}(t)
   \end{equation}
   becomes the parametrization of the set of all solutions of the
   frequency-domain classical Nehari problem associated with the
   sequence $\{ \gamma_{n} = [\ts_{0}]_{n}\}_{n = -1, -2, \dots}$
   where $\ts_{0}(t) = \sum_{n=-\infty}^{\infty} [\ts_{0}]_{n} t^{n}$
   is the {\em central solution}.  If we model $\cK'$ and $\cK''$ as
   $L^{2}$ with both scale operators given by $c \in {\mathbb C}
\mapsto
   c \in L^{2}$, then the central characteristic measure given by
   \eqref{central-charmeas} becomes a multiple of Lebesgue measure:
   $$
   \widehat \Sigma_{0} = \begin{bmatrix} 1 & \ts & 0 & \ts_{1} \\
   \ts^{*} & 1 & \ts_{2}^{*} & 0 \\
   0 & \ts_{2} & 1 & \ts_{0} \\ \ts_{1}^{*} & 0 & \ts_{0}^{*} & 1
\end{bmatrix} \cdot m.
$$
In this section we simplify notation and write simply $\widehat\Sigma_{0}$
 for the density against Lebesgue measure (i.e., we drop the $\cdot
m$ factor).

With all these conventions in order, the inverse problem can be formulated
simply as: {\em characterize which $2\times 2$ $L^{\infty}$-matrices
$\sbm{ \ts & \ts_{1} \\ \ts_{2}
  & \ts_{0} }$ arise as the Redheffer coefficient matrix for a
  frequency-domain classical Nehari problem.}  Obviously a
  necessary condition is that $\|\ts_{0}\|_{\infty} \le 1$.  In case
  $\log ( 1 - |\ts_{0}(t)|^{2})$ is not integrable with respect to
  Lebesgue measure over ${\mathbb T}$, then the solution of the
  associated Nehari problem is unique; hence, for $\ts_{0}$ to be the
  central solution for an indeterminate problem, it is also
  necessary that $\log( 1 - |\ts_{0}(t)|^{2})$ be integrable.  In
  this case, it is well known (see e.g.~\cite{Hoffman}) that there is
  a unique outer function $a$ with $a(0)>0$ so that
  \begin{equation} \label{a-fact} |a(t)|^{2} = 1 -
  |\ts_{0}(t)|^{2} \text{ a.e. for $t \in {\mathbb T}$.}
  \end{equation}
  From  Theorem
  \ref{T:MFM} (with the appropriate adjustments caused by
  \eqref{adjust}), we see that then necessarily $\ts_{0}$ determines
  $\ts_{1}$ and $\ts_{2}$ uniquely up to unimodular constant factors
  \begin{equation}  \label{a}
    \ts_{1} = \ts_{2} = a.
  \end{equation}
  According to formula \eqref{shat} in Theorem \ref{T:converse3},
  $\ts$ is then essentially uniquely determined by the condition
  \begin{align*}
  \begin{bmatrix} 1 & \ts \\ {\ts}^{*} & 1 \end{bmatrix} & =
     \begin{bmatrix} 0 & a \\ \overline{a} & 0 \end{bmatrix}
      \begin{bmatrix} 1 & \ts_{0} \\ \ts_{0}^{*} & 1 \end{bmatrix}
^{-1}
\begin{bmatrix} 0 & a \\ \overline{a} & 0 \end{bmatrix} \\
    & = \begin{bmatrix} 0 & a \\ \overline{a} & 0 \end{bmatrix}
      \begin{bmatrix} \frac{1}{\overline{a} a} &
-\frac{ \ts_{0}^{*}}{\overline{a} a} \\
-\frac{\ts_{0}}{\overline{a} a} & \frac{1}{\overline{a} a}
   \end{bmatrix}
  \begin{bmatrix} 0 & a \\ \overline a & 0 \end{bmatrix}
       \text{ (where we use $1 - |\ts_{0}|^{2} = \overline{a}
       a$)} \\
       & = \begin{bmatrix} 1 & -\frac{a}{\overline{a}} \ts_{0}^{*} \\
       - \frac{\overline{a}}{a} \ts_{0} & 1 \end{bmatrix}
 \end{align*}
  from which we read off that necessarily
 \begin{equation}   \label{s}
     \ts = - \frac{a}{\overline{a}} \ts_{0}^{*}=:b.
  \end{equation}
  For our setting here (with coefficient spaces chosen to be
  ${\mathbb C}$ rather than the whole spaces $\cK'_{+} =
  \ell^{2}({\mathbb Z}_{+})$ and $\cK''_{-} = \ell^{2}({\mathbb
  Z}_{-})$), condition \eqref{hats0=0} translates simply to
  \begin{equation}  \label{ts(0)=0}
      \ts(0) = 0.
   \end{equation}
   A pair $(a,b)$ that satisfies the above properties, namely
   \begin{enumerate}
   \item $a,b\in H^\infty$,
   \item $a$ is outer, $b(0)=0$, and
   \item $|a|^2+|b|^2=1$ almost everywhere,
   \end{enumerate}
   is called a $\gamma$-{\sl generating pair} in the sense that any
function $\varphi$
   of the form \eqref{param-Nehari}  with $\ts, \ts_{1}, \ts_{2},
   \ts_{0}$ given by
   \begin{equation}  \label{tS}
   \widetilde S=\begin{bmatrix} \ts & \ts_{1} \\ \ts_{2} & \ts_{0}
\end{bmatrix} =
       \begin{bmatrix} b & a \\ a
	   &  - \frac{a}{\overline{a}} \, \overline{b} \end{bmatrix}
  \end{equation}
   has $\| \varphi\|_{\infty} \le 1$
   with negative Fourier coefficients $[\varphi]_{-1}, [\varphi]_{-2},
   \dots $ independent of the choice of free-parameter Schur-class
   function $\omega$.

   If it is the case that the formula \eqref{param-Nehari}
   parametrizes the set of {\em all} solutions of a Nehari problem,
   it is then said that $(a,b)$ is a {\em Nehari pair}. It is known
that
   not every $\gamma$-generating pair is a Nehari pair---see
   \cite{Kh-IEOT95, Kh-regulariz, Kh-JFA96} for background on this
   and for some more refined results. The following characterization
of Nehari
pairs was obtained in \cite{Kh-regulariz, Kh-exposed}.  We now show
how the result can be obtained as a corollary of Theorem
\ref{T:converse2}.

\begin{theorem} \label{T:Nehari}
\cite{Kh-regulariz, Kh-exposed}
A $\gamma$-generating pair $(a,b)$ is
    a Nehari pair if and only if $($in addition to conditions
$(1)-(3)$
    above$)$ $(a,b)$ satisfies the fourth condition
    \begin{equation}   \label{4}
	\begin{bmatrix} \overline{b} \\ P_{H^{2 \perp}} \overline{a}
	\end{bmatrix} \in \operatorname{clos} \left\{ \begin{bmatrix}
	\overline{a} h \\ P_{H^{2 \perp}} \overline{\ts}_{0} h
    \end{bmatrix} \colon h \in H^{2 \perp} \right\}
    \end{equation}
    where $\overline{\ts}_{0}: = -\frac{a}{\overline{a}}\overline{b}$.
    \end{theorem}

    \begin{proof}
   By Theorem \ref{T:converse2} above, properties
   \eqref{10.7'} complete the list of necessary and sufficient
conditions
   for a given $\gamma$-generating pair $(a,b)$ to be a Nehari pair.
   When specialized to the Nehari problem, they read as follows:
   a $\gamma$-generating pair $(a,b)$ is a Nehari pair if and only if
   \begin{equation}\label{10.7''}
\begin{bmatrix} 1  \\
   \overline b  \\
   0  \\ \overline a
\end{bmatrix}\mathbb C
=   \operatorname{clos} \begin{bmatrix} 1 & b & 0 & a \\
   \overline b & 1 & \overline a & 0 \\
   0 & a & 1 & \ts_{0} \\ \overline a & 0 & \overline\ts_{0} & 1
\end{bmatrix}
\begin{bmatrix} 0 \\ 0\\ H^{2\perp} \\ H^2 \end{bmatrix}  \ominus
 \operatorname{clos} \begin{bmatrix} 1 & b & 0 & a \\
   \overline b & 1 & \overline a & 0 \\
   0 & a & 1 & \ts_{0} \\ \overline a & 0 & \overline\ts_{0} & 1
\end{bmatrix}
\begin{bmatrix} 0 \\ 0\\ H^{2\perp} \\ t H^2 \end{bmatrix}
\end{equation}
and
\begin{equation}\label{10.7'''}
\overline t \begin{bmatrix} b\\ 1 \\ a \\ 0 \end{bmatrix}\mathbb C =
    \operatorname{clos} \begin{bmatrix} 1 & b & 0 & a \\
     \overline b & 1 & \overline a & 0 \\
      0 & a & 1 & \ts_{0} \\ \overline a & 0 & \overline\ts_{0} & 1
\end{bmatrix}
\begin{bmatrix} 0 \\ 0\\ H^{2\perp} \\ H^2 \end{bmatrix}
\ominus
 \operatorname{clos} \begin{bmatrix} 1 & b & 0 & a \\
   \overline b & 1 & \overline a & 0 \\
   0 & a & 1 & \ts_{0} \\ \overline a & 0 & \overline\ts_{0} & 1
\end{bmatrix}
\begin{bmatrix} 0 \\ 0\\ \overline t H^{2\perp} \\  H^2 \end{bmatrix}.
\end{equation}
 Since for every $\gamma$-generating pair both sides of \eqref{10.7''}
 are of dimension one and
\begin{equation*} 
\begin{bmatrix} 1  \\
   \overline b  \\
   0  \\ \overline a
\end{bmatrix}
\perp
\begin{bmatrix} 1 & b & 0 & a \\
   \overline b & 1 & \overline a & 0 \\
   0 & a & 1 & \ts_{0} \\ \overline a & 0 & \overline\ts_{0} & 1
\end{bmatrix}
\begin{bmatrix} 0 \\ 0\\ H^{2\perp} \\ t H^2 \end{bmatrix},
\end{equation*}
then (for $\gamma$-generating pairs) \eqref{10.7''} is equivalent to
\begin{equation} \label{1000}
\begin{bmatrix} 1  \\
   \overline b  \\
   0  \\ \overline a
\end{bmatrix}
\in
\operatorname{clos}
\begin{bmatrix} 1 & b & 0 & a \\
   \overline b & 1 & \overline a & 0 \\
   0 & a & 1 & \ts_{0} \\ \overline a & 0 & \overline\ts_{0} & 1
\end{bmatrix}
\begin{bmatrix} 0 \\ 0\\ H^{2\perp} \\ H^2 \end{bmatrix}.
\end{equation}
By a similar argument we see that \eqref{10.7'''} is equivalent to
\begin{equation}  \label{1000'}
    \overline{t} \begin{bmatrix}  b \\ 1 \\ a \\ 0 \end{bmatrix} \in
    \operatorname{clos} \begin{bmatrix} 1 & b & 0 & a \\
   \overline b & 1 & \overline a & 0 \\
   0 & a & 1 & \ts_{0} \\ \overline a & 0 & \overline\ts_{0} & 1
\end{bmatrix}
\begin{bmatrix} 0 \\ 0\\ H^{2\perp} \\ H^2 \end{bmatrix}.
\end{equation}

We first note some general principles concerning the space
$\cL^{\widehat \Sigma_{0}}$. Since $\widetilde S=\left[
\begin{smallmatrix} b & a
\\ a & \ts_{0} \end{smallmatrix}\right](t)$ is unitary for almost all
$t \in {\mathbb T}$, it can be seen as a consequence of
Theorem~\ref{T:Hellinger-decom}
and some Schur-complement computations that
$f=\left[\begin{smallmatrix} f_1 \\ f_2\\ f_3\\ f_4
\end{smallmatrix} \right]$ is in $\cL^{ \widehat \Sigma_{0}}$
if and only if
\begin{equation}\label{f-in-space}
f\in L^2 \text{\ and\ }
\begin{bmatrix} f_1\\ f_3 \end{bmatrix}=
\widetilde S\begin{bmatrix} f_2 \\ f_4 \end{bmatrix}.
\end{equation}
This, in particular, implies that for every $f\in \cL^{ \widehat
\Sigma_{0}}$
we have
$$
f = \begin{bmatrix} f_{1} \\ f_{2} \\ f_{3} \\
 f_{4} \end{bmatrix}
 = \widehat \Sigma_{0}
\begin{bmatrix} 0 \\ f_{2} \\ 0 \\
 f_{4} \end{bmatrix}
= \widehat \Sigma_{0}
 \begin{bmatrix} f_{1} \\ 0 \\ f_{3} \\
 0 \end{bmatrix}
.
$$
 It then follows that
 \begin{equation}  \label{LSigma0-id}
  \|f \|_{\cL^{ \widehat \Sigma_{0}}}=
 \left\| \begin{bmatrix} f_2\\ f_4 \end{bmatrix} \right\|_{L^2}
 \text{ and }
 \|f \|_{\cL^{ \widehat \Sigma_{0}}}=
 \left\| \begin{bmatrix} f_1\\ f_3 \end{bmatrix} \right\|_{L^2}.
 \end{equation}

 From the first equality in \eqref{LSigma0-id} we see that
\eqref{1000} is equivalent to
 \begin{equation} \label{1001}
 \begin{bmatrix}
   \overline b  \\
   \overline a
\end{bmatrix}
\in \operatorname{clos}
\begin{bmatrix}
   \overline a & 0 \\
 \overline\ts_{0} & 1
\end{bmatrix}
\begin{bmatrix}  H^{2\perp} \\ H^2 \end{bmatrix}.
\end{equation}
where now this closure and all closures to follow are computed in the
$L^2$-metric.
The latter condition is equivalent to
\begin{equation} \label{1002}
 \begin{bmatrix}
   \overline b  \\ \overline a - \overline {a (0}) \end{bmatrix} \in
\operatorname{clos}
P_{H^{2\perp}}
\begin{bmatrix}
   \overline a \\
 \overline\ts_{0}
\end{bmatrix}
H^{2\perp}.
\end{equation}
From the second equality in \eqref{LSigma0-id} we see that
\eqref{1000} is also equivalent to
 \begin{equation} \label{1005}
 \begin{bmatrix} 1  \\ 0 \end{bmatrix}\in
\operatorname{clos}
\begin{bmatrix} 0 & a \\ 1 &\ts_{0}
\end{bmatrix}
\begin{bmatrix}  H^{2\perp} \\ H^2 \end{bmatrix}.
\end{equation}
which in turn is equivalent to
\begin{equation} \label{1006}
 \begin{bmatrix}
   1  \\  0
\end{bmatrix} \in
\operatorname{clos} P_{H^2}
\begin{bmatrix} a \\ \ts_{0} \end{bmatrix} H^2.
\end{equation}

Similarly, the first of equations \eqref{LSigma0-id} converts
\eqref{10.7'''}
to the equivalent form
\begin{equation} \label{1003}
 \overline t \begin{bmatrix} b  \\ a \end{bmatrix} \in
\operatorname{clos} \begin{bmatrix}  0 & a \\ 1 & \ts_{0}
\end{bmatrix}
\begin{bmatrix}  H^{2\perp} \\  H^2 \end{bmatrix}.
\end{equation}
 which in turn can be rewritten as
\begin{equation} \label{1004}
 \begin{bmatrix} b/t  \\ \frac{a - a (0)}{t} \end{bmatrix}
\in
\operatorname{clos}
P_{H^2}
\begin{bmatrix} a \\ \ts_{0} \end{bmatrix} H^2,
\end{equation}
while the second of equations \eqref{LSigma0-id} can be used to
convert \eqref{10.7'''} to the seemingly different equivalent form
\begin{equation} \label{1007}
 \begin{bmatrix} \overline t  \\ 0 \end{bmatrix}
\in  \operatorname{Clos}
\begin{bmatrix}
   \overline a & 0\\
\overline \ts_{0} & 1
\end{bmatrix}
\begin{bmatrix}  H^{2\perp} \\  H^2 \end{bmatrix}.
\end{equation}
 which in turn can be rewritten as
\begin{equation} \label{1008}
 \begin{bmatrix}
\overline t  \\ 0
\end{bmatrix}
\in
\operatorname{clos}
P_{H^{2\perp}}
\begin{bmatrix}
 \overline   a \\
\overline \ts_{0}
\end{bmatrix}
H^{2\perp}.
\end{equation}

Next observe that \eqref{1003} is just the complex-conjugate version
of \eqref{1001} and hence \eqref{1003} and \eqref{1001} are
equivalent.
We conclude that in fact \eqref{10.7''} and \eqref{10.7'''} are
equivalent to each other.  Alternatively, to arrive at the same
result, observe that \eqref{1005} and \eqref{1007} are
complex-conjugate versions of each other, from which it follows that
\eqref{10.7''} and \eqref{10.7'''} are equivalent to each other.

Thus the $\gamma$-generating pair is a Nehari pair exactly when any
one of the equivalent conditions \eqref{1002}, \eqref{1006},
\eqref{1004},
or \eqref{1008} holds.  In particular, we choose condition
\eqref{1002} to arrive at condition \eqref{4} in Theorem
\ref{T:Nehari}.
\end{proof}

\begin{remark} \label{R:Neharipair-equiv} {\em
  It is known (see \cite[Section
   6]{Sarason2}) that $\gamma$-generating pairs $(a,b)$
   are in one-to-one correspondence with extreme points of the unit
ball of $H^{1}$
   via $f =\left(
   a/(1-b)\right)^2$ and via the same formula Nehari pairs
   are in one-to-one correspondence with exposed points for the unit
ball of $H^{1}$.
   The above characterization of Nehari pairs is, at the same
   time, the best known characterization of exposed points.
   }\end{remark}

\begin{remark}\label{R:Neharipair-equiv1}{\em
   We note also that (for $\gamma$-generating pairs)
\eqref{10.7''}/\eqref{10.7'''} is equivalent to
\begin{equation}\label{Nehari-subspace-eq}
 \cS := \begin{bmatrix} 1  \\
   \overline b  \\
   0  \\ \overline a
\end{bmatrix}
H^{2\perp}
\oplus
\operatorname{clos} \begin{bmatrix} 1 & b & 0 & a \\
   \overline b & 1 & \overline a & 0 \\
   0 & a & 1 & \ts_{0} \\ \overline a & 0 & \overline\ts_{0} & 1
\end{bmatrix}
\begin{bmatrix} 0 \\ 0\\ H^{2\perp} \\ H^2 \end{bmatrix}
\oplus
\begin{bmatrix} b\\
   1 \\
   a \\ 0
\end{bmatrix}
H^2
= L^{ \widehat \Sigma_{0}},
\end{equation}
which is \eqref{cK0decom-func}.
Indeed, if \eqref{10.7''}/\eqref{10.7'''} hold true, then
(as in Theorems \ref{T:converse1} and \ref{T:converse2})
$\cS$ is a
$*$-cyclic
reducing subspace, therefore, it agrees with $L^{ \widehat
\Sigma_{0}}$. Conversely,
assume that equality in \eqref{Nehari-subspace-eq} holds.
We check directly that
\begin{equation} \label{direct-check}
t\begin{bmatrix} 1  \\
   \overline b  \\
   0  \\ \overline a
\end{bmatrix}
\overline t
=
\begin{bmatrix} 1  \\
   \overline b  \\
   0  \\ \overline a
\end{bmatrix}
\perp
\begin{bmatrix} 1  \\
   \overline b  \\
   0  \\ \overline a
\end{bmatrix}
H^{2\perp}
\oplus
\operatorname{clos} \begin{bmatrix} 1 & b & 0 & a \\
   \overline b & 1 & \overline a & 0 \\
   0 & a & 1 & \ts_{0} \\ \overline a & 0 & \overline\ts_{0} & 1
\end{bmatrix}
\begin{bmatrix} 0 \\ 0\\ H^{2\perp} \\ t H^2 \end{bmatrix}
\oplus
\begin{bmatrix} b\\
   1 \\
   a \\ 0
\end{bmatrix}
H^2.
\end{equation}
From the decomposition \eqref{Nehari-subspace-eq} we conclude that
$$\begin{bmatrix} 1  \\
   \overline b  \\
   0  \\ \overline a
\end{bmatrix} \in \operatorname{clos} \begin{bmatrix} 1 & b & 0 & a
\\ \overline{b} & 1 & \overline{a} & 0 \\ 0 & a & 1 & \ts_{0} \\
\overline{a} & 0 & \overline{\ts}_{0} & 1  \end{bmatrix}
\begin{bmatrix} 0 \\ 0 \\ H^{2 \perp} \\ H^{2} \end{bmatrix}.
$$
Combining this with \eqref{direct-check} we see that
\begin{equation}\nonumber
t\begin{bmatrix} 1  \\
   \overline b  \\
   0  \\ \overline a
\end{bmatrix}
\overline t
=
\begin{bmatrix} 1  \\
   \overline b  \\
   0  \\ \overline a
\end{bmatrix}
\in
 \operatorname{clos} \begin{bmatrix} 1 & b & 0 & a \\
   \overline b & 1 & \overline a & 0 \\
   0 & a & 1 & \ts_{0} \\ \overline a & 0 & \overline\ts_{0} & 1
\end{bmatrix}
\begin{bmatrix} 0 \\ 0\\ H^{2\perp} \\ H^2 \end{bmatrix}
\ominus
 \operatorname{clos} \begin{bmatrix} 1 & b & 0 & a \\
   \overline b & 1 & \overline a & 0 \\
   0 & a & 1 & \ts_{0} \\ \overline a & 0 & \overline\ts_{0} & 1
\end{bmatrix}
\begin{bmatrix} 0 \\ 0\\ H^{2\perp} \\ t H^2 \end{bmatrix}
.
\end{equation}
This implies the equality in \eqref{10.7''} since both sides there
are of dimension one.
By a similar argument we get the equality in \eqref{10.7'''}.

The equality in \eqref{Nehari-subspace-eq} can be restated
equivalently as follows:
$f\in \cL^{ \widehat \Sigma_{0}}$ is orthogonal to ${\mathcal S}$, if
and
only if $f=0$. On the other hand, $f\in \cL^{ \widehat \Sigma_{0}}$ is
orthogonal to ${\mathcal S}$, if and
only if
$f_1, f_3$ are in $H^2$ and $f_2, f_4$ are in $H^{2\perp}$. In view
of \eqref{f-in-space}
this means that equality in \eqref{Nehari-subspace-eq} is equivalent
to the property that
the equation
\begin{equation}\label{dense-Nehari}
\widetilde S^*\begin{bmatrix} u_+\\v_+\end{bmatrix}=
\begin{bmatrix} u_-\\v_-\end{bmatrix} \text{ with }
\begin{bmatrix} u_+\\v_+\end{bmatrix}\in H^2 \text{ and }
\begin{bmatrix} u_-\\v_-\end{bmatrix}\in H^{2\perp}
\end{equation}
has only the trivial solution, where $\widetilde S^*$ is as in
\eqref{tS}.
Thus, we get an alternative characterization
of Nehari pairs (that also was obtained in \cite{Kh-regulariz,
Kh-exposed}):}
\begin{theorem}\label{T:Nehari1}
\cite{Kh-regulariz, Kh-exposed}
{\em A $\gamma$-generating pair is a Nehari pair if and only if
\eqref{dense-Nehari}
has only the trivial solution.}
\end{theorem}
 \end{remark}

   \begin{remark} \label{R:relaxed} {\em

       In the context of the scalar Nehari problem, the
       parametrization formula \eqref{FFparam-final} reads in our terms as
     \begin{equation}\label{1011}
 M_{w}|_{H^{2}}  =  \Gamma +  \Phi_{22}\sqrt{I-\Gamma^*\Gamma}
    +\Phi_{21}\omega
(I-\Phi_{11}\omega)^{-1}\Phi_{12}\sqrt{I-\Gamma^*\Gamma}
\end{equation}
where $\Gamma$ is the given Hankel operator
from $H^2$ to $H^{2 \perp}$,  $M_{w}|{H^2} \colon H^2 \to L^2$ is the
restriction to $H^2$ of the multiplication operator $M_w \colon f(t)
\mapsto w(t) f(t)$ on $L^2$ and $w$ is a solution of the given
Nehari problem,  and $\omega$ is the free-parameter
Schur-class function (see \eqref{FFparam'} and \eqref{FFparam-final}).
The   Redheffer coefficient matrix
       $\Psi$ of \cite{tH-pre} (see (\ref{1012})) simplifies to
       \begin{equation}   \label{initial-input/output-1}
	\Psi=
\begin{bmatrix} \Psi_{11} & \Psi_{12}\\
\Psi_{21} & \Psi_{22}\end{bmatrix} \colon
\begin{bmatrix} h_1 \\ \sqrt{I-\Gamma^*\Gamma}h_2 \end{bmatrix} \to
\begin{bmatrix} \widetilde s h_1 + \widetilde s_1 h_2 \\
\widetilde s_2 h_1 + P_{H^{2}}\widetilde s_0 h_2
 \end{bmatrix},\quad h_1, h_2\in H^2;
	\end{equation}
$\ts, \ts_{1},
\ts_{2}, \ts_{0}$ arise from a $\gamma$-generating pair $(a,b)$ as in
\eqref{tS}.
From the unitary property of
$\left[\begin{smallmatrix} \widetilde s & \widetilde s_{1} \\
\widetilde s_{2} & \widetilde s_{0} \end{smallmatrix} \right]$ as a
multiplication operator on $L^{2} \oplus L^{2}$ and the fact that
$P_{H^{2 \perp}} \widetilde s_{0}|_{H^{2}} = \Gamma$, it follows that
$\Psi$ is isometric as an operator  from
$\left[\begin{smallmatrix} H^{2} \\
\overline{\operatorname{Ran}}(I-\Gamma^*\Gamma)
 \end{smallmatrix} \right]
=\left[\begin{smallmatrix} H^{2} \\ H^{2} \end{smallmatrix} \right]$
(the equality
holds since the problem is indeterminate)
to $\left[\begin{smallmatrix} H^{2} \\ H^{2} \end{smallmatrix}
\right]$.
Hence $\Psi$ is also coisometric exactly when the set
\begin{equation}  \label{rangeset}
 \left\{ \begin{bmatrix}
\widetilde s h_1 + \widetilde s_1 h_2 \\
\widetilde s_2 h_1 + P_{H^{2}}\widetilde s_0 h_2
 \end{bmatrix} \colon h_{1}, h_{2} \in H^{2}   \right\}
\end{equation}
is dense in $H^{2} \oplus H^{2}$.

Note, that $H^2$ in the first entry of the initial space for $\Psi$
and $H^2$ in both entries of the target space for $\Psi$ should be
understood
as $H^2(\mathbb C)$ since $\cD_{\rho^*}$, $\cG$ and $\cD_{T'}$ are
all of dimension one,
while $H^2$ in the second entry of the initial space for $\Psi$ is
$\cD_\Gamma$.

Conversely, given any matrix
$\widetilde S=\begin{bmatrix} \ts & \ts_{1} \\ \ts_{2} & \ts_{0}
\end{bmatrix}$
arising from a $\gamma$-generating pair $(a,b)$ as in
\eqref{tS}.
Define a Hankel operator as
$\Gamma=P_{H^{2 \perp}} \widetilde s_{0}|_{H^{2}}$ and
a matrix
$$\Psi=
\begin{bmatrix} \Psi_{11} & \Psi_{12}\\
\Psi_{21} & \Psi_{22}\end{bmatrix}
 : H^{2} \oplus H^{2}\to H^{2} \oplus H^{2}$$
as in \eqref{initial-input/output-1}. Then $\Psi$ is an
isometry.
Assume that \eqref{rangeset} is dense in $H^{2} \oplus H^{2}$.
Then $\Psi$ is a coisometry. Hence, Theorem 0.3
of \cite{tH-pre} applies to this $\Psi$. The theorem tells us that $\Psi$
is the Redheffer coefficient matrix as in \eqref{FFparam-final} for some
(in general not unique) relaxed commutant lifting problem. The data of
such a problem presented in
      the proof of Theorem 0.3 in \cite{tH-pre} is as follows:
      $A: \mathbb C\oplus H^2\to \mathbb C$ is a projection on the
first
      entry, $T'=0$ on $\mathbb C$, $R$ is a certain contraction from
a
      certain subspace $\cF\subset H^2$ to $\mathbb C\oplus H^2$ and
      $Q$ is an embedding of $\cF$ into $\mathbb C\oplus H^2$.

However, the density property of \eqref{rangeset}
can be equivalently formulated as:
{\em the equation
\begin{equation}\nonumber
\widetilde S^*\begin{bmatrix} u_+\\v_+\end{bmatrix}=
\begin{bmatrix} u_-\\v_-\end{bmatrix},\quad
\begin{bmatrix} u_+\\v_+\end{bmatrix}\in H^2,
\begin{bmatrix} u_-\\v_-\end{bmatrix}\in H^{2\perp}
\end{equation}
has only the trivial solution}. In view of Theorem \ref{T:Nehari1}
above, the latter is equivalent to the property that $\widetilde S$ is the
Redheffer coefficient matrix for a Nehari problem. We conclude that:
{\em the operator $\Psi$ \eqref{initial-input/output-1} is coisometric
} (and hence unitary) {\em if and only if $\left[ \begin{smallmatrix}
\ts & \ts_{1} \\ \ts_{2} & \ts_{0} \end{smallmatrix}
\right]$ is the Redheffer coefficient matrix} (in the sense of the
present paper) {\em for  a Nehari problem.}  Thus, in this case the RCL
problem can be taken to have the special form of a Nehari problem.
In this way we arrive at an improved version (in
the context of the Nehari problem) of the result from \cite{tH-pre}.
}\end{remark}


\begin{thebibliography}{99}

\bibitem{Adamjan-Arov} V.M.~Adamjan and D.Z.~Arov, On unitary coupling
of semiunitary operators, {\em Dokl. Akad. Nauk. Arm. SSR}
\textbf{XLIII}, 5
(1966), 257-263; English translation: {\em
Amer.~Math.~Soc.~Transl.} \textbf{95} (1970), 75-129.

\bibitem{AAK68}  V.M.~Adamjan, D.Z.~Arov and M.G.~Kre\u{\i}n, Infinite
Hankel matrices and generalized problems of Carath\'eodory-Fej\'er and
I.~Schur, {\em Funkcional.~Anal.~Prilozhen.} \textbf{2} (1968) no. 4,
1--17; English translation:  {\em Functional Anal.~Appl.} \textbf{2}
(1968) no.~2, 269--281.

\bibitem{AAK71} V.M.~Adamjan, D.Z.~Arov and M.G.~Kre\u{\i}n, Infinite
Hankel block matrices and related extension problems, {\em
Izv.~Akad.~Nauk Armjan.~SSR Ser.~Mat.} \textbf{6} (1971), 87--112;
English translation:  {\em Amer.~Math.~Soc.~Transl.} (2) \textbf{111}
(1978), 133--156.

\bibitem{Arocena1} R.~Arocena, Unitary colligations and
parametrization formulas, {\em Ukrainian Math Zh.} \textbf{46} (no.
3) (1994),
147-154; English translation:  {\em Ukrainian Math.~J.}~\textbf{46}
(1994),
151-158.

\bibitem{Arocena2}  R.~Arocena, Unitary extensions of isometries and
contractive intertwining dilations, in {\em The Gohberg Anniversary
Collection: Volume II Topics in Analysis and Operator Theory} (Ed.
H.~Dym, S.~Goldberg, M.A.~Kaashoek and P.~Lancaster), pp. 13-23, OT
41, Birkh\"auser-Verlag, Basel-Boston, 1989.

\bibitem{ACS}  R.~Arocena, M.~Cotlar and C.~Sadosky, Weighted
inequalities in $L^{2}$ and lifting properties, {\em Mathematical
Analysis and Applications}, Adv.~Math.~Suppl.~Stud.~7A (1981), 95-128.

\bibitem{Arov90}  D.Z.~Arov, Regular $J$-inner matrix-functions and
related continuation problems, {\em Linear Operators in Functions
Spaces} (ed. H.~Helson, B.~Sz.-Nagy and F.-H.~Vasilescu), pp. 63--87,
\textbf{OT 43} Birkh\"auser-Verlag, Basel-Berlin-Boston, 1990.

\bibitem{AG1}
D.Z.~Arov and L.Z.~Grossman,
\newblock  Scattering matrices in the theory of unitary
extensions of isometric operators,
{\em Soviet Math.~Dokl.}, \textbf{270} (1983), 17--20.

\bibitem{AG2}
D.Z.~Arov and L.Z.~Grossman,
  Scattering matrices in the theory of unitary
extensions of isometric operators,
{\em Math.~Nachrichten}, \textbf{157} (1992), 105--123.

\bibitem{BCU} J.A.~Ball, P.T.~Carroll and Y.~Uetake, Lax-Phillips
scattering theory and well-posed linear systems:  a coordinate-free
approach, {\em Math. Control Signals Syst.}  \textbf{20} (2008),
37--79.


\bibitem{BtH} J.A.~Ball and S.~ter Horst, Multivariable
operator-valued Nevanlinna-Pick interpolation: a survey, preprint,
available at arxiv: math/0808.24557.

\bibitem{BLTT} J.A.~Ball, W.S.~Li, D.~Timotin and T.T.~Trent, A
commutant lifting theorem on the polydisc, {\em Indiana University
Math.~J.} \textbf{48} (1999) no.~2, 653--675.


\bibitem{BSV} J.A.~Ball, C.~Sadosky and V.~Vinnikov, Scattering
systems with several evolutions and multidimensional
input/state/output systems, {\em Integral Equations and Operator
Theory} \textbf{52} (2005), 323-393.

\bibitem{BT-AIP} J.A.~Ball and T.T.~Trent, The abstract interpolation
problem and commutant lifting:  A coordinate-free approach, in {\em
Operator Theory and Interpolation: International Workshop on Operator
Theory and Applications, IWOTA96} (Ed.~H.~Bercovici and C.~Foia\c{s}),
pp.~51-83,  \textbf{OT115} Birkh\"auser-Verlag, Basel-Boston, 2000.

\bibitem{BTV} J.A.~Ball, T.T.~Trent and V.~Vinnikov, Interpolation
and commutant lifting for multipliers on reproducing kernel Hilbert
spaces, {\em Operator Theory and Analysis} (ed.~H.~Bart, I.~Gohberg
and A.C.M.~Ran), pp. 83--138, \textbf{OT 122} Birkh\"auser-Verlag,
Basel-Berlin-Boston, 2001.

\bibitem{FRKHS}  J.A.~Ball and V.~Vinnikov, Formal reproducing kernel
Hilbert spaces:  the commutative and noncommutative settings,
in {\em Reproducing Kernel Spaces and Applications} (Ed. D. Alpay) pp.
77-134, \textbf{OT 143} Birkh\"auser-Verlag (Basel-Boston-Berlin),
2003.


\bibitem{BFF} A.~Biswas, C.~Foias and A.E.~Frazho, Weighted commutant
lifting, {\em Acta Sci.~Math.~(Szeged)} \textbf{65} (1999), 657--686.

\bibitem{Kheifets-Dubov} S.S.~Boiko, V.K.~Dubovoy and A.Ya.~Kheifets,
Measure Schur complements and spectral functions of unitary operators
with respect to different scales, in {\em Operator Theory, System
Theory and Related Topics (Beer-Sheva/Rehovot 1997)}, pp.~89-138,
\textbf{OT 123}, Birkh\"auser-Verlag, Basel-Boston, 2001.



\bibitem{CS}  M.~Cotlar and C.~Sadosky, Integral representations of
bounded Hankel forms defined in scattering systems with a
multiparametric evolution group, in {\em Contributions to Operator
Theory and its Applications} (Ed. I.~Gohberg, J.W.~Helton and
L.~Rodman), \textbf{OT 35} Birkh\"auser-Verlag, Basel-Boston, 1988.

\bibitem{Davidson}  K.~Davidson, {\em Nest Algebras: Triangular Forms
for Operator Algebras on Hilbert Space}, Harlow, Essex, England:
Longman Scientific \& Technical, New York, Wiley, 1988.


\bibitem{FFbook}  C.~Foias and A.E.~Frazho, {\em The Commutant
Lifting Approach to Interpolation Problems}, \textbf{OT 44}
Birkh\"auser-Verlag, Basel-Boston, 1990.

\bibitem{FFGK}  C.~Foias, A.E.~Frazho, I.~Gohberg and
M.A.~Kaashoek, {\em Metric Constrained Interpolation, Commutant
Lifting and Systems}, \textbf{OT 100} Birkh\"auser-Verlag,
Basel-Boston, 1998.

\bibitem{FFK} C.~Foias, A.E.~Frazho and M.A.~Kaashoek, Relaxation of
metric constrained interpolation and a new lifting theorem, {\em
Integral Equations and Operator Theory} \textbf{42} (2002), 253--310.

\bibitem{Frazho84}  A.E.~Frazho, Complements to models for
noncommuting operators, {\em J.~Functional Analysis} \textbf{59}
(1984) no.~3, 445--461.

\bibitem{FtHK1} A.E.~Frazho, S.~ter Horst and M.A.~Kaashoek, Coupling
and relaxed commutant lifting problem, {\em Integral Equations and
Operator Theory} \textbf{54} (2006), 33--67.


\bibitem{FtHK2} A.E.~Frazho, S.~ter Horst and M.A.~Kaashoek, All
solutions to the relaxed commutant lifting problem, {\em Acta
Sci.~Math.(Szeged)} \textbf{72} (2006), 299--318.

\bibitem{Grenander} U.~Grenander and G. Szeg\"o, {\em T\"oplitz
forms and their applications}, University of California
Press, Berkeley-Los Angeles, 1958, iv+ 245 pp.


\bibitem{Hoffman} K.~Hoffman, {\em Banach Spaces of Analytic
Functions}, Prentice-Hall, 1962; reprinted: Dover Publications, New
York, 1988.

\bibitem{tH1} S.~ter Horst, Relaxed commutant lifting and Nehari
interpolation, Ph.D.~Thesis, Vrije Universiteit, Amsterdam, 2007:
available online:  www.darenet.nl

\bibitem{tH2} S.~ter Horst, Relaxed commutant lifting and a relaxed
Nehari problem:  Redheffer state space formulas, {\em Mathematische
Nachrichten}, to appear.


\bibitem{tH-pre} S.~ter Horst, Redheffer representations and relaxed
commutant lifting theory, preprint.

\bibitem{Kats1} V.~Katsnelson, Left and right Blaschke-Potapov
products and Arov-singular matrix valued functions, {\em Integral
Equations and Operator theory}, \textbf{13} (1990), 236--248.

\bibitem{Kats2} V.~Katsnelson, Weighted spaces of pseudocontinuable
functions and approximation by rational functions with prescribed
poles, {\em Z. Anal. Anwendungen}, \textbf{12} (1993), 27--47.

\bibitem{KKY} V.~Katsnelson, A.~Kheifets and P.~Yuditskii,
An abstract interpolation problem and the theory of extensions of
isometric operators. (Russian) In {\em Operators in function
spaces and problems in function theory}, \textbf{146} (1987),
83--96, "Naukova Dumka", Kiev; English translation in {\em Topics
in interpolation theory}, \textbf{OT 95} (1997), 283--298,
Birkh\"auser, Basel.

\bibitem{Kh-AIP1} A.~Kheifets, The Parseval equality in an abstract
problem of interpolation, and the union of open systems I, {\em
Teor. Funktsii Funktsional. Anal. i Prilozhen.}, (Russian)
\textbf{49} (1988), 112--120; English translation: {\em J. Soviet
Math.}
\textbf{49} (1990), no. 4, 1114--1120.

\bibitem{Kh-AIP2} A.~Kheifets, The Parseval equality in an abstract
problem of interpolation, and the union of open systems II, {\em
Teor. Funktsii Funktsional. Anal. i Prilozhen.}, (Russian)
\textbf{50} (1988), 98--103; English translation: {\em J. Soviet
Math.}
\textbf{49} (1990), no. 6, 1307--1310.


\bibitem{Kh-PhD} A.~Kheifets, Scattering matrices and Parseval
equality in Abstract Interpolation Problem, PhD Thesis, 1990,
Kharkov State University.


\bibitem{Kh-AAK}  A.~Kheifets,
The Nevanlinna-Adamyan-Arov-Kre\u{\i}n theorem in the semidefinite
case. (Russian) {\em Teor. Funktsii Funktsional. Anal. i
Prilozhen.} \textbf{56} (1991), 128--137; English translation:
{\em J. Math. Sci.} \textbf{76}, no. 4 (1995), 2542--2549.



\bibitem{Kh-IEOT95} A.~Kheifets, On a necessary but not sufficient
condition for a $\gamma$-generating pair to be a Nehari pair, {\em
Integral Equations and Operator Theory} \textbf{21} (1995), 334--341.

\bibitem{Kh-regulariz} A.~Kheifets, On regularization of
$\gamma$-generating
pairs, {\em J. Funct. Anal.} \textbf{130} (1995), no. 2, 310--333.

\bibitem{Kh-JFA96} A.~Kheifets, Hamburger moment problem:  Parseval
equality and $A$-singularity, {\em J. Funct. Anal.} \textbf{141}
(1996), no. 2, 374--420.

\bibitem{Kh-exposed} A.~Kheifets Nehari's interpolation problem and
exposed
points of the unit ball in the Hardy space $H\sp 1$, in {\em
Proceedings of the Ashkelon Workshop on Complex Function Theory}
(1996), 145-151, Israel Math. Conf. Proc. \textbf{11}, Bar-Ilan Univ.,
Ramat Gan, 1997.

\bibitem{Kheifets-Berkeley} A.~Kheifets, The abstract interpolation
problem and applications, in {\em Holomorphic Spaces (Berkeley,
CA, 1995)}, 351-379, Math. Sci. Res. Inst. Publ., \textbf{33},
Cambridge Univ. Press, Cambridge, 1998.

\bibitem{Kheifets-IWOTA96}  A.~Kheifets, Parametrization of solutions
of the Nehari problem and nonorthogonal dynamics, in {\em Operator
Theory and Interpolation} (Ed. H.~Bercovici and C.~Foia\c{s}), pp.
213-233, \textbf{OT 115} Birkh\"auser-Verlag, Basel-Boston, 2000.

\bibitem{HarmAIP}  A.~Kheifets, Abstract interpolation scheme
for harmonic functions, in {\em Interpolation Theory, Systems Theory
and Related Topics} (Ed. D.~Alpay, I.~Gohberg and V.~Vinnikov), pp.
287-317, \textbf{OT 134} Birkh\"auser-Verlag, Basel-Boston, 2002.

\bibitem{KhYu}  A.~Kheifets and  P.~Yuditskii, An analysis and
extension
of V. P. Potapov's approach to interpolation problems with
applications to the generalized bi-tangential
Schur-Nevanlinna-Pick problem and $J$-inner-outer factorization,
in {\em Matrix and operator valued functions}, pp. 133--161,
\textbf{OT 72} Birkh\"auser-Verlag, Basel-Boston, 1994.

\bibitem{Kupin} S.~Kupin, Lifting theorem as a special case of
abstract interpolation problem, {\em Zeitschrift f\"ur Analysis und
ihre Anwendungen} \textbf{15} (1996), 789-798.


\bibitem{Lax}  P.D.~Lax, {\em Functional Analysis},
Wiley-Interscience, 2002.

\bibitem{LiTimotin}  W.S.~Li and D.~Timotin, The relaxed intertwining
lifting in the coupling approach, {\em Integral Equations and Operator
Theory} \textbf{54} (2006), 97--111.

\bibitem{McCS} S.A.~McCullough and S.~Sultanic, Ersatz commutant
lifting with test functions, {\em Complex Analysis and Operator
Theory} \textbf{1} (2007) no.~4, 581--620.

\bibitem{Moran} M.D.~Moran, On intertwining dilations, {\em J.
Math.~Anal.~Appl.} \textbf{141} (1989), 219-234.

\bibitem{MS}  P.S.~Muhly and B.~Solel, Tensor algebras over
$C^{*}$-correspondences: representations, dilations and
$C^{*}$-envelopes, {\em J.~Functional Analysis} \textbf{158} (1998),
389--457.

\bibitem{MS04} P.S.~Muhly and B.~Solel, Hardy algebras,
$W^{*}$-correspondences and interpolation theory, {\em Math.~Annalen}
\textbf{330} (2004), 353--415.

\bibitem{NF-CLT} B.~Sz.-Nagy and C.~Foias, Dilatation des
commutants d'op\'erateurs, {\em C.R.~Acad.~Sci.~Paris, Serie A},
\textbf{266} (1968), 493--495.

\bibitem{NF}  B.~Sz.-Nagy and C.~Foias,  {\em Harmonic Analysis of
Operators on Hilbert Space}, North Holland/American Elsevier,
Amsterdam-New York, 1970.

\bibitem{Nev19} R.~Nevanlinna,  \"Uber beschr\"ankte
Funktionen, die in gegebene Punkten vorgeschriebene Werte annehmen,
{\em Ann.~Acad.~Sci.~Fenn.~Ser. A}  \textbf{13} (1919), no.~1.

\bibitem{Nev29} R. Nevanlinna,  \"Uber beschr\"ankte
Funktionen,
{\em Ann.~Acad.~Sci.~Fenn.~Ser.~A} \textbf{32}  (1929), no.~7.

\bibitem{Peller} V.~Peller, {\em Hankel Operators and Their
Applications}, Springer Monographs in Mathematics, Springer, New
York-Berlin, 2003.

\bibitem{Popescu-CLT1} G.~Popescu, Isometric dilations for infinite
sequences of noncommuting operators, {\it Trans.~Amer.~Math.~Soc.}
\textbf{316} (1989), 523--536.

\bibitem{Popescu98} G.~Popescu, Interpolation problems in several
variables, {\em J.~Math.~Anal.~Appl.} \textbf{227} (1998), 227--250.

\bibitem{Popescu03} G.~Popescu, Multivariable Nehari problem and
interpolation, {\em J.~Funct.~Anal.} \textbf{200} (2003),
536--581.

\bibitem{Popescu06} G.~Popescu, {\em Entropy and Multivariable
Interpolation}, Memoirs of the American Mathematical Society Number
{\bf 868}, American Mathematical Society, Providence, 2006.


\bibitem{RFM}  J.A.~Reneke, R.E.~Fennell and R.B.~Minton, {\em
Structured Hereditary Systems}, Marcel Dekker, New York-Basel, 1987.


\bibitem{Sarason}  D.~Sarason, Generalized interpolation in
$H^{\infty}$, {\em Trans.~American Math.~Soc.} \textbf{127} (1967),
179--203.

\bibitem{Sarason-Halmos}  D.~Sarason, New Hilbert spaces from old,
in {\em Paul Halmos:  Celebrating 50 Years of Mathematics} (ed.
J.H.~Ewing and F.W.~Gehring), pp. 195-204, Springer-Verlag, 1985.

\bibitem{Sarason2}  D.~Sarason, Exposed points in
$H^{1}$, I, in
{\em The Gohberg Anniversary Collection:  Volume II: Topics in
Analysis and Operator Theory} (Ed. H.~Dym, S.~Goldberg,
M.A.~Kaashoek, and P.~Lancaster), pp.~485--496, \textbf{OT41},
Birkh\"auser-Verlag, Basel-Boston-Berlin, 1989.


\bibitem{Sarason3}  D.~Sarason, Exposed points in
$H^{1}$, II, in {\em Topics in Operator Theory: Ernst. D.~Hellinger
Memorial Volume} (Ed.~L.~de Branges, I.~Gohberg, and J.~Rovnyak),
pp.~333--347, \textbf{OT  48},
Birkh\"auser-Verlag, Basel-Boston-Berlin, 1990.

\bibitem{Schmulyan1} Yu.~L.~Schmulian, Operator Hellinger integral,
{\em Mat. Sbornik}, \textbf{49(91)} (1959), 381-430 (Russian).

\bibitem{Schmulyan2} Yu.~L.~Schmulian, On reduction of operator
Hellinger integral to Lebesgue integral, {\em Izv. Vussh. Uchebn.
Zaved. (Matem.)}, \textbf{2(33)} (1963), 164-175 (Russian).

\bibitem{Schmulyan3} Yu.~L.~Schmulian, Topics on theory of
operators with finite nonhermitian rank, {\em Mat. Sbornik},
\textbf{57(99)}, (1962), 105-136 (Russian).


\bibitem{Sultanic}  S.~Sultanic, Commutant lifting theorem for the
Bergman space, {\em Integral Equations and Operator Theory}
\textbf{35} (2001) no.~1, 639--649.

\bibitem{TV} S.~Treil and A.~Volberg, A fixed point approach to
Nehari's problem and its applications, in {\em Toeplitz Operators and
Related Topics:  The Harold Widom Anniversary volume} (Ed. E.~Basor
and I.~Gohberg), pp. 165--186, \textbf{OT 71}
Birkh\"auser-Verlag, Basel-Berlin-Boston, 1994.


\end{thebibliography}
\end{document}